\documentclass[journal]{IEEEtran}

\usepackage[T1]{fontenc}
\usepackage[utf8]{inputenc}
\usepackage{amsmath,amssymb,amsfonts,amsthm}
\usepackage{mathtools}
\usepackage{bm}
\usepackage{graphicx}
\usepackage[caption=false,font=footnotesize]{subfig}
\usepackage{booktabs}
\usepackage{multirow}
\usepackage{xcolor}
\usepackage[normalem]{ulem}
\usepackage{algorithm}
\usepackage[noend]{algpseudocode}
\usepackage{url}
\usepackage{array}
\usepackage{cite}
\usepackage[hidelinks]{hyperref}
\usepackage[capitalize]{cleveref}

\graphicspath{{./}}

\newcommand{\E}{\mathbb{E}}

\newtheorem{theorem}{Theorem}
\newtheorem{lemma}{Lemma}
\newtheorem{proposition}{Proposition}
\newtheorem{corollary}{Corollary}
\newtheorem{assumption}{Assumption}
\theoremstyle{remark}
\newtheorem{remark}{Remark}
\newtheorem*{notation*}{Notation}
\crefname{assumption}{Assumption}{Assumptions}
\Crefname{assumption}{Assumption}{Assumptions}

\begin{document}

\title{Bilevel Learning via Strategically Inexact Stochastic Gradient Descent}

\author{\thanks{This work has been submitted to the IEEE for possible publication. Copyright may be transferred without notice, after which this version may no longer be accessible.}Mohammad~Sadegh~Salehi,
        Subhadip~Mukherjee,~\IEEEmembership{Member,~IEEE,}
        Lindon~Roberts,
        and~Matthias~J.~Ehrhardt%
\thanks{M.~S.~Salehi is an independent researcher, London, UK (e-mail: sadeghsalehi1997@gmail.com).}%
\thanks{S.~Mukherjee is with the Department of Electronics \& Electrical Communication Engineering, Indian Institute of Technology (IIT), Kharagpur, India (e-mail: smukherjee@ece.iitkgp.ac.in).}%
\thanks{L.~Roberts is with the School of Mathematics and Statistics, University of Melbourne, Parkville VIC 3010, Australia (e-mail: lindon.\allowbreak roberts@\allowbreak unimelb.\allowbreak edu.\allowbreak au).}
\thanks{M.~J.~Ehrhardt is with the Department of Mathematical Sciences, University of Bath, Bath, BA2 7AY, UK (e-mail: m.ehrhardt@bath.ac.uk).}
\thanks{LR is supported by the Australian Research Council Discovery Early Career Award DE240100006. MJE acknowledges support from the EPSRC (EP/T026693/1, EP/V026259/1, EP/Y037286/1) and the European Union Horizon 2020 research and innovation programme under the Marie Sk\l{}odowska-Curie grant agreement REMODEL. MSS acknowledges support from the EPSRC Centre for Doctoral Training in Statistical Applied Mathematics at Bath (SAMBa), under the project EP/S022945/1. SM acknowledges support from the ANRF under the Prime Minister Early Career Research Grant program (ANRF/ECRG/2024/001178/ENS).}}

\markboth{IEEE Transactions on Computational Imaging}%
{Salehi \MakeLowercase{\textit{et al.}}: Bilevel Learning via Strategically Inexact Stochastic Gradient Descent}

\maketitle

\begin{abstract}
Bilevel optimization is a central tool in computational imaging for high-dimensional hyperparameter tuning; prominent applications include learning data-adaptive regularizers and optimizing forward operators in variational regularization.  However, due to the nested nature of the problem, exact gradients with respect to the parameters (so-called hypergradients) are not available: the precision of the hypergradient is linearly related to the computational cost, thus requiring algorithms which can solve the problem without unnecessary precision.
 In learned-regularizer pipelines, the inner reconstruction solver is typically run for a fixed number of iterations, leaving the hypergradient accuracy uncontrolled and existing convergence theory inapplicable.
In this work, we analyze the schedules that practitioners actually run: we prove convergence of inexact stochastic bilevel optimization and establish rates under both decaying accuracy and step size schedules, showing that, with optimal configurations, convergence occurs at an $\mathcal{O}(k^{-1/4})$ rate in expectation.  This is the optimal rate of nonconvex stochastic gradient descent and matches single-loop stochastic bilevel methods; with warm-started inner solves, total costs also agree up to constants.
Experiments on image denoising and inpainting with convex ridge regularizers and input-convex neural networks confirm our analysis: decreasing step sizes improve stability, accuracy scheduling is more critical than step size strategy, and adaptive preconditioning (e.g., Adam) further boosts performance, particularly for input-convex neural networks.  In a budget-matched comparison with four state-of-the-art single-loop methods, accuracy-scheduled stochastic gradient descent and its Adam variant attain the highest test peak signal-to-noise ratio. 

\end{abstract}


\section{Introduction}

Bilevel optimization has become a versatile tool in machine learning and inverse problems \cite{Crockett_2022}, particularly in imaging applications where the quality of reconstruction depends critically on the choice of regularization parameters that encode the image prior. In variational regularization, an image is reconstructed from a measurement $y$ by solving
$\min_{x \in \mathbb{R}^n} \{h(x, \theta) \coloneqq D(Ax, y) + R_\theta(x)\}$,
where $D$ denotes a data fidelity term, $A$ is a forward operator, $y$ is the noisy measurement, and $R_\theta$ is a parametric regularizer. Bilevel learning chooses the parameter $\theta$ to minimize a supervised loss of this reconstruction (the upper-level problem), with the reconstruction problem itself as the lower-level problem, yielding a nested optimization structure that can adapt regularizers to specific datasets. This framework has proven effective for learning data-adaptive priors in imaging inverse problems \cite{InputConvex,goujon2022CRR,goujon2023WCRR,hertrich2025}. It can also be used to optimize physical acquisition models, such as sampling strategies in MRI reconstruction \cite{Sherry,Valkonen_single}, and to design forward models in seismic imaging \cite{Downing_2024}, by parameterizing the forward operator $A$ with a learnable parameter $\theta$.

In this work, we consider the stochastic bilevel optimization problem

\begin{subequations}\label{upper_stochastic}
\begin{align}
    \min_\theta \; & f(\theta) \coloneqq \E_{v}\big[f_v(\theta)\big], \label{Upper_stochastic}\\
    \text{s.t.} \; & \hat{x}_i(\theta) \coloneqq \arg\min_{x} h_i(x,\theta), \quad i = 1,\dots, m, \label{lower_stochastic}
\end{align}
\end{subequations}
where $f_v(\theta) \coloneqq \frac{1}{m}\sum_{i=1}^m v_i f_i(\theta)$ with $f_i(\theta) \coloneqq g_i(\hat{x}_i(\theta))$, and $v \in \mathbb{R}^m$ is a random sampling vector with $\E[v_i] = 1$ (made precise in \cref{sec:convergence}). 
The upper-level functions $g_i$ measure the quality of the lower-level solutions, e.g., $g_i(x) = \|x - x_i^*\|^2$ with $x_i^*$ denoting the desired solution.

Despite its theoretical appeal, large-scale bilevel optimization presents significant computational challenges. Computing the exact solution of the lower-level problems $\hat x_i(\theta)$ is typically intractable, and as a result, practical algorithms must rely on inexact solutions, which in turn yield approximate hypergradients that can significantly influence convergence behavior \cite{Pedregosa}. Moreover, the approximation of the hypergradient leads to a biased estimator, which necessitates a careful analysis of convergence behavior \cite{salehi2024inexactstochastic}.

\subsection{Related work}

Derivative-free bilevel methods \cite{DFO} avoid hypergradients but scale poorly with the number of parameters, whereas gradient-based methods \cite{maclaurin2015gradientbased,unrolledPaper,grazzi20a,MAID} scale well but predominantly address deterministic formulations \cite{Ochs,Valkonen_single,MAID} and do not exploit the speed-up and generalization benefits of stochastic optimization \cite{salehi2024inexactstochastic}. Hypergradients are computed by unrolling \cite{Ochs,unrolledPaper}, by Jacobian-free approximations \cite{JFB_2022}, or through the implicit function theorem with an approximate inverse Hessian \cite{ehrhardt2024efficientgradientbasedmethodsbilevel,ramzi2023shine}, and a priori and a posteriori bounds on the resulting error are available \cite{boundMatthias,AD,grazzi20a}. Algorithms built on these approximations, however, usually fix the number of lower-level iterations \cite{ramzi2023shine,Stochastic_khanduri} and assume a constant upper-level step size, leaving a gap between guarantees and implementations. Accuracy-adaptive \cite{MAID,bogensperger_adaptively_2025,FastADP} and single-step \cite{Valkonen_single} methods reduce the cost of deterministic bilevel learning. Inexact stochastic hypergradients were introduced in \cite{salehi2024inexactstochastic}, with convergence guarantees via biased stochastic gradient descent \cite{SGD_better,biasedSGD}, but without rates and without analyzing decaying step size and accuracy schedules.

Single-loop stochastic bilevel methods interleave a small, fixed amount of lower-level progress with each upper-level step. The two-timescale method TTSA \cite{hong2023TTSA}, ALSET \cite{chen2021ALSET} and its single-timescale variant \cite{chen2022singletimescale}, stocBiO \cite{ji2021stocbio}, and SUSTAIN \cite{Stochastic_khanduri} reach $\mathcal{O}(\epsilon^{-2})$ stochastic oracle calls for $\E\|\nabla f\|^2 \le \epsilon$, the optimal rate of single-level nonconvex stochastic optimization; SOBA and SABA \cite{dagreou2022SOBA} and MA-SOBA \cite{chen2024MASOBA} additionally track the solution of the Hessian linear system, and variance reduction improves the $\epsilon$-order in the finite-sum setting \cite{dagreou2022SOBA,chu2024SPABA}. These analyses count oracle calls under variance assumptions, and the lower-level error is neither monitored nor certified. We instead analyze the \emph{solve-to-tolerance} regime common in variational imaging: the lower-level problem is solved by any linearly convergent solver to an a posteriori verifiable accuracy $\epsilon_k$, and the schedules $(\epsilon_k, \alpha_k)$ are design parameters. In a common metric (\cref{rem:metric}), scheduled inexact stochastic gradient descent (ISGD) attains the same optimal rate as ALSET, SOBA, and MA-SOBA (\cref{tab:rates}), and with warm starts the same total complexity up to constants (\cref{rem:warmstart}).

\subsection{Contributions}
Our previous work \cite{salehi2024inexactstochastic} introduced inexact stochastic hypergradients with convergence guarantees via biased stochastic gradient descent. The convergence rates for decaying accuracy and step size schedules, the bilevel error bound of \cref{lem:bilevel_error}, the complexity results, and all experiments here, including the comparison with single-loop methods, are new. This work extends the theoretical foundation of inexact stochastic bilevel optimization by establishing explicit convergence rates under polynomial and logarithmic accuracy and step size decay schedules.  Our analysis reveals that optimal theoretical performance is achieved with accuracy decay $\epsilon_k = \mathcal{O}(k^{-1/4})$ and step size decay $\alpha_k = \mathcal{O}(k^{-1/2-\delta})$ with $0<\delta<0.5$ (optimally when $\delta \approx 0$), yielding convergence rates of $\mathcal{O}(k^{-1/4})$.  This matches the rate of stochastic gradient descent (SGD) for smooth nonconvex functions, and with a horizon-tuned constant step size the $\delta$-gap disappears (\cref{cor:fixed_horizon}). We further make the bilevel structure of the error model explicit: a new lemma (\cref{lem:bilevel_error}) instantiates the abstract hypergradient-error assumption with the lower-level strong convexity and smoothness constants. We also derive the resulting total computational complexity (\cref{cor:total_complexity}). 

Through comprehensive numerical experiments on image denoising and inpainting with multiple regularizer architectures, namely Convex Ridge Regularizers (CRR) \cite{goujon2022CRR} and Input-Convex Neural Networks (ICNN) \cite{InputConvex}, we uncover several key insights that bridge the gap between theory and practice. First, accuracy scheduling is more critical than step size strategy for achieving a desirable performance. Second, while decreasing step sizes are required for asymptotic convergence guarantees, fixed step sizes often yield superior results in practice under limited budgets, in line with our fixed-horizon analysis. Third, adaptive optimization methods such as Adam \cite{kingma2015adam} substantially outperform standard SGD \cite{robbins1951stochastic}, particularly for more complex regularizer architectures.  Fourth, a budget-matched comparison with single-loop stochastic bilevel methods (TTSA, SOBA, SABA, MA-SOBA) on the same learned-regularizer tasks relates the two algorithmic paradigms empirically under a common computational cost meter; with a loose initial tolerance and warm starts, the scheduled methods attain the highest test peak signal-to-noise ratio (PSNR).

\section{Proposed method}
\label{sec:alg}
Considering the upper-level problem \eqref{Upper_stochastic}, the gradient of the upper-level objective, called the \emph{hypergradient}, is given by

\begin{equation}\label{grad_upper_stochastic}
    \nabla f_v(\theta) = \frac{1}{m} \sum_{i=1}^{m} v_i \nabla f_i(\theta),
\end{equation}
with $\nabla f_i(\theta) = \partial \hat{x}_i(\theta)^T \nabla g_i(\hat{x}_i(\theta))$.
To ensure the hypergradient is well-defined (i.e., $\hat{x}_i(\theta)$ is continuously differentiable), we assume that each lower-level function $h_i(\cdot,\theta)$ is $\mu$-strongly convex and continuously differentiable, with its first and second derivatives with respect to $x$ continuous in $\theta$. In addition, each upper-level loss $g_i$ is assumed to be continuously differentiable. These conditions hold for the learned convex regularizers used in our experiments (\cref{sec:experiments}), where strong convexity of the lower-level objective follows from the strongly convex data fidelity (denoising) or an explicit quadratic penalty (inpainting).

Each iteration $k \in \mathbb{N}$ of SGD computes
$\theta^{k+1} = \theta^k - \alpha_k \nabla f_{v^k}(\theta^k)$,
where $v^k \sim \mathcal{D}$ are independent and identically distributed (i.i.d.) and $\alpha_k>0$ is the step size.

Since, in general, the lower-level solution $\hat{x}_i(\theta^k)$ at each upper-level iteration $k\in\mathbb{N}$ can only be approximated, one typically has access only to the approximate lower-level solution $\tilde{x}_i(\theta^k)$, where $\|\hat{x}_i(\theta^k) - \tilde{x}_i(\theta^k)\|\leq \epsilon$. Hence, corresponding to each element $v_i^k$ of the sampling vector $v^k$ for all $i\in \{1,2, \dots,m \}$, the stochastic hypergradient $v_i^k\partial \hat{x}_i(\theta^k)^T \nabla g_i(\hat{x}_i(\theta^k))$ can be approximated using methods described in \cite{boundMatthias, AD} with an error of $\mathcal{O}(\epsilon)$. This results in an inexact stochastic hypergradient in the upper-level problem, which we denote by
\begin{align}
    z^k \coloneqq z_{v^k}(\theta^k) = \nabla f_{v^k}(\theta^k) + e_{v^k}(\theta^k), \quad v^k \sim \mathcal{D}. \label{eq:inexacthypergradient}
\end{align}
Note that ensuring the boundedness of this error can be done utilizing the error bounds provided in \cite{boundMatthias,grazzi20a}; \cref{lem:bilevel_error} below makes this connection explicit in terms of the lower-level problem constants. Putting these ingredients together results in the proposed algorithm Inexact Stochastic Gradient Descent (ISGD), see \cref{alg:isgd}.

\begin{algorithm}
\caption{Inexact Stochastic Gradient Descent (ISGD) with varying step size and accuracy}
\label{alg:isgd}
\begin{algorithmic}[1]
\State Input: Initial parameters $\theta^0 \in \mathbb{R}^d$, sampling $\mathcal{D}$, step sizes $\{\alpha_k\}_{k=0}^\infty$, accuracies $\{\epsilon_k\}_{k=0}^\infty$.
\For{$k = 0,1, \dots$}
    \State{Sample a mini-batch $v^k \sim \mathcal{D}$.}
    \State{Compute the approximate stochastic gradient $z^k$ with error $\mathcal{O}(\epsilon_k)$.}
    \State{Update the iterate: \begin{align}
        \theta^{k+1} = \theta^k - \alpha_k z^k \label{ISGD_update}
    \end{align}}
\EndFor
\end{algorithmic}
\end{algorithm}

\section{Convergence analysis}
\label{sec:convergence}
\emph{Sampling model.} The sampling vector $v \in \mathbb{R}^m$ in \eqref{Upper_stochastic} has independent entries with $\E[v_i] = 1$, e.g., $v_i \sim \mathrm{Binomial}(m, 1/m)$; this covers the common practical schemes, such as uniform mini-batch or single-sample selection. All results only use $\E[v_i]=1$ and $\E[v_i^2]<\infty$, and therefore hold for any of these unbiased sampling schemes. With exact lower-level solutions $\hat{x}_i(\theta)$, $f_v$ and $\nabla f_v$ are unbiased estimators of $f$ and $\nabla f$, i.e., $\E_{v \sim \mathcal{D}} [f_v(\theta)] = f(\theta)$ and $\E_{v \sim \mathcal{D}} [\nabla f_v(\theta)] = \nabla f(\theta)$. 

To prove the convergence of \cref{alg:isgd}, we make the following regularity assumptions; all proofs are given in the supplementary material.
\begin{notation*}
For any $v \sim \mathcal{D}$ and for all \(\theta \in \mathbb{R}^d\),  the error of the hypergradient is denoted by $\|e_v(\theta)\|$ and the expected squared error by $E^2(\theta) \coloneqq \E_{v\sim \mathcal{D}} [\|e_v(\theta)\|^2]$. As a measure of stationarity after $K$ iterations we use $\mathcal{E}_K \coloneqq \min_{0 \le k \le K-1} \E\big[\|\nabla f(\theta^k)\|^2\big]$; by Jensen's inequality, $\min_{0 \le k \le K-1} \E[\|\nabla f(\theta^k)\|] \le \mathcal{E}_K^{1/2}$.
\end{notation*}

\begin{assumption}\label{assumption1} Each $f_i$ is $L_{\nabla f_i}$-smooth, which means $f$ and each $f_i$ are continuously differentiable with $L_{\nabla f}$ and $L_{\nabla f_i}$ Lipschitz gradients, respectively. Moreover, each $f_i$ is bounded below by $f_i^{\text{inf}}$.
\end{assumption}
Since, in practice, in each iteration $k\in \mathbb{N}$ of \cref{alg:isgd}, the error $\|e_{v^k}(\theta^k)\|$ is controlled by a deterministic a posteriori upper bound, we consider the following assumption.

\begin{assumption}\label{assumption2}
    At each iteration $k \in \mathbb{N}$ of \cref{alg:isgd}, for the error of the hypergradient we have $E^2(\theta^k) \leq \epsilon_k^2$, where $\{ \epsilon_k\}_{k= 0}^\infty$ is a deterministic non-negative sequence.
\end{assumption}

\subsection{From lower-level accuracy to hypergradient error}
\label{subsec:bilevel_error}
\Cref{assumption2} is stated abstractly, but in the bilevel setting it is not merely an oracle hypothesis: the error sequence $\{\epsilon_k\}$ is \emph{implementable} by controlling two quantities that the practitioner monitors anyway, namely the accuracy of the lower-level solve and the residual of the Hessian linear system.  Related a priori and a posteriori bounds appear in \cite{boundMatthias,grazzi20a,salehi2024inexactstochastic}; we state the version adapted to our sampling model.

\begin{lemma}[Bilevel instantiation of \cref{assumption2}]\label{lem:bilevel_error}
Fix $\theta \in \mathbb{R}^d$ and let, for each $i$, the lower-level objective $h_i(\cdot, \theta)$ be $\mu$-strongly convex and twice continuously differentiable with $L_h$-Lipschitz gradient, $L_H$-Lipschitz Hessian $\nabla^2_{xx} h_i$, and cross-derivative satisfying $\|\nabla^2_{\theta x} h_i(x,\theta)\| \le B_J$ with Lipschitz constant $L_J$ in $x$. Let each $g_i$ have $L_g$-Lipschitz gradient with $\|\nabla g_i(x)\| \le G$ on a neighborhood of $\hat{x}_i(\theta)$ containing the approximate solutions, with $\delta$ small enough that the segment $[\tilde{x}_i, \hat{x}_i(\theta)]$ remains in this neighborhood. Suppose the approximate hypergradient for sample $i$ is computed from an approximate lower-level solution $\tilde{x}_i$ with $\|\tilde{x}_i - \hat{x}_i(\theta)\| \le \delta$, and an approximate solution $\tilde{w}_i$ of the linear system $\nabla^2_{xx} h_i(\tilde{x}_i, \theta)\, w = \nabla g_i(\tilde{x}_i)$ with residual at most $\tau$. Then the per-sample hypergradient error satisfies
\begin{equation}\label{eq:per_sample_error}
\|e_i(\theta)\| \le c_1 \delta + c_2 \tau,
\end{equation}
with
\begin{equation*}
c_1 = B_J\Big(\frac{L_g}{\mu} + \frac{G L_H}{\mu^2}\Big) + \frac{L_J G}{\mu},
\qquad
c_2 = \frac{B_J}{\mu}.
\end{equation*}
Moreover, for the sampling model of \eqref{Upper_stochastic} with $\E[v_i] = 1$ and $\E[v_i^2] \le \nu^2 < \infty$,
\begin{equation}\label{eq:expected_error}
\E_{v \sim \mathcal{D}}\big[\|e_v(\theta)\|^2\big] \le \nu^2\,\big(c_1 \delta + c_2 \tau\big)^2 .
\end{equation}
In particular, running the lower-level solver and the linear-system solver at iteration $k$ to tolerances $\delta_k$ and $\tau_k$ with $c_1 \delta_k + c_2 \tau_k \le \epsilon_k/\nu$ realizes \cref{assumption2}.
\end{lemma}

\begin{remark}
\Cref{lem:bilevel_error} shows where the bilevel structure enters the analysis: the constants $c_1, c_2$ degrade as the lower-level problem becomes less strongly convex ($\mu \downarrow 0$), and the two tolerances $(\delta_k, \tau_k)$ give the practitioner independent control over the two sources of hypergradient error. In our implementation, $\delta_k$ is enforced through the a posteriori bound $\|\tilde{x} - \hat{x}\| \le \|\nabla_x h(\tilde{x},\theta)\|/\mu$ and $\tau_k$ through the residual of the minimum residual method (MINRES) on the Hessian system, so \cref{assumption2} is verifiable at run time rather than hypothesized.
\end{remark}

The following assumption is crucial for establishing the convergence of \cref{alg:isgd} and for controlling the decay of the upper-level step size.
\begin{assumption}\label{assump:sequences}
The step size sequence $\{\alpha_k\}_{k=1}^\infty$ is non-increasing and positive, and satisfies:
\begin{enumerate}[(i)]
    \item $\sum_{k=1}^\infty \alpha_k^2 < \infty$ \quad (square summability condition)
    \item $\lim_{k\to \infty} \frac{1}{k\alpha_k} = 0$ \quad (step size decay condition)
\end{enumerate}
\end{assumption}
These step size conditions are closely related to the classical Robbins--Monro criteria \cite{ROBERTS_SGD,robbins1951stochastic} that ensure the convergence of stochastic approximation methods such as SGD.  Condition $(ii)$, $\lim_{k \to \infty} \frac{1}{k\alpha_k} = 0$, is a stronger requirement than the usual non-summability condition $\sum_{k = 1}^\infty \alpha_k = \infty$. It enforces that $\alpha_k$ decays slower than $1/k$.

\paragraph{Proof strategy}
The analysis follows the biased-SGD template of \cite{SGD_better,ROBERTS_SGD,salehi2024inexactstochastic}. Under \cref{assumption1} and $\E[v_i^2] < \infty$, the stochastic hypergradient satisfies the expected smoothness (ES), or ABC, condition \cite[Proposition~2]{SGD_better}: there exist $A, B, C \ge 0$ with
\begin{equation}\label{eq:ABC}
\E_{v \sim \mathcal{D}}\big[\|\nabla f_v(\theta)\|^2\big] \le 2A\big(f(\theta) - f^{\text{inf}}\big) + B\|\nabla f(\theta)\|^2 + C .
\end{equation}
Transferring \eqref{eq:ABC} to the inexact estimator $z_v$ at the price of the error $E^2(\theta)$ and using the $L_{\nabla f}$-smoothness of $f$ gives a weighted descent inequality. For fixed $\eta > 0$, $\zeta > \frac{1}{2}$, and $\rho \in (0,1)$, let $\tilde{B} \coloneqq \frac{1+\eta}{\eta} B$.

\begin{theorem}\label{convergence_thm}
    Let Assumptions \ref{assumption1}, \ref{assumption2}, \ref{assump:sequences} hold, and let $\alpha_0 \le \frac{\rho (2\zeta -1)}{\zeta L_{\nabla f}\tilde{B}}$.
    Then there exist constants $C_1,C_2>0$ such that
    \begin{equation} \label{eq_conv_thm}
        \mathcal{E}_K \leq \frac{C_1}{K\alpha_K} + C_2 L_K,
    \end{equation}
    where $L_K\coloneqq \frac{1}{K\alpha_K}\sum_{k=0}^{K-1} \alpha_k \epsilon_k^2$.
    Moreover, if $\lim_{K\to\infty} L_K=0$, then
    $\lim_{K\to \infty} \mathcal{E}_K = 0$.
\end{theorem}

\subsection{Examples of decay regimes with corresponding convergence rates}

Bounding $L_K$ in \cref{convergence_thm} for polynomial and logarithmic schedules (Propositions~2 and~3 in the supplementary material) yields the following rates.
\begin{corollary}\label{convergence_rate}
Let Assumptions \ref{assumption1}, \ref{assumption2}, and \ref{assump:sequences} hold and a suitable initial step size $\alpha_0>0$ be chosen as in \cref{convergence_thm}.

\begin{enumerate}
\item \textbf{Polynomial step size and polynomial error.}
If $\alpha_k = \mathcal{O}(k^{-q})$ for $\tfrac{1}{2}<q<1$ and $\epsilon_k = \mathcal{O}(k^{-p})$ for $p>0$, then
\begin{equation}\label{rate_poly_poly}
\mathcal{E}_K^{1/2} =
\begin{cases}
\mathcal{O}\!\big(K^{-\min\{2p,\,1-q\}/2}\big), & q+2p\neq1,\\
\mathcal{O}\!\Big(\sqrt{\dfrac{\log K}{K^{1-q}}}\Big), & q+2p=1.
\end{cases}
\end{equation}

\item \textbf{Polynomial step size and logarithmic error.}
If $\alpha_k = \mathcal{O}(k^{-q})$ for $\tfrac{1}{2}<q<1$ and $\epsilon_k = \mathcal{O}((\log k)^{-p})$ for $p>0$, then
\begin{equation}\label{rate_poly_log}
\mathcal{E}_K^{1/2}
= \mathcal{O}\!\big((\log K)^{-p}\big).
\end{equation}
\end{enumerate}
\end{corollary}

\begin{proof}
From \cref{convergence_thm} and \cref{assump:sequences}(ii) we have
$\mathcal{E}_K = \mathcal{O}(L_K)$.
Now utilizing Propositions~2--3, we conclude the required result.
\end{proof}
\paragraph{Interpretation and guideline for parameter choice}

If $2p < 1-q$, the hypergradient accuracy limits the convergence and the gradient-rate is $\mathcal{O}(k^{-p})$; if $2p > 1-q$, the step size decay limits it and the rate is $\mathcal{O}(k^{-(1-q)/2})$; on the boundary $2p=1-q$ a logarithmic factor appears (Table~III of the supplementary material summarizes all regimes). The best asymptotic rate achievable within these families is $\mathcal{O}(k^{-1/4})$, obtained by taking $p\ge 1/4$ and letting $q\downarrow 1/2$.

Note that even in cases where the rate does not depend on the accuracy schedule exponent $p$, its choice remains crucial in practice. In particular, we are interested in making $p$ as small as possible, since it directly affects the computational cost.

\subsection{Total computational complexity}
\label{subsec:total_complexity}
The rates above are stated per upper-level iteration, but each iteration hides an inner solve whose cost grows as the tolerance tightens. We now make the total cost explicit.

\begin{corollary}[Total complexity of scheduled ISGD]\label{cor:total_complexity}
Let the assumptions of \cref{convergence_rate} hold with the rate-optimal schedules $\epsilon_k = \mathcal{O}(k^{-1/4})$ and $\alpha_k = \mathcal{O}(k^{-(1/2+\delta)})$ for small $\delta > 0$, and let the tolerances $(\delta_k, \tau_k)$ realizing $\epsilon_k$ via \cref{lem:bilevel_error} be met by linearly convergent solvers: an accelerated (Nesterov-type) lower-level solver contracting at rate $1 - \Theta(1/\sqrt{\kappa})$ with $\kappa = L_h/\mu$, and MINRES (or CG) on the $\mu$-strongly-convex Hessian system. Then reaching $\mathcal{E}_K^{1/2} \le \epsilon$ requires
$K(\epsilon) = \mathcal{O}(\epsilon^{-4-\delta'})$
upper-level iterations (for $\delta' \to 0$ as $\delta \to 0$), and the total number of lower-level and Hessian-system iterations, in the worst case of cold-started inner solves, is
$\mathcal{O}(\sqrt{\kappa}\, \epsilon^{-4-\delta'} \log(1/\epsilon))$,
i.e., $\tilde{\mathcal{O}}(\epsilon^{-4})$ treating the condition number $\kappa$ as a constant. 
\end{corollary}

\begin{remark}[Metric conventions]\label{rem:metric}
Two conventions coexist for the target accuracy. Writing rates in the common currency of the squared gradient norm, standard SGD on a smooth nonconvex objective satisfies $\mathcal{E}_K$ $= \mathcal{O}(K^{-1/2})$, which is optimal for unbiased bounded-variance oracles \cite{arjevani2023lower}. This one rate is quoted as $\mathcal{O}(\epsilon^{-2})$ iterations when $\epsilon$ bounds the \emph{squared} norm ($\E\|\nabla f\|^2 \le \epsilon$) and as $\mathcal{O}(\epsilon^{-4})$ when $\epsilon$ bounds the \emph{norm} ($\E\|\nabla f\| \le \epsilon$). Single-loop bilevel analyses \cite{chen2021ALSET,chen2022singletimescale,dagreou2022SOBA,chen2024MASOBA} state $\mathcal{O}(\epsilon^{-2})$ in the squared convention; our \cref{convergence_rate} states $\mathcal{O}(K^{-1/4})$ for the norm. Converted to a common metric, our rate for the squared gradient norm is $\mathcal{O}(K^{-(1-q)}) \to \mathcal{O}(K^{-1/2})$ as $q \downarrow 1/2$: \emph{identical, in upper-level iterations, to ALSET, SOBA, and MA-SOBA, to standard single-level SGD, and to the lower bound of \cite{arjevani2023lower} for this oracle class}. \Cref{tab:rates} collects the rates and assumptions in unified form.
\end{remark}

\begin{table*}[tbhp]
\centering
\caption{Convergence rates and assumptions for stochastic bilevel methods (nonconvex upper level, $\mu$-strongly-convex lower level), unified to the squared-gradient metric. ``Upper iters.'' is the iteration count for $\E\|\nabla f\| \le \epsilon$ (norm convention); papers quoting $\mathcal{O}(\epsilon^{-2})$ use the squared convention $\E\|\nabla f\|^2 \le \epsilon$, which coincides with $\epsilon^{-4}$ here (\cref{rem:metric}). SGD row shown for reference.}
\label{tab:rates}
\footnotesize
\setlength{\tabcolsep}{5pt}
\begin{tabular}{@{}lccll@{}}
\toprule
\textbf{Method} & $\mathcal{E}_K$ & \textbf{Upper iters.} & \textbf{Stochastic model} & \textbf{Additional assumptions} \\
\midrule
SGD (single-level) \cite{Ghadimi_SGD} & $K^{-1/2}$ (optimal \cite{arjevani2023lower}) & $\epsilon^{-4}$ & unbiased, bounded variance & $L$-smoothness \\
\midrule
ISGD (this work, \cref{convergence_rate}) & $K^{-(1-q)} \!\to\! K^{-1/2}$ & $\epsilon^{-4-\delta'}$ & \emph{certified biased}, ES/ABC & Hessian-Lipschitz $h_i$ (\cref{lem:bilevel_error}) \\
ISGD, fixed horizon (\cref{cor:fixed_horizon}) & $K^{-1/2}$ & $\epsilon^{-4}$ & \emph{certified biased}, ES/ABC & as above \\
\midrule
TTSA \cite{hong2023TTSA} & $K^{-2/5}$ & $\epsilon^{-5}$ & unbiased, bounded var. & Lipschitz $\nabla,\nabla^2$; projection \\
ALSET \cite{chen2021ALSET} & $K^{-1/2}$ & $\epsilon^{-4}$ & unbiased, bounded var. & Lipschitz $\nabla,\nabla^2$ \\
SOBA \cite{dagreou2022SOBA} & $K^{-1/2}$ & $\epsilon^{-4}$ & unbiased, bounded var. & $\nabla^3$ Lipschitz; bounded $\nabla_1 F$ \\
MA-SOBA \cite{chen2024MASOBA} & $K^{-1/2}$ & $\epsilon^{-4}$ & unbiased, bounded var. & second-order smoothness \\
SABA \cite{dagreou2022SOBA} & $N^{2/3} K^{-1}$ & $N^{2/3}\epsilon^{-2}$ & finite sum (var.\ reduction) & per-sample Lipschitz $\nabla,\nabla^2$ \\
\bottomrule
\end{tabular}
\end{table*}

\begin{corollary}[Fixed-horizon constant step]\label{cor:fixed_horizon}
Fix a horizon $K$ and let Assumptions \ref{assumption1} and \ref{assumption2} hold with the constant step size $\alpha_k \equiv \alpha = c/\sqrt{K}$ (with $c$ such that the step size bound of \cref{convergence_thm} holds) and accuracies $\epsilon_k = \mathcal{O}(k^{-1/4})$. Then
\begin{equation*}
\mathcal{E}_K = \mathcal{O}\big(K^{-1/2}\big).
\end{equation*}
\end{corollary}

This removes the $\delta$-gap of the anytime schedules: with a horizon-tuned constant step---exactly the regime in which single-loop rates are stated---scheduled ISGD attains the same $\mathcal{O}(K^{-1/2})$ rate, matching the lower bound of \cite{arjevani2023lower} for its oracle class. No choice of $(\alpha_0, \epsilon_0, p, q)$ can improve on this exponent, since the upper-level recursion is an (inexact) SGD iteration and the bound is information-theoretic; over-solving the lower level ($p > 1/4$) increases cost without improving the rate.

\begin{remark}[Warm starts remove the logarithmic factor]\label{rem:warmstart}
The $\log$ factor in \cref{cor:total_complexity} stems from cold-starting the inner solver. Under warm starts the picture improves: if the hypergradient norms remain bounded along the trajectory, then $\|\hat{x}_i(\theta^{k+1}) - \hat{x}_i(\theta^k)\| = \mathcal{O}(\alpha_k/\mu)$ by the Lipschitz continuity of $\theta \mapsto \hat{x}_i(\theta)$, so the inner solver starts at distance $\mathcal{O}(\delta_k + \alpha_k)$ from the new solution while the required tolerance $\delta_{k+1} = \Theta((k+1)^{-1/4})$ shrinks only polynomially. The number of inner iterations per upper step is then $\mathcal{O}\big(\sqrt{\kappa}\,\log\tfrac{\delta_k + \alpha_k}{\delta_{k+1}}\big) = \mathcal{O}(\sqrt{\kappa})$ on average, and the total complexity of \cref{cor:total_complexity} improves to $\mathcal{O}(\sqrt{\kappa}\,\epsilon^{-4})$---matching the total oracle complexity of ALSET, SOBA, and MA-SOBA up to the condition-number constant, without any logarithmic overhead. 
\end{remark}

\section{Numerical experiments}
\label{sec:experiments}
Throughout our numerical experiments, based on \cref{convergence_rate}, we select accuracy decay exponents $p \in \{0,\allowbreak 0.25,\allowbreak 0.5,\allowbreak 1,\allowbreak 2\}$ and step size decay exponents $q \in \{0,\allowbreak 0.5,\allowbreak 1\}$. The theoretically optimal choice is $(p, q) = (0.25, 0.5)$, which achieves the best convergence rate of $\mathcal{O}(k^{-1/4})$. Note that $q = 0$ (fixed step size) only guarantees convergence to a neighborhood of the solution rather than exact convergence \cite{salehi2024inexactstochastic}, though we include it to assess practical performance under finite computational budgets.

Across all experiments, we use Nonmonotone Accelerated Proximal Gradient Descent (nmAPG) \cite{nmAPG} to solve the lower-level problems. Since the lower-level objectives are smooth, we replace the proximal updates with standard gradient steps. Moreover, we use the MINRES method to approximate the inverse Hessian of the lower-level objective through matrix--vector products computed via its corresponding linear system. Reconstruction quality is reported as PSNR and the structural similarity index measure (SSIM).

All of the code is available on the GitHub repository\footnote{\Urlmuskip=0mu plus 3mu \url{https://github.com/MohammadSadeghSalehi/Inexact_stochastic_bilevel_learning/}}. Moreover, parts of the code and the general structure for computing hypergradients are adapted from \cite{hertrich2025}\footnote{\url{https://github.com/johertrich/LearnedRegularizers}}.

\subsection{Denoising with Convex Ridge Regularizer}\label{subsec:denoise_CRR}

In this section, we perform an image denoising task on colored images from the BSDS500 dataset \cite{BSDS500}. Each image is augmented (by random and potentially overlapping crops) to $64 \times 64$ pixels. We consider 128 images in mini-batches of size 8 for training, and 16 images of size $192 \times 192$ (cropped to the center) as the test set. We then learn the parameters of a CRR \cite{goujon2022CRR} by solving \eqref{upper_stochastic} with $g_i(x) = \|x - x_i^*\|_2^2$ and $h_i(x,\theta) = \|x - y_i\|_2^2 + R_\theta(x)$,

where \( x_i^* \) denotes the ground truth image, and \( y_i \) is the corresponding noisy observation corrupted by additive Gaussian noise:
$y_i = x_i^* + \eta_i$, $\eta_i \sim \mathcal{N}(0, \sigma^2 I)$, $\sigma = 25/255.$

As a data-adaptive choice for the regularizer $ R_\theta(x) $ in imaging problems, the Field of Experts (FoE) architecture \cite{FoE} is implemented as a composition of convolutional layers followed by smooth, convex potential functions. Following the advancements of this architecture and the convention introduced in \cite{goujon2022CRR,goujon2023WCRR}, we consider the CRR defined as:
$$
R_\theta(x) = \sum_{i=1}^{C} \sum_{j=1}^{H} \sum_{k=1}^{W} \left[ \psi^\beta\left( \exp(s) \odot \mathcal{W}(x) \right) \right]_{i,j,k},
$$
where $ \mathcal{W} $ is a multi-layer convolution operator composed of multiple convolution layers, $ \psi^\beta \colon \mathbb{R} \to \mathbb{R}^+ $ is a smooth convex potential, \( s \) is a learnable scaling parameter, \( \odot \) denotes the element-wise product, and $C, H, W$ represent the channels ($1$ for grayscale and $3$ for colored images), height, and width of images, respectively.

The convolution operator \( \mathcal{W} \) is composed of three zero-padded convolutional layers with kernel sizes \( 5 \times 5 \), increasing the number of output channels as \( 4 \rightarrow 8 \rightarrow 64 \).  Initialization and normalization details are given in the supplementary material, Appendix~E.

The potential function $ \psi^\beta $ is shared across all filters and scaled via the learnable factor $ \alpha_j $ to obtain individual channel-wise potentials. We use smooth convex potentials (the Huber function, i.e., the Moreau envelope of the $\ell_1$-norm, or the log-cosh function), which ensure convexity of $R_\theta$; the gradient $\nabla R_\theta$ is available in closed form through the adjoint convolution operator $\mathcal{W}^\top$ \cite{goujon2022CRR}. This regularizer is compatible with our required regularity assumptions on the lower-level problem.

In this experiment, we vary the initial accuracies $\epsilon_0 \in \{10^0, 10^{-2}, 10^{-4}\}$, initial step sizes $\alpha_0 \in \{10^{-2}, 10^{-3}\}$, and the exponents for the accuracy ($k^{-p}$) and step size schedules ($k^{-q}$),  $p \in \{0, 0.25, 0.5, 1, 2\}$ and $q \in \{0, 0.5, 1\}$. We first compare the performance of ISGD for each initial accuracy separately, as shown in \cref{fig:1e0_1e-2_1e-4_denoise}. We display only configurations that achieved reasonable training performance, filtering out divergent or severely suboptimal settings. 

\begin{figure*}[tbhp]
\centering
\captionsetup[subfloat]{farskip=0pt,captionskip=0pt}
\newcommand{\dn}[1]{\includegraphics[width=0.305\textwidth]{Figs/Denoising/#1}}
{\footnotesize\makebox[0.315\textwidth]{$\epsilon_0 = 1$}\hfill\makebox[0.315\textwidth]{$\epsilon_0 = 10^{-2}$}\hfill\makebox[0.315\textwidth]{$\epsilon_0 = 10^{-4}$}}\\[1pt]
\subfloat[]{\label{fig:1e0_fixed}\dn{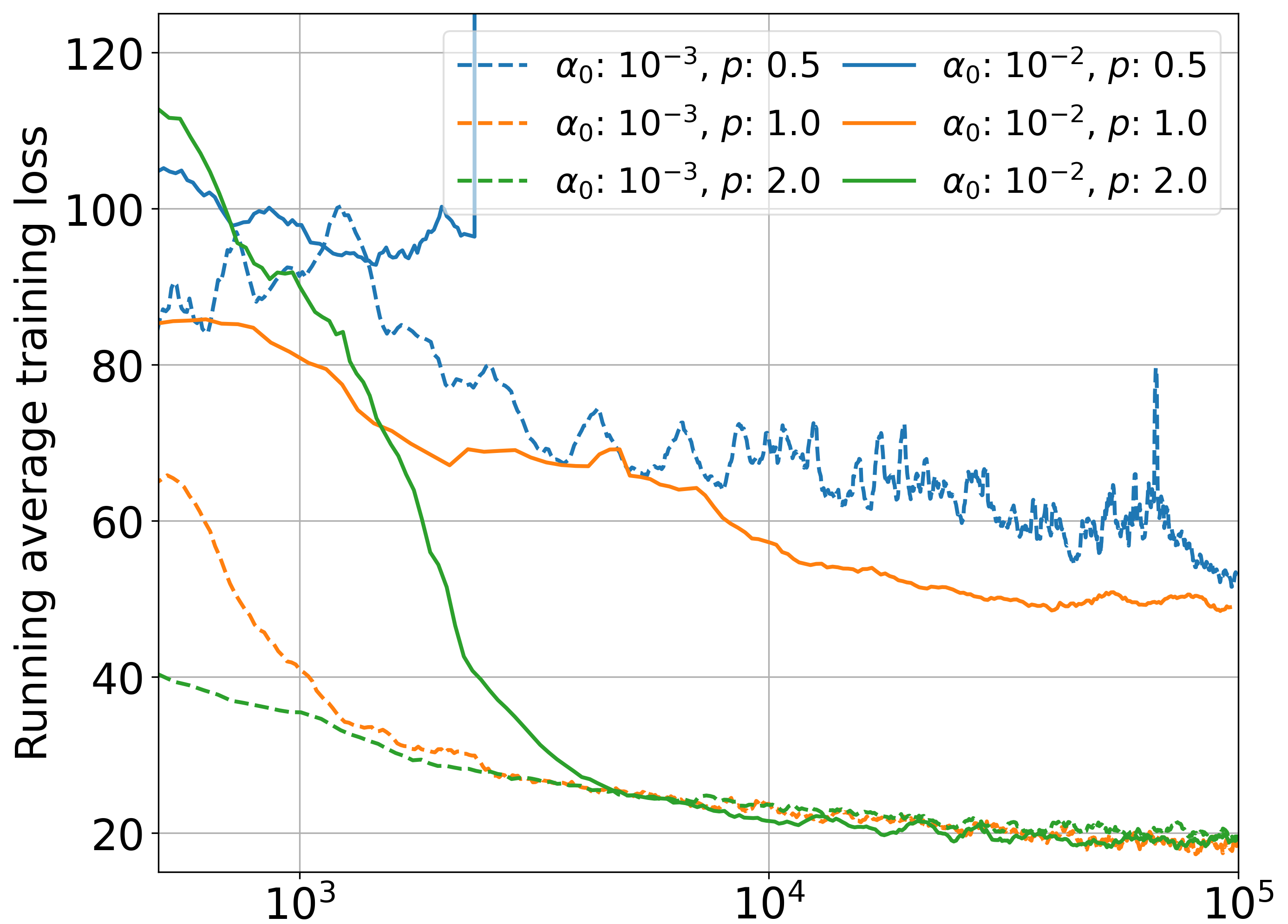}}\hfill
\subfloat[]{\label{fig:1e-2_fixed}\dn{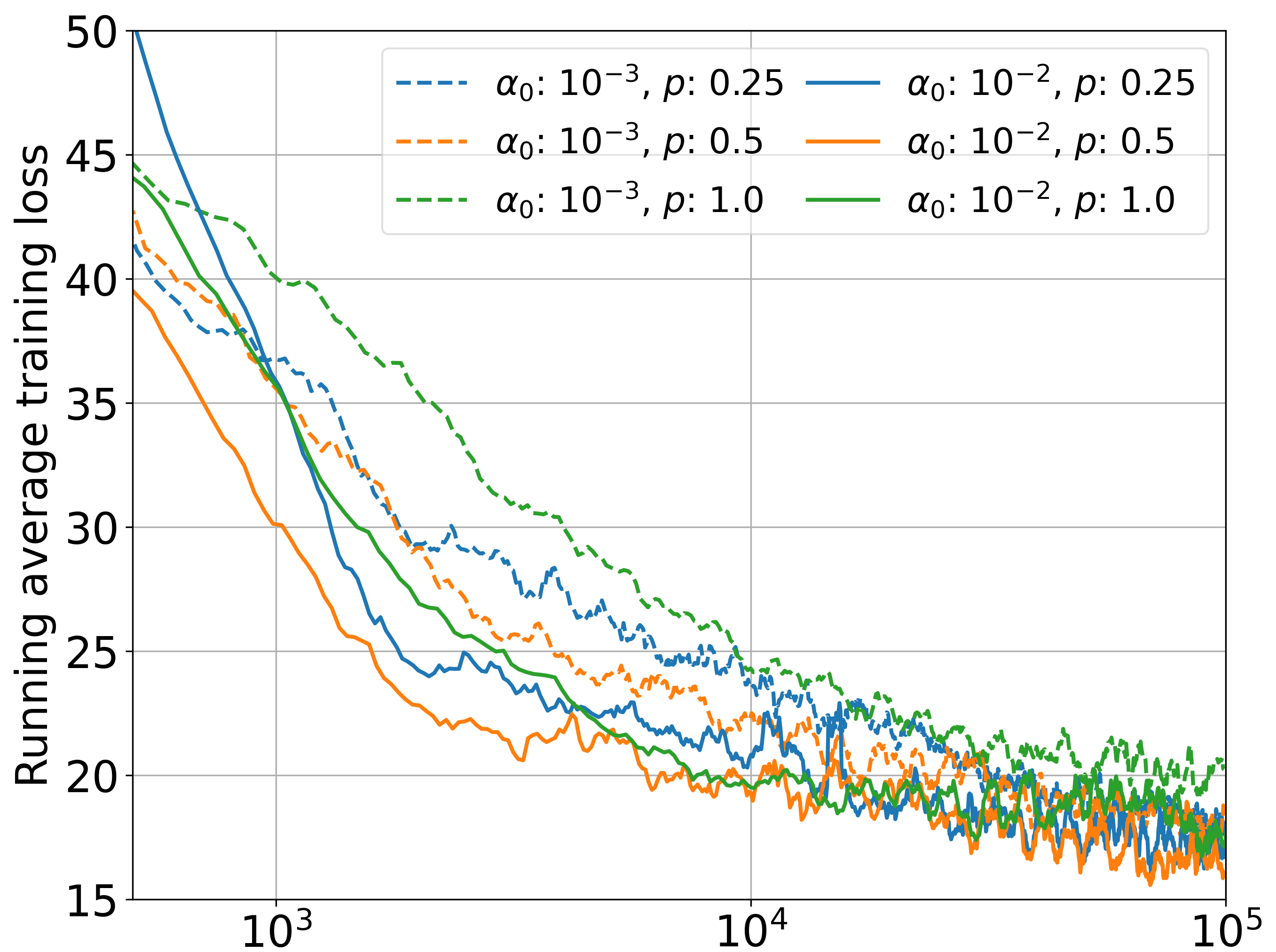}}\hfill
\subfloat[]{\label{fig:1e-4_fixed}\dn{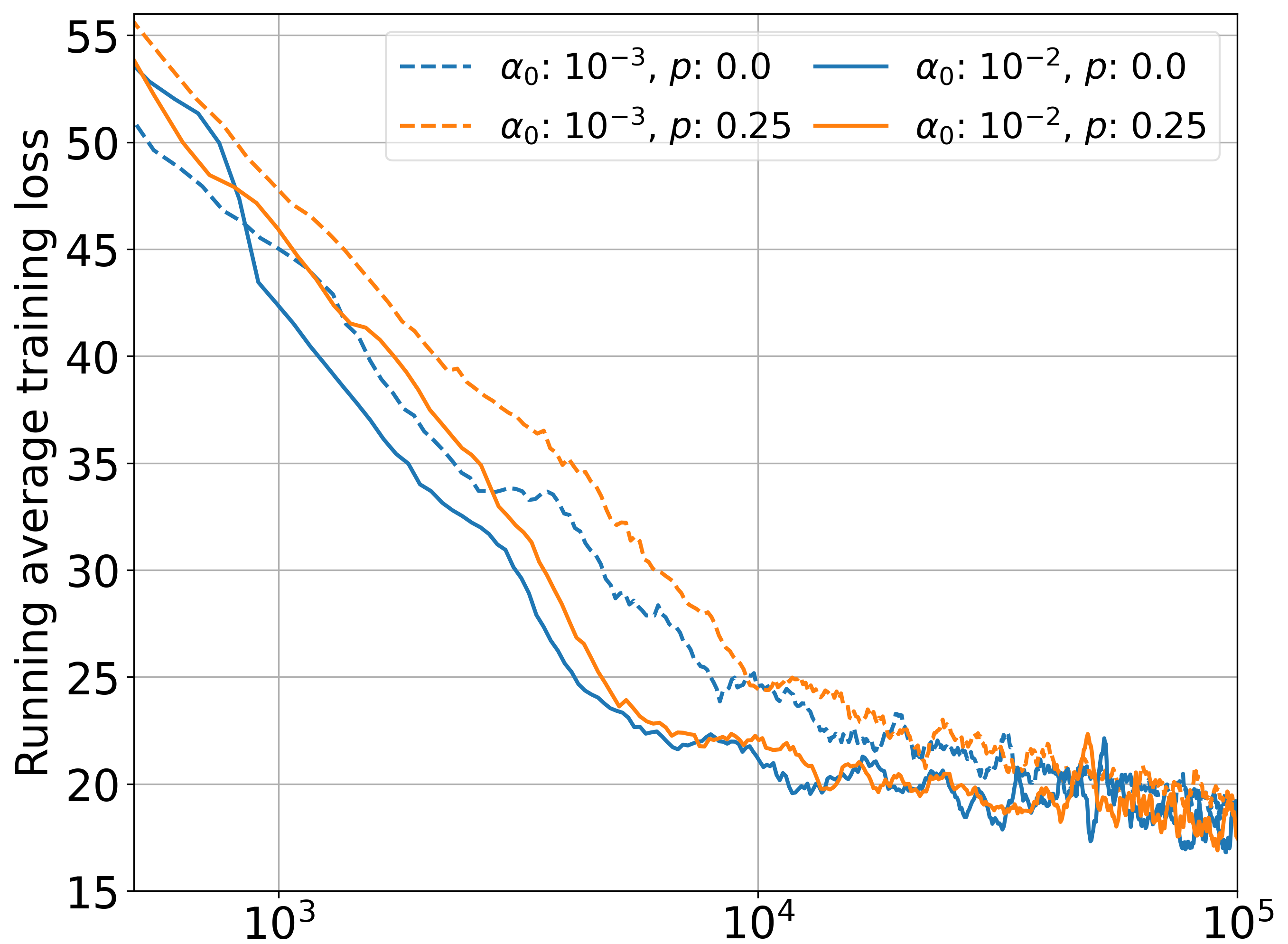}}\\
\subfloat[]{\label{fig:1e0_psnr}\dn{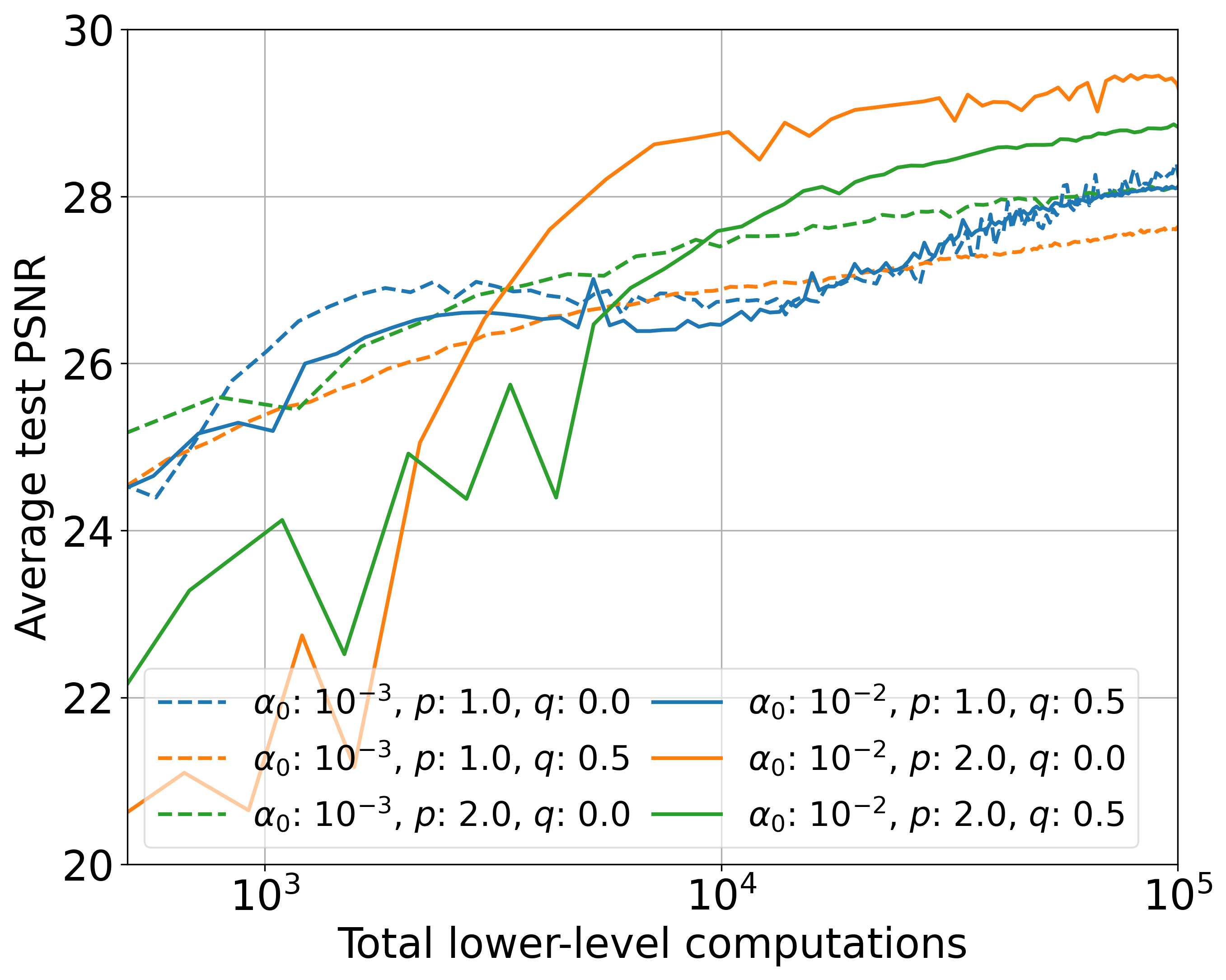}}\hfill
\subfloat[]{\label{fig:1e-2_psnr}\dn{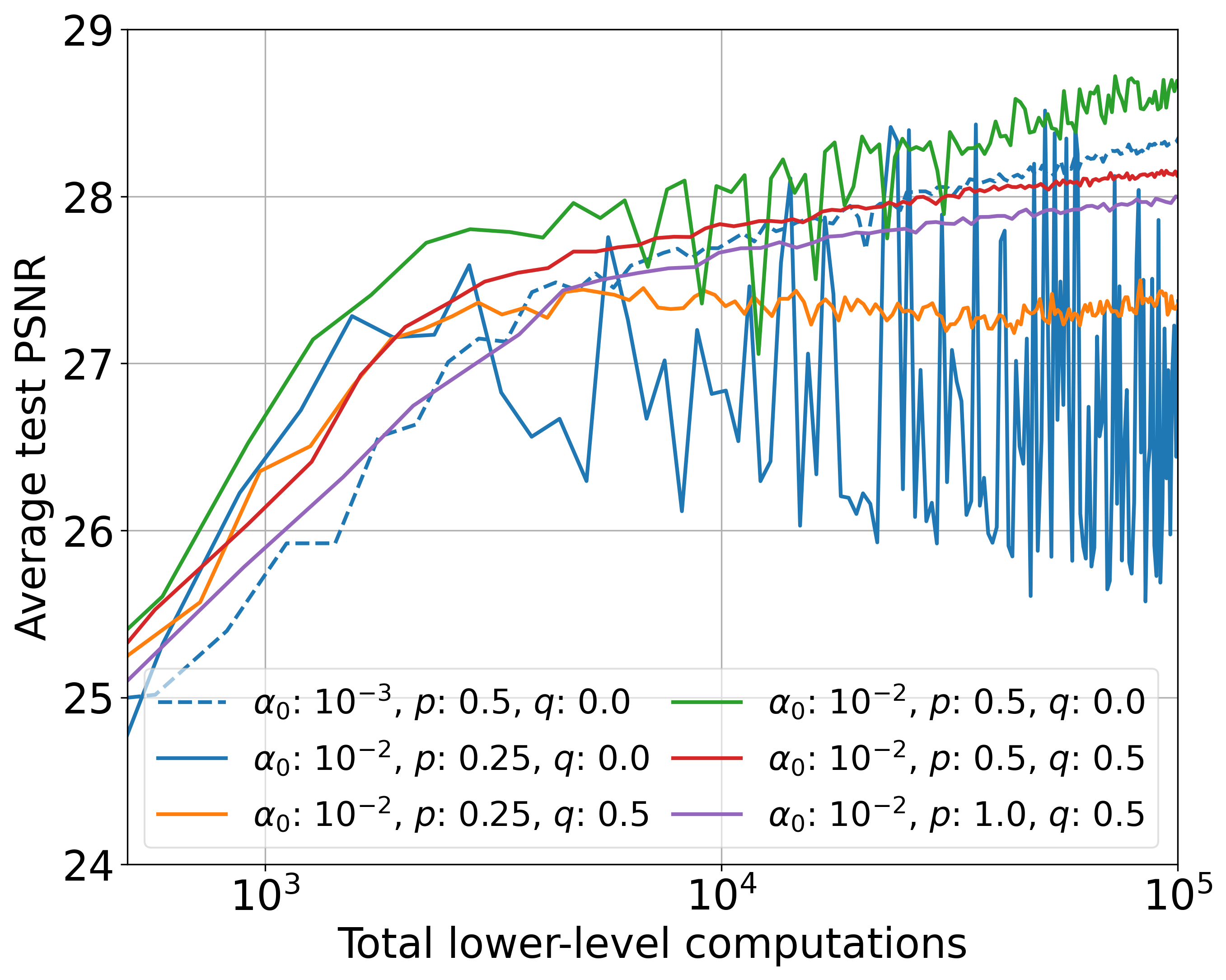}}\hfill
\subfloat[]{\label{fig:1e-4_psnr}\dn{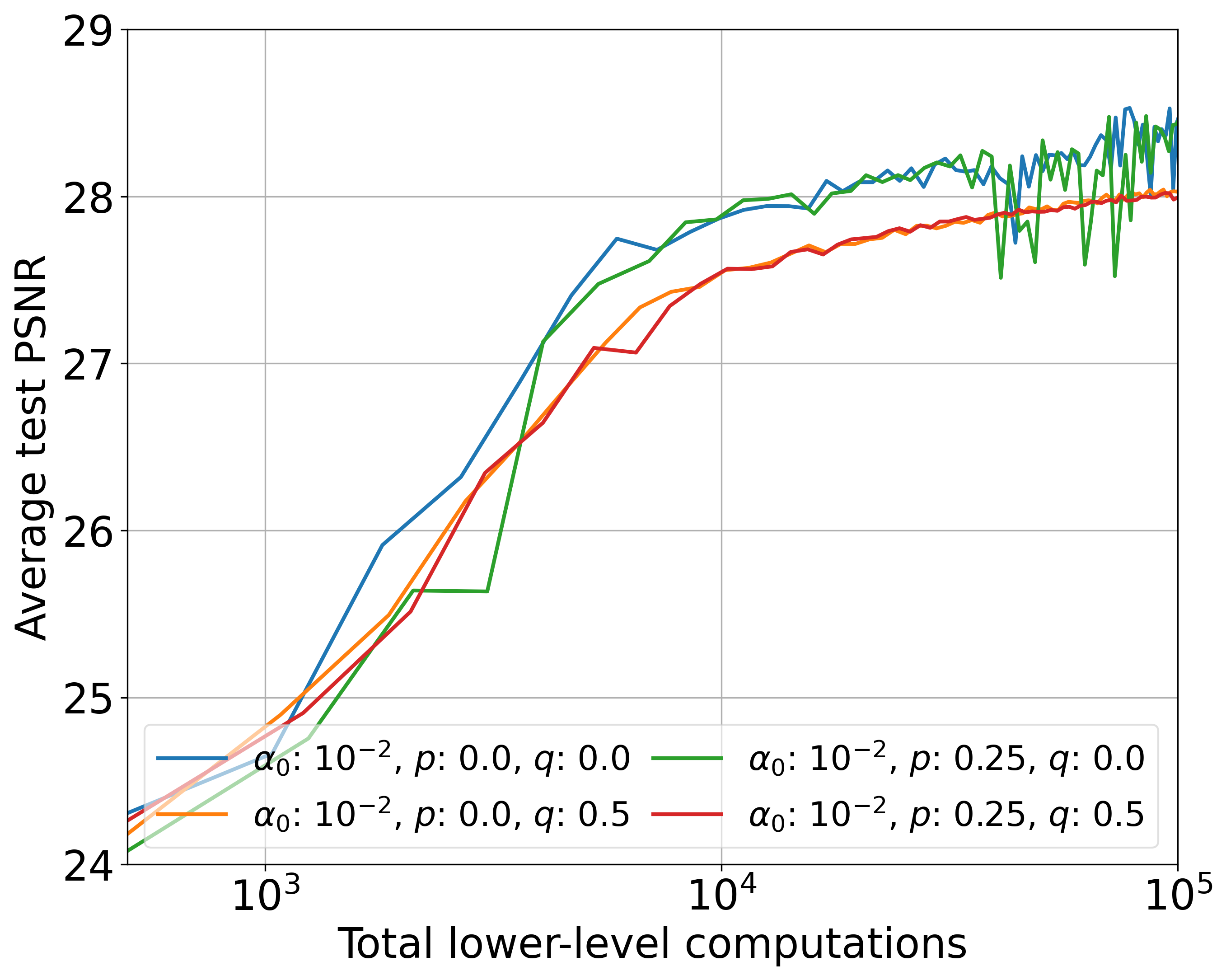}}
\caption{The accuracy schedule, more than the step size, determines finite-budget performance (denoising with CRR). Columns: initial accuracy $\epsilon_0$. (a)--(c) Training loss with a fixed step size for different accuracy decay exponents $p$ and initial step sizes $\alpha_0$: a loose $\epsilon_0$ needs an aggressive decay, a tight one hardly any. (d)--(f) Test PSNR of the best fixed and decreasing step size configurations: once $p$ is chosen well, a fixed step size gives the highest PSNR within the budget, and a decreasing one the smallest fluctuations. Training losses with decreasing step sizes are shown in Fig.~5 of the supplementary material.}
\label{fig:1e0_1e-2_1e-4_denoise}
\vspace*{-1pt}
\end{figure*}

We evaluated performance based on two criteria: convergence speed under a limited computational budget and lower overall loss with more stable behavior throughout training, especially at its final stages. We measure computational cost as the total number of lower-level solver iterations plus the linear system solver (used in computing the hypergradient) iterations. This measure captures the dominant sources of computational effort in bilevel optimization while remaining independent of implementation or hardware details \cite{MAID}. We set a budget of $10^5$ total computational cost to ensure a fair comparison across experiments. 

First, by starting from a large accuracy $\epsilon_0 = 1$, it can be seen in Figures~\ref{fig:1e0_fixed} and \ref{fig:1e0_psnr} (and Fig.~5(a) of the supplementary material) that only aggressive accuracy decrease schedules such as $k^{-1}$ and $k^{-2}$ result in a meaningful reduction in loss values. 

While a large tolerance of $\epsilon_0 = 1$ combined with a fixed large step size of $\alpha_0 = 10^{-2}$ and any accuracy schedule less aggressive than $k^{-2}$ tends to be suboptimal or even divergent, as shown in Fig.~5(a) of the supplementary material, a modest step size decay such as $k^{-0.5}$ significantly improves performance, making it comparable to the fixed $\alpha_0 = 10^{-3}$, and stabilizes the training loss behavior. 

Finally, the test PSNR results in \cref{fig:1e0_psnr} show that while decreasing the step size reduces test-time variance, its benefits are most pronounced when combined with a larger initial step size ($\alpha_0 = 10^{-2}$). The best overall performance is achieved using a fixed step size of $\alpha_0 = 10^{-2}$. These results suggest that the choice of the accuracy schedule is more critical.  

Starting with a mid-level accuracy of $\epsilon_0 = 10^{-2}$, a larger initial step size of $\alpha_0 = 10^{-2}$ performs better compared to the corresponding settings for $\alpha_0 = 10^{-3}$, as shown in \cref{fig:1e-2_fixed} and Fig.~5(b) of the supplementary material. In \cref{fig:1e-2_fixed}, the accuracy decay rate $p = 1$ appears overly aggressive, while $p = 0.25$ is insufficient. The intermediate value $p = 0.5$ yields the best performance.

With a decreasing step size (Fig.~5(b) of the supplementary material), the overall performance and the ranking of suitable accuracy decay schedules remain similar to those observed with a fixed step size, underscoring the importance of selecting an appropriate accuracy schedule. Interestingly, the setting $\alpha_0 = 10^{-2}, p = 0.25, q = 0.5$, while effective during training (Fig.~5(b) of the supplementary material) and providing the optimal theoretical rate (\cref{convergence_rate}), leads to suboptimal test performance. This is further supported by the test PSNR results in \cref{fig:1e-2_psnr}, where a decreasing step size yields significantly lower variance, but
the best performance under a fixed computational budget is achieved using a fixed step size combined with the appropriate accuracy schedule $k^{-0.5}$, reaching $0.5$ dB PSNR improvement over the counterpart with decreasing step size.

Starting with a high accuracy level of $\epsilon_0 = 10^{-4}$, Figures~\ref{fig:1e-4_fixed} and \ref{fig:1e-4_psnr} (and Fig.~5(c) of the supplementary material) show that, in a finite-time setting, decreasing the accuracy is less critical and may even slightly slow down performance due to more computations than necessary. A similar pattern to the mid-accuracy case is observed for the step size: the larger step size ($\alpha_0 = 10^{-2}$) yields the best performance even without any decay, while decreasing step size variants reduce variance but result in overall weaker fixed-budget outcomes (see \cref{fig:1e-4_psnr}).

However, this high-accuracy scenario is less attractive in practice, as it requires significant computational effort from the early stages of training, and what constitutes sufficiently high accuracy for a given problem is not known a priori.  Note that validation and testing are computationally expensive in the bilevel setting, unlike in common machine learning tasks, as they require solving the lower-level problem with many iterations to achieve high-accuracy inference. Thus, one might want to avoid it while still ensuring that the method does not behave erratically, which encourages the use of a decreasing step size.

Comparing the best configurations across all initial accuracies (Fig.~6 in the supplementary material) shows that a larger $p$ is needed for a larger $\epsilon_0$, that fixed step sizes give the best finite-budget performance, and that decreasing step sizes reduce training variance and yield more stable test results. These results suggest that, by selecting an appropriate accuracy-decay schedule (independent of the initial accuracy) and using a fixed step size, one can reliably obtain a well-trained denoising prior.

\subsection{Inpainting with Convex Ridge Regularizer}
To analyze our approach on a more complicated inverse problem (non-trivial forward operator), we consider an image inpainting task on $96 \times 96$ grayscale crops of $128$ BSDS500 images~\cite{BSDS500}, with $70\%$ of the pixels removed by a random Bernoulli mask and Gaussian noise ($\sigma = 0.02$) on the observed pixels. The data fidelity becomes $\|Ax - y\|_2^2$ with the masking operator $A$, and the CRR of \cref{subsec:denoise_CRR} (one input channel) is augmented by $\tfrac{\xi}{2}\|x\|^2$, $\xi = 10^{-6}$, to ensure strong convexity; full details are given in the supplementary material, Appendix~E.

Considering a large initial accuracy of $\epsilon_0 = 10^{-1}$ and a large step size of $\alpha_0 = 10^{-2}$, chosen via grid search, together with accuracy and step size decay exponents $p \in \{0.5, 1\}$ and $q \in \{0, 0.25, 0.5\}$, the training and testing results for learning a CRR regularizer for image inpainting are shown in Fig.~8 of the supplementary material. As illustrated in Fig.~8(a), for both values of $p$, the fixed step size ($q=0$) and the moderately decaying step size ($q=0.25$) perform comparably during training, with the fixed-step variant achieving higher PSNR in testing. In contrast, the more aggressive decay $q=0.5$ leads to slower training and lower test PSNR for both values of $p$. Beyond a computational cost of about $10^3$, the test PSNR of the $p=1$ settings surpasses that of $p=0.5$.  Similar to the denoising case, starting from a large initial accuracy, a fixed step size combined with a suitable accuracy-decay schedule can yield the best overall training and testing performance. The reconstructions in \cref{fig:inpainting_ckp} show the same ranking: with $p = 1$, the fixed step size gives the highest test PSNR at every checkpoint ($29.52$\,dB at a cost of $2 \times 10^4$, against $29.15$ and $28.37$\,dB for $q = 0.25$ and $q = 0.5$).

\begin{figure}[!t]
\centering
\newcommand{\ik}[2]{\includegraphics[height=#1\columnwidth]{Figs/Checkpoints/inpainting/tiles/#2}}
\setlength{\tabcolsep}{0.6pt}
\begin{tabular}{@{}c@{\hspace{1pt}}c@{\hspace{6pt}}c@{\hspace{1pt}}c@{}}
\ik{0.152}{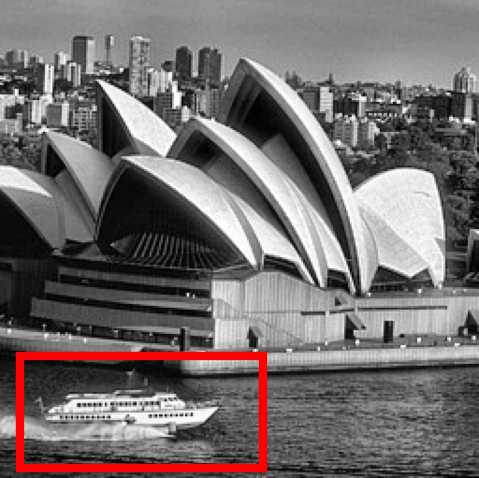} & \ik{0.152}{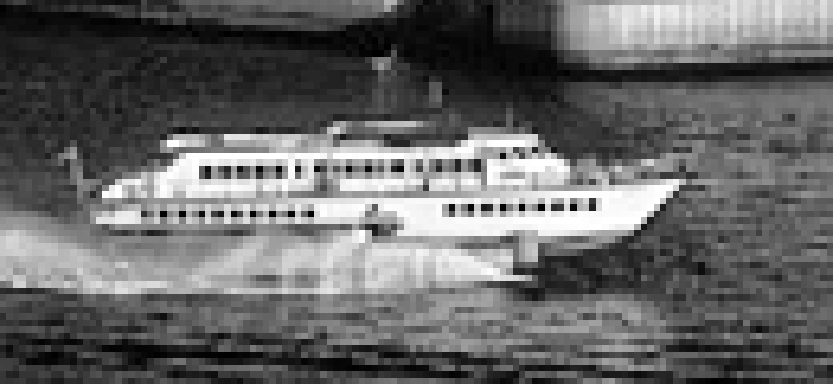} & \ik{0.152}{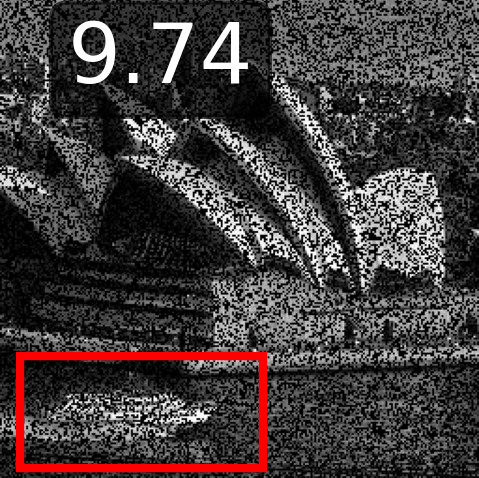} & \ik{0.152}{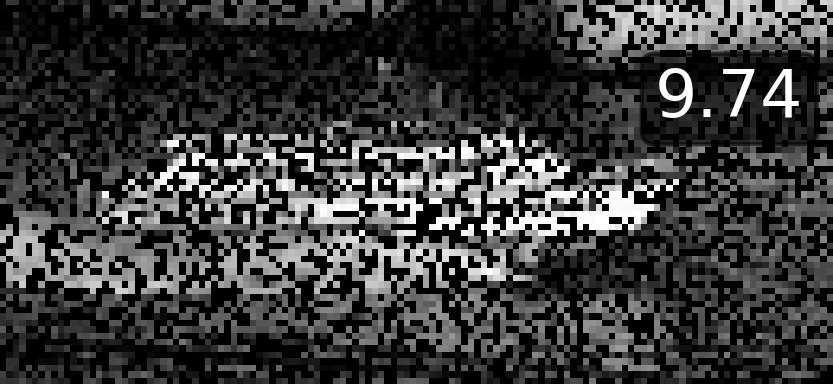}\\
\scriptsize Ground truth & \scriptsize Enlarged & \scriptsize Observed & \scriptsize Enlarged
\end{tabular}\\[3pt]
\newcommand{\ir}[1]{\includegraphics[width=0.93\columnwidth,trim=0 12 0 12,clip]{Figs/Checkpoints/inpainting/cropped_zoom/#1_inpai_zoomed_row.png}}
\newcommand{\ql}[1]{\rotatebox[origin=c]{90}{\scriptsize #1}}
\renewcommand{\arraystretch}{0}
\begin{tabular}{@{}m{0.035\columnwidth}@{\hspace{2pt}}m{0.93\columnwidth}@{}}
\ql{$q{=}0$} & \ir{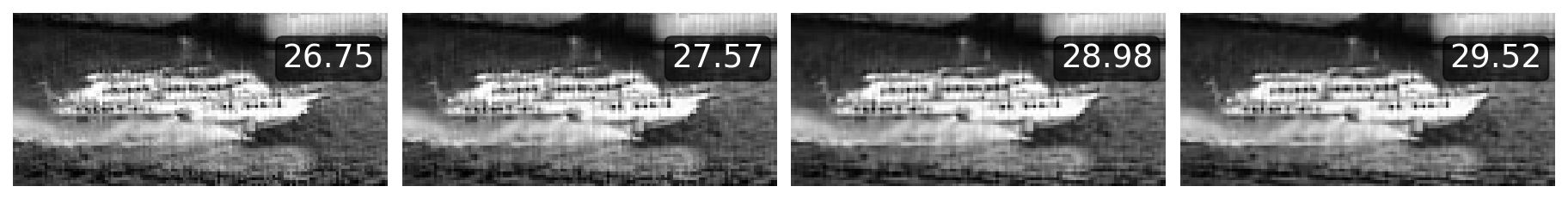}\\
\ql{$q{=}0.25$} & \ir{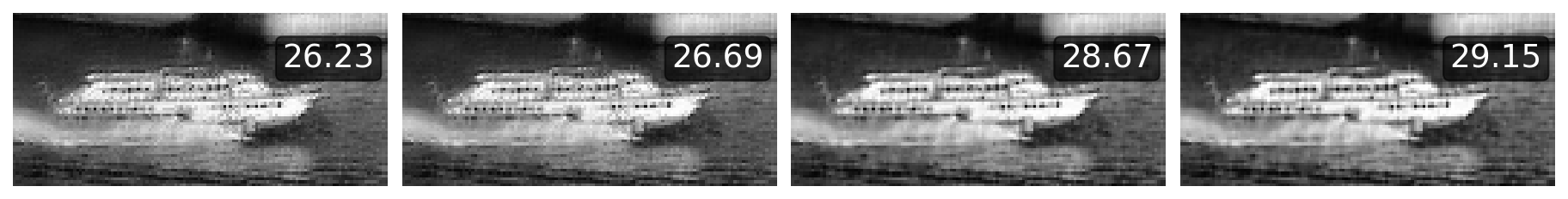}\\
\ql{$q{=}0.5$} & \ir{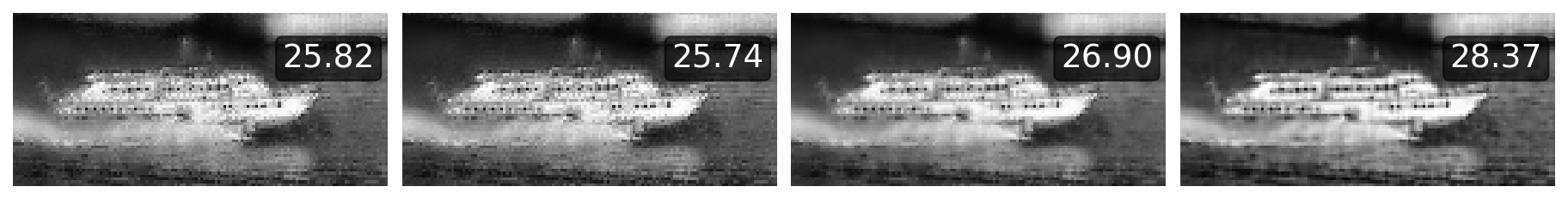}
\end{tabular}
\caption{Inpainting with the CRR regularizer ($70\%$ of the pixels missing). Top: ground truth and observed test image, each with the enlarged region. Rows: the enlarged reconstruction at training checkpoints of computational cost $2{,}500$, $5{,}000$, $10{,}000$, and $20{,}000$ (left to right) for $\epsilon_0 = 10^{-1}$, $\alpha_0 = 10^{-2}$, $p = 1$, and step size decay exponents $q \in \{0, 0.25, 0.5\}$; inset numbers are test PSNR (dB). The fixed step size is ahead at every checkpoint.}
\label{fig:inpainting_ckp}
\end{figure}

\subsection{Preconditioning matters: ISGD vs IAdam  for  Input-Convex Neural Network Regularizer}
\label{subsec:icnn}
In this experiment, we consider the ICNN architecture \cite{amos2017input,InputConvex} as the regularizer \( R_\theta(x) \) with $2$ layers and the following form:
$$
R_\theta(x) = \sum_{i=1}^{C} \sum_{j=1}^{H} \sum_{k=1}^{W} \bigl[ \phi\bigl( W_z \phi(W_x x + b_x) + b_z \bigr) \odot \exp(s) \bigr]_{i,j,k},
$$
where $W_x$ and $W_z$ denote 2D convolution operators, \( \odot \) is the element-wise product, and \( W_x \), \( W_z \), \( b_x \), \( b_z \), and \( s \) are trainable parameters. Note that  $W_z$  is constrained to be non-negative; we enforce this by clamping its negative weights to zero. 

The activation function \( \phi: \mathbb{R} \to \mathbb{R} \) is a smoothed clipped ReLU (quadratic on $[0, \nu)$, linear beyond, with smoothing parameter $\nu = 10^{-3}$). This architecture ensures convexity of the regularizer \( R_\theta(x) \) with respect to \( x \) and satisfies the regularity assumptions required in our theoretical analysis.

Fig.~9 in the supplementary material compares ISGD with IAdam, where IAdam refers to the Adam algorithm \cite{kingma2015adam} with the inexact stochastic hypergradient replacing the exact stochastic hypergradient.  It spans $\epsilon_0 \in \{1, 10^{-2}, 10^{-4}\}$ with $\alpha_0 = 10^{-2}$.

Since ICNN parameters can vary significantly in magnitude, the optimization process may be affected by using the same step size across all of them. As shown in Fig.~9, preconditioning and diagonal rescaling in Adam provide clear advantages over SGD across all settings.  In contrast, ISGD variants with fixed step size display higher variance and perform worse than IAdam under the same computational budget.

On the other hand, introducing a decreasing step size, while not significantly affecting IAdam's performance, substantially stabilizes training and reduces the variance of ISGD variants, making them comparable to IAdam by the end of training, as shown in Fig.~9(b).  

\begin{figure}[!t]
\centering
\newcommand{\ck}[1]{\includegraphics[width=0.176\columnwidth]{Figs/Checkpoints/ICNN/tiles/#1}}
\newcommand{\rl}[2]{\rotatebox[origin=c]{90}{\scriptsize\begin{tabular}{@{}c@{}}#1\\#2\end{tabular}}}
\newcommand{\lb}[2]{\scriptsize\begin{tabular}[t]{@{}c@{}}#1\\#2\end{tabular}}
\setlength{\tabcolsep}{0.6pt}
\renewcommand{\arraystretch}{0.5}
\begin{tabular}{@{}>{\centering\arraybackslash}m{0.05\columnwidth}*{5}{>{\centering\arraybackslash}m{0.176\columnwidth}}@{}}
\multicolumn{6}{c}{%
  \renewcommand{\arraystretch}{0.8}%
  \begin{tabular}{@{}c@{\hspace{1pt}}c@{\hspace{8pt}}c@{\hspace{1pt}}c@{}}
  \ck{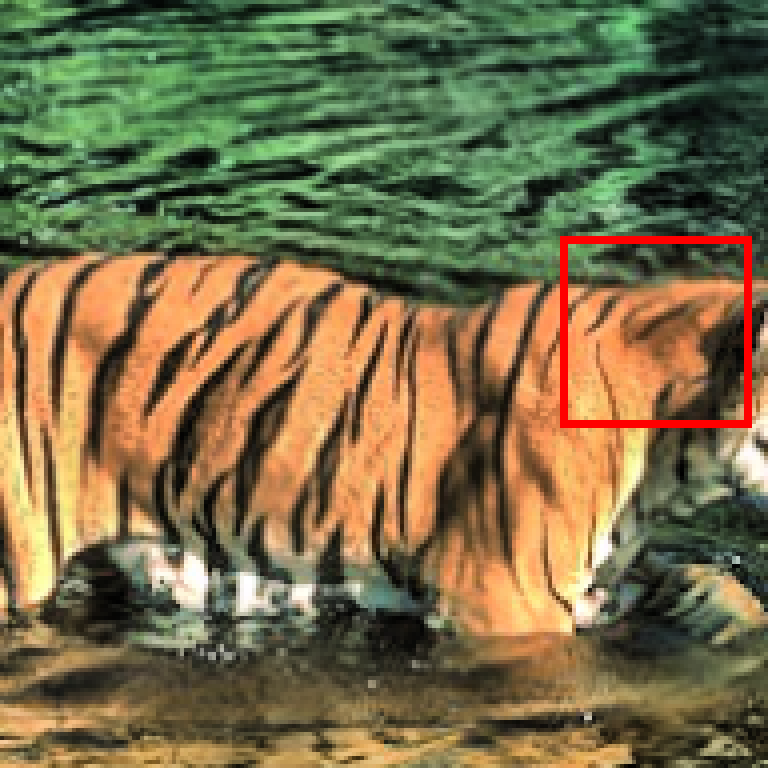} & \ck{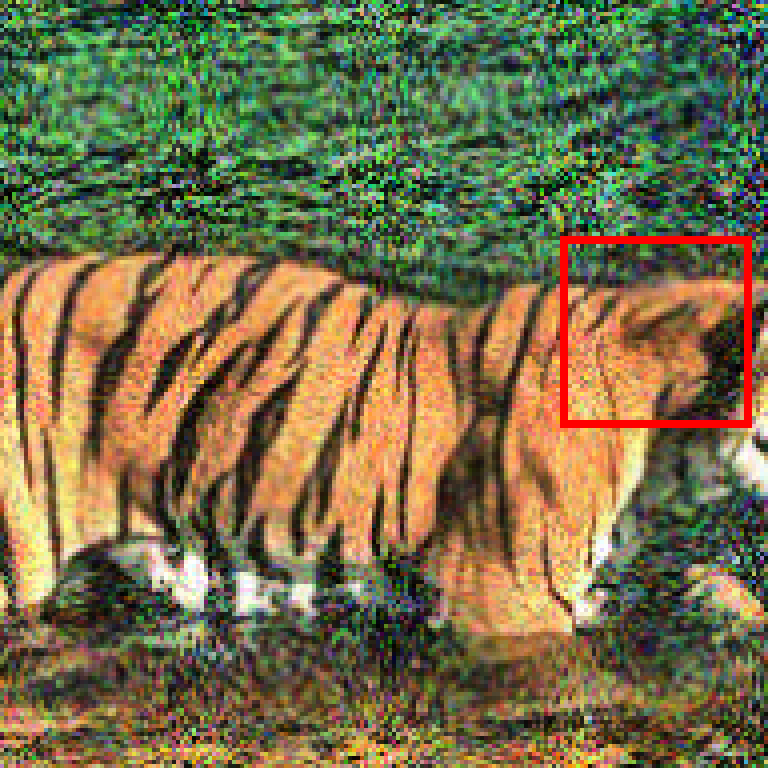} & \ck{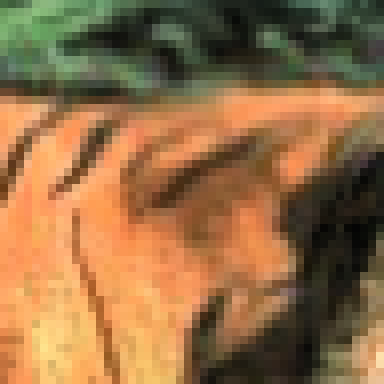} & \ck{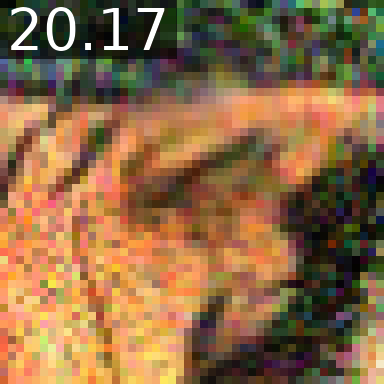} \\[-3pt]
  \lb{Ground}{truth} & \lb{Noisy}{input} & \lb{Ground truth}{(zoom)} & \lb{Noisy input}{(zoom)}
  \end{tabular}} \\ \noalign{\vspace{5pt}}
 & \multicolumn{5}{c}{\scriptsize Training checkpoints (increasing computational cost) $\longrightarrow$} \\ \noalign{\vspace{2pt}}
\rl{ISGD}{$q{=}0$}    & \ck{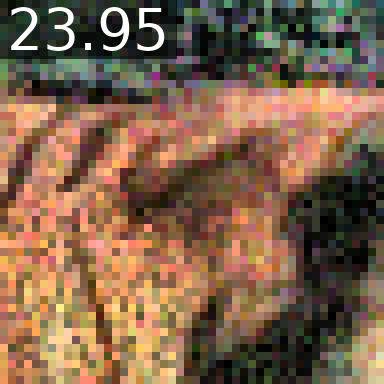} & \ck{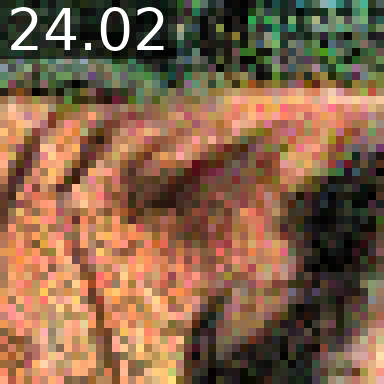} & \ck{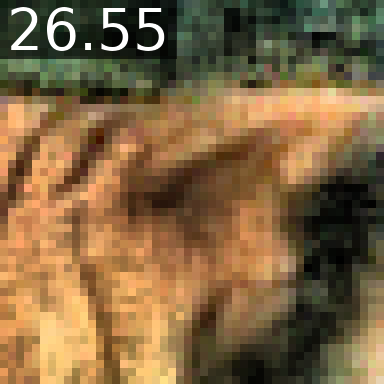} & \ck{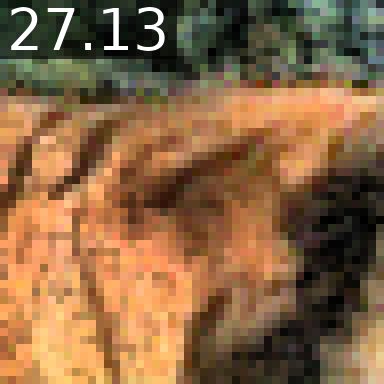} & \ck{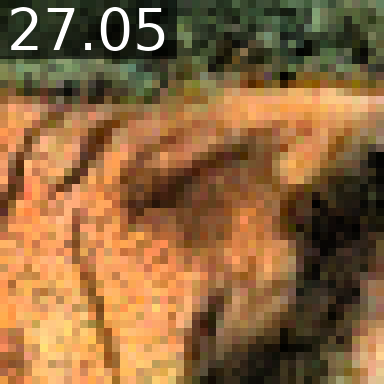} \\[1pt]
\rl{ISGD}{$q{=}0.5$}  & \ck{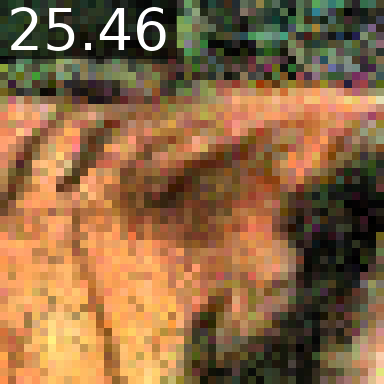} & \ck{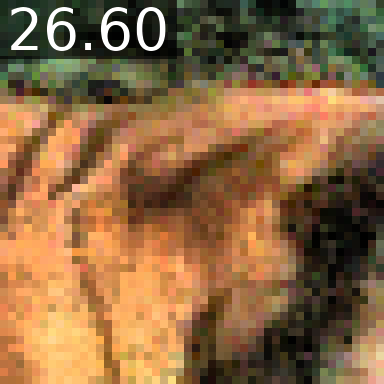} & \ck{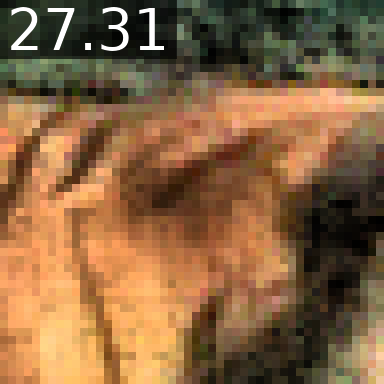} & \ck{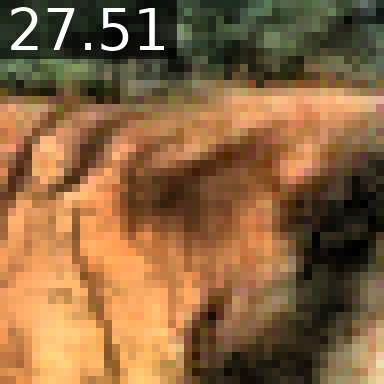} & \ck{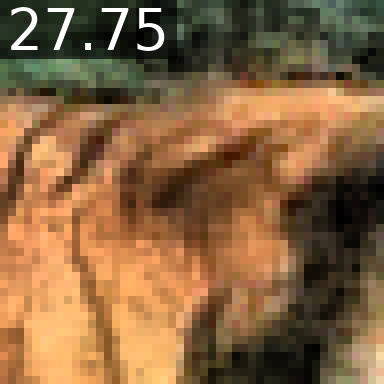} \\[1pt]
\rl{IAdam}{$q{=}0$}   & \ck{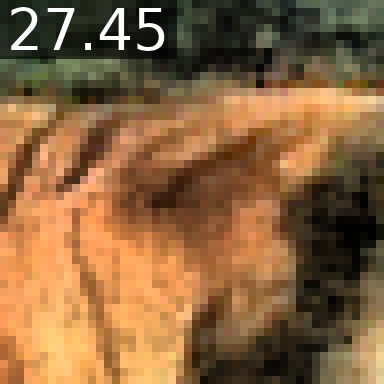} & \ck{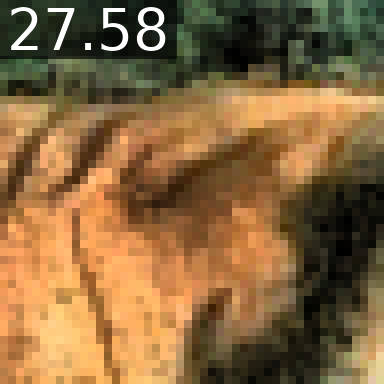} & \ck{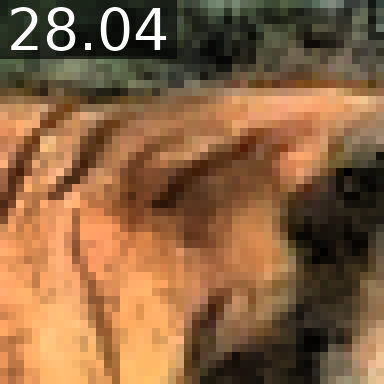} & \ck{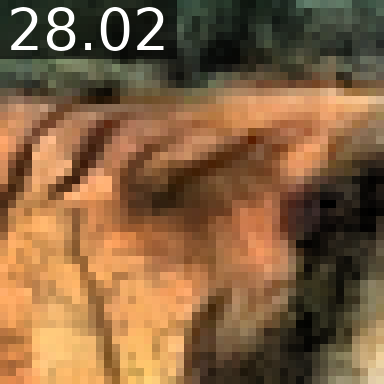} & \ck{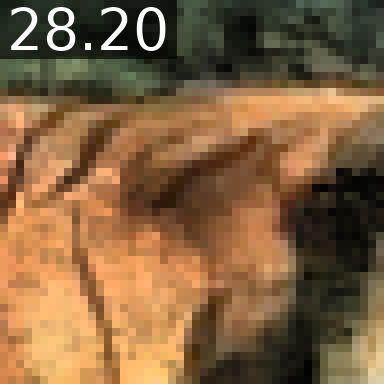} \\[1pt]
\rl{IAdam}{$q{=}0.5$} & \ck{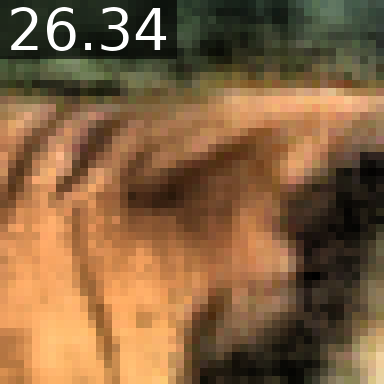} & \ck{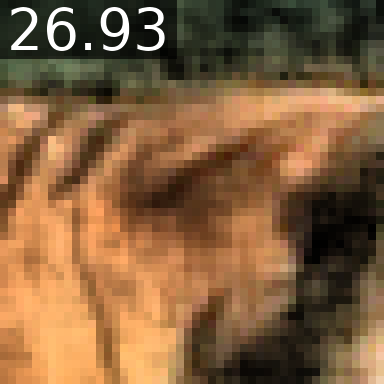} & \ck{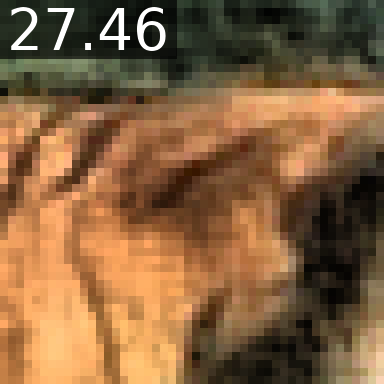} & \ck{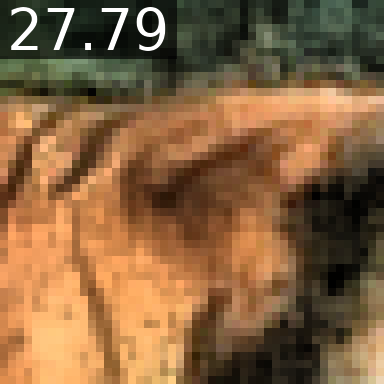} & \ck{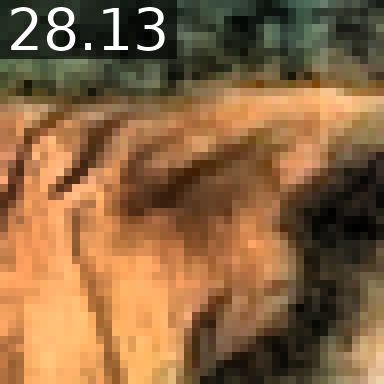} \\
\end{tabular}
\caption{Denoising with the ICNN regularizer. Top: ground truth and noisy test image, with the red region enlarged $4\times$. Rows: the enlarged region at five training checkpoints of increasing computational cost for the best configurations of ISGD (rows 1--2) and of IAdam, its Adam variant (rows 3--4); inset numbers are test peak signal-to-noise ratio (PSNR, dB). All runs use $\alpha_0 = 10^{-2}$; rows 1, 2, and 4 use $(\epsilon_0, p) = (1, 1)$ and row 3 uses $(10^{-2}, 0.25)$. At the first checkpoint, IAdam is $2.0$--$3.5$\,dB ahead of ISGD.}
\label{fig:ICNN_ckp}
\end{figure}

Considering the training checkpoints in \cref{fig:ICNN_ckp}, the importance of preconditioning and the superior performance of IAdam under a limited computational budget become even more evident. The substantial gap of $2.0$--$3.5$\,dB PSNR in the early phase of training, compared with ISGD, highlights the role of rescaling the hypergradient in accelerating convergence.

\subsection{Comparison with single-loop stochastic bilevel methods}
\label{subsec:singleloop_comparison}
To complement the complexity discussion in \cref{subsec:total_complexity}, we compare scheduled ISGD/IAdam against representatives of the single-loop family under a common computational cost meter: TTSA \cite{hong2023TTSA}, SOBA and its variance-reduced variant SABA \cite{dagreou2022SOBA}, and MA-SOBA \cite{chen2024MASOBA}. Each lower-level gradient step and each Hessian--vector product is counted toward a shared cost budget of $2\times 10^4$, so the comparison is budget-matched rather than iteration-matched. Warm-starting of the lower-level iterate is enabled for all methods---the single-loop methods maintain persistent iterates by construction, and the scheduled methods use the per-sample warm starts of \cref{rem:warmstart}---and all methods receive a comparable tuning effort by long-horizon grid search on the denoising task of \cref{subsec:denoise_CRR}: over the step sizes $(\alpha, \beta)$ for the baselines, and over $(\alpha_0, \epsilon_0, p)$ for ISGD/IAdam. We run each baseline in its published form, under its own theoretical step-size schedule: the two-timescale decay $\alpha_k \propto k^{-3/5}$, $\beta_k \propto k^{-2/5}$ for TTSA \cite{hong2023TTSA}, fixed coupled steps (their Theorem~1) and decreasing steps $\beta_t \propto t^{-1/2}$ (their Theorem~2) for SOBA, and constant steps for SABA (their Theorem~3) \cite{dagreou2022SOBA}; step-size constants are tuned by grid search over a horizon long enough to expose late-stage instability. The correctness checks of our implementations are reported in the supplementary material, Appendix~F. Considerable tuning effort was nevertheless required: the single-loop analyses assume steps bounded by $\mathcal{O}(1/L_H)$ with $L_H$, the Lipschitz constant of $\nabla_x h(\cdot,\theta)$, known and static, whereas on this problem class $L_H$ grows by orders of magnitude as the regularizer's scaling parameters are learned. Aggressively-stepped configurations that look strong at small budgets therefore diverge later in training; any run reaching a non-finite loss is flagged and invalidated, never reported through its pre-divergence peak. Respecting the coupling between the timescales---the upper-level step must be an order of magnitude smaller than the lower-level step (cf.\ the condition $\gamma^2 \lesssim \rho$ in \cite{dagreou2022SOBA})---proved decisive. By contrast, the solve-to-tolerance framework is immune by construction: its inner solver adapts the step size automatically (backtracking with Barzilai--Borwein estimates) and the Hessian system is solved by a Krylov method, which is step-size free. 

\IfFileExists{Figs_comparison/comparison_cost_loss.png}{%
\begin{figure}[tbhp]
\centering
\subfloat[]{%
    \label{fig:comparison_loss}
    \includegraphics[width=0.46\columnwidth]{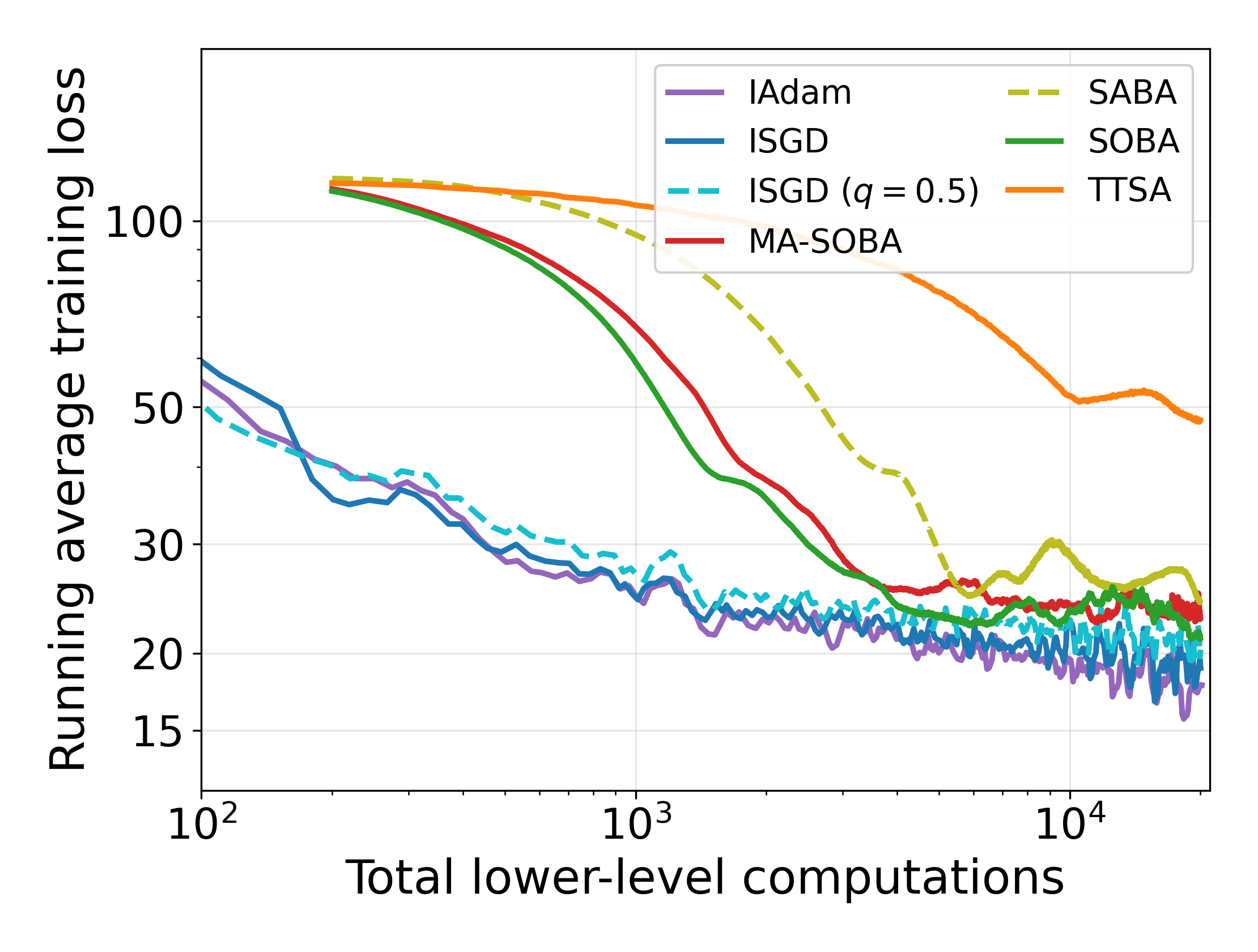}
}\hfill
\subfloat[]{%
    \label{fig:comparison_psnr}
    \includegraphics[width=0.46\columnwidth]{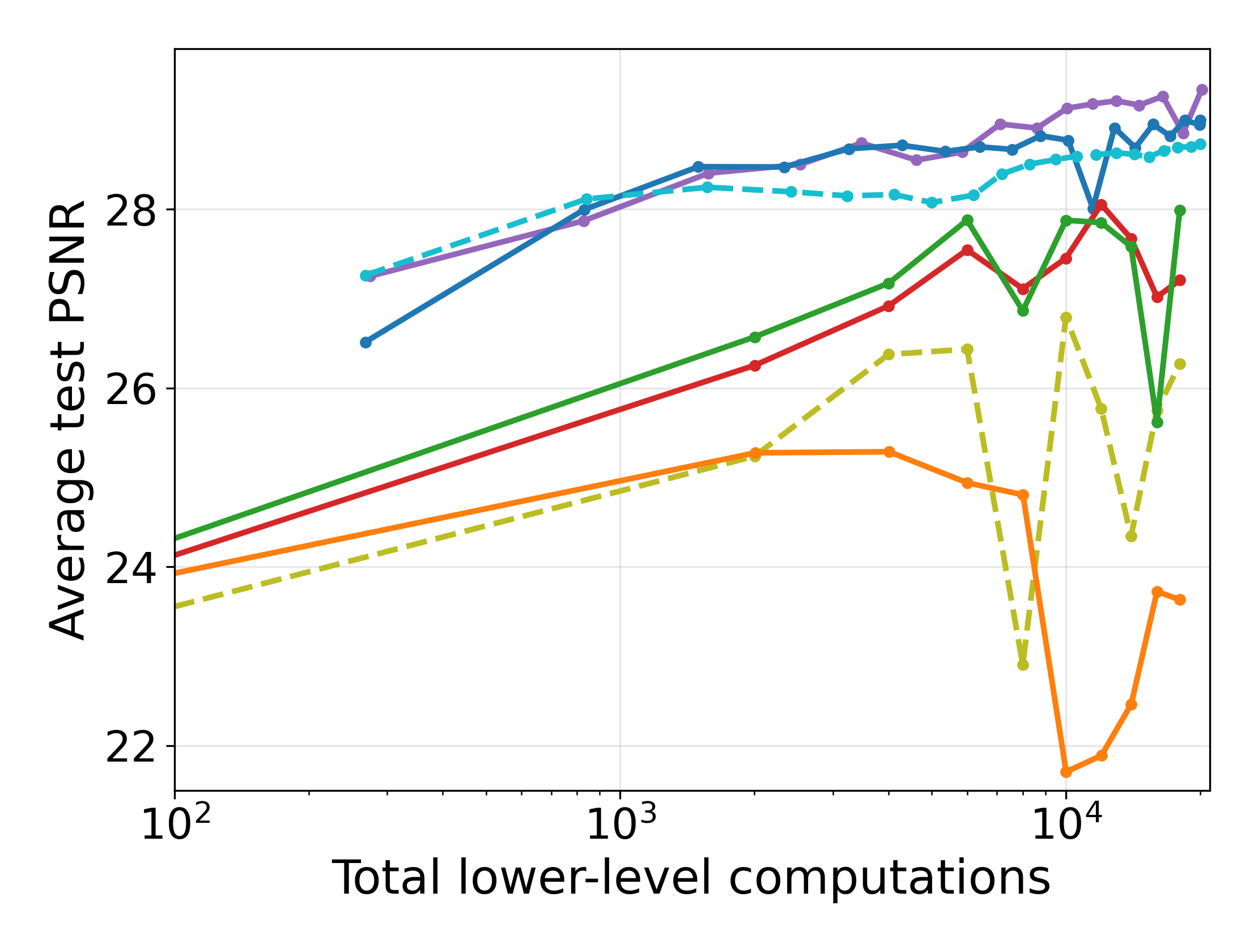}
}
\caption{Budget-matched comparison ($2\times 10^4$ total cost) of scheduled ISGD/IAdam with single-loop stochastic bilevel methods (TTSA, SOBA, SABA, MA-SOBA) on denoising with a CRR regularizer. (a) Running-average training loss and (b) test PSNR, both against the shared computational cost meter (lower-level iterations + Hessian-system iterations); axes cropped to the informative range, legend in (a).}
\label{fig:singleloop_comparison}
\end{figure}
}{}

\begin{table}[tbhp]
\centering
\caption{Best test PSNR/SSIM within a total computational budget of $2\times 10^4$ (denoising with CRR on a fixed dataset). All methods use warm-started lower-level iterates. Baselines run in their published form under their own step-size schedules; constants of all methods are tuned by long-horizon grid search (over $(\alpha,\beta)$ for the baselines and over $(\alpha_0, \epsilon_0, p)$ for the scheduled methods).}
\label{tab:summary}
\begin{tabular}{lcc}
\toprule
\textbf{Method (Configuration)} & \textbf{PSNR (dB)} & \textbf{SSIM} \\
\midrule
IAdam ($\alpha_0{=}10^{-2}$, $\epsilon_0{=}10^{-1}$, $p{=}1$) & \textbf{29.34} & \textbf{0.828} \\
ISGD ($\alpha_0{=}10^{-3}$, $\epsilon_0{=}10^{-1}$, $p{=}1$) & 29.00 & 0.814 \\
ISGD decaying $\alpha$ (per ``ISGD'' with $q{=}0.5$) & 28.73 & 0.800 \\
MA-SOBA \cite{chen2024MASOBA} ($\alpha{=}10^{-3},\beta{=}10^{-2}$) & 28.06 & 0.758 \\
SOBA \cite{dagreou2022SOBA} ($\alpha{=}10^{-3},\beta{=}10^{-2}$) & 27.99 & 0.775 \\
SABA \cite{dagreou2022SOBA} ($\alpha{=}3{\cdot}10^{-4},\beta{=}3{\cdot}10^{-3}$) & 26.79 & 0.695 \\
TTSA \cite{hong2023TTSA} ($\alpha_k{\propto}k^{-3/5}$, $\beta_k{\propto}k^{-2/5}$) & 25.29 & 0.619 \\
\bottomrule
\end{tabular}
\end{table}

 Properly tuned, the single-loop methods are genuinely competitive (\cref{tab:summary}): MA-SOBA ($28.06$\,dB) and SOBA ($27.99$\,dB) come within about a decibel of the best scheduled configurations. The scheduled methods nevertheless lead once their accuracy schedule is chosen well: starting from a loose tolerance $\epsilon_0 = 10^{-1}$ with $p = 1$, ISGD reaches $29.00$\,dB and IAdam $29.34$\,dB. The mechanism is visible in \cref{fig:comparison_psnr}: a loose initial tolerance makes the early iterations cheap, so the solve-to-tolerance curves enter the plot far to the left and are already above $27$\,dB by a cost of $10^3$, where the single-loop methods are still near $25$\,dB.  With the same $\alpha_0$, $\epsilon_0$ and $p$, the theoretically supported decaying step size ($q = 0.5$) reaches $28.73$\,dB against $29.00$\,dB for the fixed step---a gap of only $0.27$\,dB---while attaining a slightly \emph{higher} PSNR early on (at a cost of $10^3$, $28.11$ vs.\ $28.00$\,dB) and a smaller test-PSNR fluctuation over the final evaluations. Much of the apparent finite-budget penalty of decaying step sizes is therefore attributable to the accompanying accuracy schedule rather than to the step-size decay itself. SABA trails its non-variance-reduced counterpart at this budget, and TTSA is slowest: its two-timescale decay, which is what kept it stable, sacrifices early progress, and the more aggressive tuned configuration that looked strongest at the tuning horizon diverged at the deployment scale and was invalidated. Beyond raw performance, the practical distinction is \emph{robustness}: reaching these numbers required respecting the timescale coupling, long-horizon tuning, and accepting a residual divergence risk, whereas every scheduled ISGD/IAdam configuration in this paper ran without any stability tuning of the lower-level solver.  One caveat applies: the comparison uses a fixed dataset (deterministic crops and noise realizations), as required by the persistent iterates of the single-loop methods.

\paragraph{Empirical validation of \cref{cor:fixed_horizon} and \cref{rem:warmstart}} On the same fixed-dataset benchmark (Fig.~10 in the supplementary material), with $\epsilon_k = \mathcal{O}(k^{-1/4})$ a horizon-tuned constant step $\alpha = 10^{-3}$ attains $28.99$\,dB ($0.814$ SSIM) against $26.36$\,dB for the anytime decay $\alpha_k \propto k^{-0.55}$ with $\alpha_0 = 10^{-2}$, as predicted by \cref{cor:fixed_horizon}; this margin mixes schedule shape and initial step size, and shrinks to the $0.27$\,dB above when $\alpha_0$, $\epsilon_0$, and $p$ are held fixed. With persistent warm starts, the number of inner iterations per upper-level step \emph{decreases} from $15.9$ (first fifth of training) to $3.2$ (last fifth), approaching the $\mathcal{O}(1)$ amortized regime of \cref{rem:warmstart}, whereas cold starts grow from $15.0$ to $33$--$41$, consistent with the $\log k$ term of \cref{cor:total_complexity}; at equal budget, warm starting completes $1.76\times$ more upper-level iterations and gains $2.0$\,dB ($28.73$ vs.\ $26.70$).

\section{Conclusion and Future Work}
\label{sec:conclusions}
In this work, we advanced the theoretical and practical analysis of inexact stochastic bilevel optimization, establishing explicit convergence rates under different accuracy and step size schedules, and validating these findings on large-scale imaging problems.  Once metric conventions are unified, the scheduled method attains the same optimal rate as single-loop stochastic bilevel methods, with a matching total complexity under warm starts, and in a budget-matched comparison on learned-regularizer denoising it attains the highest test PSNR. Our analysis shows that while decaying step sizes are necessary for asymptotic guarantees, accuracy scheduling plays a more decisive role in practice, and preconditioned adaptive methods such as Adam can substantially enhance performance, with around a $2$\,dB PSNR advantage for ICNN regularizers in the early training phase. Experiments across denoising and inpainting tasks confirm that theory-informed schedules improve stability, but also highlight that fixed or more aggressive schedules often deliver superior results under realistic computational budgets. As practical defaults, we recommend starting from a loose tolerance $\epsilon_0$ with an aggressive accuracy decay, and using a horizon-tuned constant step size when the budget is known (\cref{cor:fixed_horizon}) or a decaying one when validation is too expensive to monitor training.

Developing non-asymptotic convergence theory that captures finite-time behavior, designing adaptive accuracy and step size schedules guided by data and budget constraints, and extending our analysis to nonconvex lower-level problems remain key challenges for future work.

\bibliographystyle{IEEEtran}
\bibliography{references}

\clearpage
\section*{Supplementary material}
The appendices below are the supplementary material of the submitted manuscript; equations, results, tables, and figures are numbered consecutively with the main text.

\appendices
\section{Proof of Lemma~1}
\label{supp:lemma1}
\begin{proof}
Write $H_i(x) \coloneqq \nabla^2_{xx} h_i(x, \theta)$, $J_i(x) \coloneqq \nabla^2_{\theta x} h_i(x, \theta)$, and $w_i(x) \coloneqq H_i(x)^{-1} \nabla g_i(x)$, so that by the implicit function theorem the exact per-sample hypergradient is $\nabla f_i(\theta) = -J_i(\hat{x}_i)\, w_i(\hat{x}_i)$, and its approximation is $-J_i(\tilde{x}_i)\, \tilde{w}_i$. Decompose
\begin{multline*}
\|e_i(\theta)\| \le \underbrace{\|J_i(\tilde{x}_i)\| \,\|\tilde{w}_i - w_i(\tilde{x}_i)\|}_{\text{linear-system error}} \\
+ \underbrace{\|J_i(\tilde{x}_i) w_i(\tilde{x}_i) - J_i(\hat{x}_i) w_i(\hat{x}_i)\|}_{\text{lower-level error}} .
\end{multline*}
For the first term, $\mu$-strong convexity gives $\|H_i(\tilde{x}_i)^{-1}\| \le 1/\mu$, so a residual of $\tau$ yields $\|\tilde{w}_i - w_i(\tilde{x}_i)\| \le \tau/\mu$ and the term is bounded by $B_J \tau / \mu$. For the second term, add and subtract $J_i(\tilde{x}_i) w_i(\hat{x}_i)$:
\begin{equation*}
\|J_i(\tilde{x}_i)\| \, \|w_i(\tilde{x}_i) - w_i(\hat{x}_i)\| + \|J_i(\tilde{x}_i) - J_i(\hat{x}_i)\| \, \|w_i(\hat{x}_i)\|.
\end{equation*}
The map $x \mapsto w_i(x)$ is Lipschitz with constant $L_g/\mu + G L_H/\mu^2$: indeed,
$w_i(x) - w_i(x') = H_i(x)^{-1}\big(\nabla g_i(x) - \nabla g_i(x')\big) + \big(H_i(x)^{-1} - H_i(x')^{-1}\big)\nabla g_i(x')$, and $\|H_i(x)^{-1} - H_i(x')^{-1}\| \le \|H_i(x)^{-1}\|\|H_i(x') - H_i(x)\|\|H_i(x')^{-1}\| \le L_H \|x - x'\|/\mu^2$. Together with $\|w_i(\hat{x}_i)\| \le G/\mu$ and the $L_J$-Lipschitz continuity of $J_i$, this gives the bound (5) with the stated constants.

For (6), since $e_v(\theta) = \frac{1}{m}\sum_{i=1}^m v_i e_i(\theta)$, the triangle inequality and (5) give $\|e_v(\theta)\| \le \big(\frac{1}{m}\sum_i v_i\big)(c_1\delta + c_2\tau)$. Taking expectations and using $\big(\frac{1}{m}\sum_i v_i\big)^2 \le \frac{1}{m}\sum_i v_i^2$ (convexity of $t \mapsto t^2$), so that $\E\big[\big(\frac{1}{m}\sum_i v_i\big)^2\big] \le \frac{1}{m}\sum_i \E[v_i^2] \le \nu^2$, yields the claim; note that no independence of the $v_i$ is required. For $v_i \sim \mathrm{Binomial}(m, 1/m)$ one has $\E[v_i^2] = 2 - 1/m$, hence $\nu^2 \le 2$; for samplings with $\frac{1}{m}\sum_i v_i \equiv 1$ (e.g., multinomial minibatch selection) the factor improves to $1$.
\end{proof}

\section{Auxiliary results and proof of Theorem~1}
\label{supp:proofs}
Note that by weakening Assumption~2 to the case where the total expectation over all realizations  $\theta^k$  of the squared error is bounded, the theoretical results still hold; however, this condition is not practically verifiable.

\begin{lemma}\label{lemma_ABC}
         \cite[Proposition 2]{SGD_better} Expected Smoothness (ES): Let Assumption 1 hold, $f$ is bounded below by $f^{\text{inf}}$, and  $\E[v_i^2]<\infty$ for $1\leq i\leq m$. Then there exist constants $A, B, C\geq 0$ such that for all $\theta \in \mathbb{R}^d$, the stochastic gradient satisfies
     \begin{equation*}
         \E_{v \sim \mathcal{D}}[ \|\nabla f_{v}(\theta)\|^2] \leq 2A ( f(\theta) - f^{\text{inf}} )+ B \|\nabla f(\theta)\|^2 + C.
     \end{equation*}
\end{lemma}

The following auxiliary lemmas will be utilized in proving the convergence of inexact stochastic gradient.
\begin{lemma}\label{Young'sLemma} \cite[Lemma 2]{salehi2024inexactstochastic}
    For all $\theta \in \mathbb{R}^d$, $v \sim \mathcal{D}$, and $\zeta>0$ we have
    \begin{equation}\label{Young's}
        |\langle \nabla f(\theta), \E_{v\sim \mathcal{D}}[e_v(\theta)]\rangle| \leq \frac{\|\nabla f(\theta)\|^2}{2\zeta} + \frac{\zeta}{2} (\E_{v \sim D}[\|e_v(\theta)\|])^2
    \end{equation}
\end{lemma}

The following lemma relates the second moment of the inexact stochastic gradient with the stochastic gradient.
\begin{lemma}\label{second_moment_connect}\cite[Lemma 3]{salehi2024inexactstochastic}
Let $\eta>0$. The following relation between the second moment of the exact and inexact stochastic gradients holds
\begin{multline}\label{i_e_expectation}
    \E_{v \sim \mathcal{D}}[\|z_v(\theta)\|^2] \leq \frac{1+\eta}{\eta} \E_{v \sim \mathcal{D}} [\|\nabla f_v(\theta)\|^2] \\ + (1+\eta) \E_{v \sim \mathcal{D}}[\|e_v(\theta)\|^2].
\end{multline}
\end{lemma}

\begin{proposition}\label{ES_IBES}\cite[Proposition 1]{salehi2024inexactstochastic}
    Suppose Assumption~1 holds and $\E[v_i^2]<\infty$ for $1\leq i\leq m$. Let $\eta >0$, $\zeta > 0$, and  $b(\zeta) = 1 - \frac{1}{2\zeta}$, $c(\zeta, \theta) = \frac{\zeta}{2}E^2(\theta)$, $\tilde{A}(\eta) = \frac{1+\eta}{\eta} A$,
    $\tilde{B}(\eta) = \frac{1+\eta}{\eta} B$,
    $\tilde{C}(\eta, \theta) = \frac{1+\eta}{\eta} C + (1+\eta)E^2(\theta)$.
    Then, for all $\theta \in \mathbb{R}^d$, the inexact stochastic gradient estimator satisfies
     \begin{align}
        \langle \nabla f(\theta), \E_{v\sim \mathcal{D}}[z_v(\theta)]\rangle &\geq b(\zeta)\|\nabla f(\theta)\|^2 -c(\zeta, E^2(\theta)),\label{biasineq} \\
        \begin{split}
         \E_{v \sim \mathcal{D}}[ \|z_{v}(\theta)\|^2] &\leq 2\tilde{A}(\eta) ( f(\theta) - f^{\text{inf}} )\\
         &\quad + \tilde{B}(\eta) \|\nabla f(\theta)\|^2 + \tilde{C}(\eta, E^2(\theta)).\label{tildeABC}
        \end{split}
     \end{align}
\end{proposition}

The following lemma extends the ideas of \cite[Lemma 2]{SGD_better} and \cite[Lemma 3.3]{ROBERTS_SGD} to the case that a biased inexact stochastic gradient (4) is employed. It plays a key role in proving the convergence of Algorithm~1. It plays a key role in proving the convergence of Algorithm~1.

\begin{lemma}\label{SGD_IBES_descent_lemma}
     Suppose Assumptions 1 and 2 hold, and $z^k$ denotes the approximate hypergradient defined in (3).  Consider a non-increasing positive sequence of step sizes $\{\alpha_k\}_{k=0}^{\infty}$ with $\alpha_0 \leq \frac{\rho (2\zeta -1)}{\zeta L_{\nabla f}\tilde{B}}$,  where $0<\rho <1$ and $\zeta>\frac{1}{2}$. Defining $\bar{\rho} \coloneqq \frac{(1-\rho)(2\zeta-1)}{\zeta}>0$ and $\mu_w \coloneqq \gamma_1 (1 + \alpha_0^2 L_{\nabla f}\tilde{A})$, we have
    \begin{multline*}
    \frac{\bar{\rho}}{2} \sum_{k=0}^{K-1} w_k r_k + \frac{w_{K-1}}{\alpha_{K-1}} \gamma_{K} \leq \frac{\mu_w}{\alpha_0} \\ + \frac{(1+\eta)L_{\nabla f}}{2} \left(\frac{C}{\eta} \sum_{k=0}^{K-1} w_k \alpha_k+ \sum_{k=0}^{K-1} \epsilon_k^2 w_k \alpha_k \right) +  \frac{\zeta}{2}\sum_{k=0}^{K-1} \epsilon_k^2 w_k,
    \end{multline*}
where $r_k \coloneqq \E\left[\|\nabla f(\theta^k)\|^2\right]$, $w_0 \coloneqq 1$ and $w_k \coloneqq \frac{w_{k-1} \alpha_k}{\alpha_{k-1}(1+L_{\nabla f}\alpha_k^2 \tilde{A})}$ for $k\geq 1$, and $\gamma_k \coloneqq \E[f(\theta^k)] - f^{\text{inf}}$.
\end{lemma}
\begin{proof}
    Since $f$ is $L_{\nabla f}$-smooth, for all $k=0,1,\ldots$ we have
    \begin{equation*}
    \begin{split}
        f({\theta}^{k+1}) &\leq  f(\theta^k) + \langle \nabla f(\theta^k), {\theta}^{k+1}- \theta^k \rangle+ \frac{L_{\nabla f}}{2} \|\theta^{k+1}-\theta^k\|^2 \\ &= f(\theta^k) +  \alpha_k \langle \nabla f(\theta^k), - z^k \rangle + \frac{\alpha_k^2 L_{\nabla f}}{2} \|z^k\|^2.
        \end{split}
    \end{equation*}
Taking conditional expectation over $\theta^k$ and utilizing \cref{ES_IBES} for bounding the terms $\langle \nabla f(\theta^k),  \E_{v^k\sim \mathcal{D}}[z^k] \rangle$ and $\E_{v^k\sim \mathcal{D}} \left [\|z^k\|^2\right ]$, we have
\begin{align*}
    \E[f(\theta^{k+1}) | \theta^k] &\leq f(\theta^k) - \alpha_k \langle \nabla f(\theta^k),  \E_{v^k\sim \mathcal{D}}[z^k] \rangle \\
    &\quad + \frac{\alpha_k^2 L_{\nabla f}}{2} \E_{v^k\sim \mathcal{D}} \left [\|z^k\|^2\right ]\\
    &\hspace{-3.2em}\leq f(\theta^k) - \frac{2\zeta-1}{2\zeta} \alpha_k \|\nabla f(\theta^k)\|^2 + \frac{\zeta E^2_k}{2} \alpha_k  \\&\hspace{-3.2em}\quad+ \frac{\alpha_k^2 L_{\nabla f}}{2} (2\tilde{A} ( f(\theta^k) - f^{\text{inf}} )+ \tilde{B} \|\nabla f(\theta^k)\|^2 + \tilde{C} ),
\end{align*}
where $E^2_k \coloneqq \E_{v^k\sim \mathcal{D}} [e_{v^k}(\theta^k)^2]$.
After simplification, we get
\begin{multline*}
    \E[f(\theta^{k+1}) | \theta^k] \leq f(\theta^k)   + \alpha_k \left(\frac{\alpha_k L_{\nabla f} \tilde{B}}{2} -\frac{2\zeta-1}{2\zeta}\right) \|\nabla f(\theta^k)\|^2 \\+ \alpha_k^2L_{\nabla f} \tilde{A}(f(\theta^k) - f^{\text{inf}} ) \\ + \frac{\alpha_k^2 L_{\nabla f}}{2}\left(\frac{1+\eta}{\eta} C + (1+\eta)E^2_k\right) +  \frac{\zeta E^2_k}{2} \alpha_k.
\end{multline*}
Subtracting $f^{\text{inf}}$ from both sides, then using the tower property and rearranging, we get
\begin{multline*}
    \E[f(\theta^{k+1})] - f^{\text{inf}} + \alpha_k\left(\frac{2\zeta-1}{2\zeta} - \frac{\alpha_k L_{\nabla f} \tilde{B}}{2}\right) \E\left [\|\nabla f(\theta^k)\|^2\right]  \\ \leq (1+\alpha_k^2L_{\nabla f} \tilde{A}) \E[f(\theta^k) - f^{\text{inf}}] \\ + \frac{\alpha_k^2 L_{\nabla f}}{2}\left(\frac{1+\eta}{\eta} C + (1+\eta)\E[E^2_k]\right) + \frac{\zeta \,\E[E^2_k]}{2}\alpha_k.
\end{multline*}
Assumption 2 of the main text and the definition of the step size $\alpha_k$ yield
\begin{multline*}
 \frac{\bar{\rho}}{2} \alpha_k r_k \leq \gamma_k (1+\alpha_k^2L_{\nabla f} \tilde{A}) - \gamma_{k+1} \\ +\frac{\alpha_k^2 L_{\nabla f}}{2}\left(\frac{1+\eta}{\eta} C + (1+\eta)\epsilon^2_k\right) + \frac{\zeta \epsilon_k^2}{2} \alpha_k.
\end{multline*}
Multiplying both sides by $\frac{w_k}{\alpha_k}$, we get
\begin{multline*}
\frac{\bar{\rho}}{2}w_k r_k \leq \frac{w_k (1+\alpha_k^2L_{\nabla f} \tilde{A})}{\alpha_k}\gamma_k  - \frac{w_k}{\alpha_k}\gamma_{k+1} \\ + \frac{\alpha_k L_{\nabla f}}{2}\Big(\frac{1+\eta}{\eta} C + (1+\eta)\epsilon^2_k\Big) w_k +  \frac{\zeta \epsilon_k^2}{2} w_k,
\end{multline*}
for all $k=0,1,\ldots$.
By definition of $w_k$, we have $\frac{w_k (1+\alpha_k^2L_{\nabla f} \tilde{A})}{\alpha_k} = \frac{w_{k-1}}{\alpha_{k-1}}$ for all $k\geq 1$.
So, summing both sides for $k = 0,1, \dots,K-1 $, we have
\begin{multline*}
\frac{\bar{\rho}}{2} \sum_{k=0}^{K-1} w_k r_k \leq \frac{\mu_w}{\alpha_0} - \frac{w_{K-1}}{\alpha_{K-1}} \gamma_{K} + \frac{(1+\eta)L_{\nabla f}}{2\eta} C \sum_{k=0}^{K-1} w_k \alpha_k \\ + \frac{(1+\eta)L_{\nabla f}}{2} \sum_{k=0}^{K-1} \epsilon_k^2 w_k \alpha_k + \sum_{k=0}^{K-1}  \frac{\zeta \epsilon_k^2}{2}w_k.
\end{multline*}
Rearranging yields the required result.
\end{proof}

\begin{theorem}\label{bound_f_expectation}
    Let Assumptions 1 and 2 hold and the initial step size $\alpha_0$ satisfy the bound in \cref{SGD_IBES_descent_lemma}. Moreover, let $\mu_w$, $r_k$, $\gamma_k$, and $w_k$ be as defined in \cref{SGD_IBES_descent_lemma}. We have
\begin{multline*}
\min_{0\leq k\leq K-1} \E\left[\|\nabla f(\theta^k)\|^2\right] \leq \frac{1}{\bar{\rho}Kw_K} \Bigg( \frac{2\mu_w}{\alpha_0} \\ +  {(1+\eta)L_{\nabla f}} \left(\frac{C}{\eta} \sum_{k=0}^{K-1} w_k \alpha_k+ \sum_{k=0}^{K-1} \epsilon_k^2 w_k \alpha_k \right) + \sum_{k=0}^{K-1} {\zeta}\epsilon_k^2 w_k\Bigg).
\end{multline*}
\end{theorem}
\begin{proof}
    Starting from the descent lemma above, we have
    \begin{multline*}
        \frac{\bar{\rho}}{2} \sum_{k=0}^{K-1} w_k r_k \leq \frac{\bar{\rho}}{2} \sum_{k=0}^{K-1} w_k r_k + \frac{w_{K-1}}{\alpha_{K-1}} \gamma_{K} \\ \leq \frac{\mu_w}{\alpha_0} + \frac{(1+\eta)L_{\nabla f}}{2} \left(\frac{C}{\eta} \sum_{k=0}^{K-1} w_k \alpha_k+ \sum_{k=0}^{K-1} \epsilon_k^2 w_k \alpha_k \right) \\ + \frac{\zeta}{2}\sum_{k=0}^{K-1}  \epsilon_k^2 w_k .
    \end{multline*}
Dividing both sides by $W_K \coloneqq \sum_{k = 0}^{K-1} w_k$ yields
\begin{multline}\label{supp:temp1}
    \frac{\bar{\rho}}{2W_K} \sum_{k=0}^{K-1} w_k r_k \leq \frac{\mu_w}{\alpha_0 W_K} - \frac{w_{K-1}}{\alpha_{K-1} W_K} \gamma_{K} \\+ \frac{(1+\eta)L_{\nabla f}}{2W_K} \left(\frac{C}{\eta} \sum_{k=0}^{K-1} w_k \alpha_k+ \sum_{k=0}^{K-1} \epsilon_k^2 w_k \alpha_k \right) \\ + \frac{\zeta}{2 W_K}\sum_{k=0}^{K-1}  w_k \epsilon_k^2.
\end{multline}
In order to simplify, since $W_K, \gamma_{K}$, $\alpha_{K-1}$ and all $w_k$ are non-negative, we know
\begin{multline}\label{supp:temp2}
    \frac{\bar{\rho}}{2} \left (\min_{0\leq k\leq K-1} r_k\right ) = \frac{\bar{\rho}}{2W_K}\left( \min_{0\leq k\leq K-1} r_k \right) \sum_{k=0}^{K-1} w_k \\ \leq  \frac{\bar{\rho}}{2W_K}\sum_{k=0}^{K-1} w_k r_k + \frac{w_{K-1}}{\alpha_{K-1} W_K} \gamma_{K}.
\end{multline}
Utilizing \eqref{supp:temp2} in \eqref{supp:temp1} yields
\begin{multline} \label{supp:temp3}
    \frac{\bar{\rho}}{2} \left (\min_{0\leq k\leq K-1} r_k \right ) \leq \frac{\mu_w}{\alpha_0 W_K} \\ +  \frac{(1+\eta)L_{\nabla f}}{2W_K} \left(\frac{C}{\eta} \sum_{k=0}^{K-1} w_k \alpha_k+ \sum_{k=0}^{K-1} \epsilon_k^2 w_k \alpha_k \right) \\ + \frac{\zeta}{2W_K}\sum_{k=0}^{K-1}  w_k \epsilon_k^2.
\end{multline}
Since $\{\alpha_k\}_{k=0}^{\infty}$ is non-increasing, we have $\frac{\alpha_k}{\alpha_{k-1}}\leq 1$. Moreover, $\frac{1}{1+L_{\nabla f}\tilde{A}\alpha_k^2}\leq 1$. Hence,
$$w_k = w_{k-1} \left (\frac{\alpha_k}{\alpha_{k-1}}\right ) \left (\frac{1}{1+L_{\nabla f}\tilde{A}\alpha_k^2} \right ) \leq w_{k-1},$$
for all $k\geq 1$.
Utilizing the inequality above together with the definition of $W_K$, we find
\begin{equation}\label{supp:temp4}
    \frac{1}{W_K} = \frac{1}{\sum_{k=0}^{K-1} w_k} \leq \frac{1}{\sum_{k=0}^{K-1} \min_{0\leq k \leq K-1} w_k} = \frac{1}{Kw_{K}}
\end{equation}
Using \eqref{supp:temp4} in \eqref{supp:temp3} and multiplying by $\frac{2}{\bar{\rho}}$ gives us the desired result.
\end{proof}

\subsection{Proof of Theorem~1}
\begin{proof}
    Starting from \cref{bound_f_expectation} we can write
    \begin{multline}\label{bound_r}
        \min_{0\leq k\leq K-1} r_k \leq \frac{1}{\bar{\rho}} \Bigg( \frac{2\mu_w}{K\alpha_0 w_K} \\ +  \frac{(1+\eta)L_{\nabla f}}{Kw_K} \left(\frac{C}{\eta} \sum_{k=0}^{K-1} w_k \alpha_k+ \sum_{k=0}^{K-1} \epsilon_k^2 w_k \alpha_k \right) \\ + \frac{\zeta}{Kw_K}\sum_{k=0}^{K-1} \epsilon_k^2 w_k\Bigg).
    \end{multline}
By definition of the weighting sequence we have
\begin{equation}\label{w_k}
    w_k =\frac{\alpha_k}{\alpha_0}\prod_{j=1}^k \frac{1}{1+L_{\nabla f}\tilde{A}\alpha_j^2},
\end{equation}
with the convention $\prod_{j=1}^{0} c_j \coloneqq 1$.
Recall that, for a sequence of positive real numbers $\{y_j\}_{j \in \mathbb{N}}$, $\prod_{k=0}^\infty y_k$ converges to a finite value if and only if $\sum_{k=0}^\infty \log (y_k)$ converges to a finite value. Let $Q \coloneqq \sum_{k=0}^\infty \alpha_k^2$, and so $Q <\infty$ by Assumption~3(i). For $y>-1$, we know $\log (1+y)\leq y$, therefore
$$\sum_{k=1}^\infty \log (1+L_{\nabla f}\tilde{A}\alpha_k^2) \leq \sum_{k=1}^\infty L_{\nabla f}\tilde{A}\alpha_k^2 \leq L_{\nabla f}\tilde{A} Q.$$
It implies that $P \coloneqq \prod_{k = 0}^\infty (1+L_{\nabla f}\tilde{A}\alpha_k^2) <\infty$. Using the definition of $P$ and \eqref{w_k}, we have
\begin{equation}\label{lim_Kw_k}
    \frac{1}{K w_K} = \frac{\alpha_0}{K \alpha_K} \prod_{j=1}^{K} (1+L_{\nabla f}\tilde{A} \alpha_j^2) \leq \frac{\alpha_0 P}{K\alpha_K}.
\end{equation}
Moreover, we know
\begin{equation}\label{bound_Q}
   \sum_{k = 0}^{K-1} w_k\alpha_k \leq \sum_{k = 0}^\infty \frac{\alpha_k^2}{\alpha_0}\prod_{j = 1}^k \frac{1}{1 + L_{\nabla f}\tilde{A}\alpha_j^2} \leq \frac{1}{\alpha_0} \sum_{k = 0}^\infty \alpha_k^2 = \frac{Q}{\alpha_0} < \infty.
\end{equation}
For the term $\sum_{k=0}^{K-1} \epsilon_k^2w_k$, from \eqref{w_k} one can write
\begin{equation}\label{eps_w_k}
    \sum_{k=0}^{K-1} \epsilon_k^2w_k = \frac{1}{\alpha_0}\sum_{k=0}^{K-1}  {\alpha_k}\epsilon_k^2\prod_{j=1}^k \frac{1}{1+L_{\nabla f}\tilde{A}\alpha_j^2} \leq \frac{1}{\alpha_0} \sum_{k=0}^{K-1}  {\alpha_k}\epsilon_k^2.
\end{equation}
Lastly, for the term $\sum_{k=0}^{K-1} \epsilon_k^2 w_k \alpha_k$, since $\alpha_k \leq \alpha_0$, we find
\begin{equation}\label{eps_w_k_alpha}
    \sum_{k=0}^{K-1} \epsilon_k^2 w_k\alpha_k \leq \sum_{k=0}^{K-1}  {\alpha_k}\epsilon_k^2\prod_{j=1}^k \frac{1}{1+L_{\nabla f}\tilde{A}\alpha_j^2} \leq \sum_{k=0}^{K-1}  {\alpha_k}\epsilon_k^2.
\end{equation}
Utilizing \eqref{bound_Q}, \eqref{eps_w_k}, and \eqref{eps_w_k_alpha} on the right-hand side of \eqref{bound_r}, we find
\begin{multline}\label{final_bound}
    \min_{0\leq k\leq K-1} r_k \leq \frac{1}{\bar{\rho}} \Bigg( \frac{2 \gamma_1(1+\alpha_0^2 L_{\nabla f}\tilde{A}) P}{K \alpha_K} \\ +  \frac{(1+\eta)L_{\nabla f} \alpha_0 P}{K\alpha_K} \left(\frac{C Q}{\eta \alpha_0} + \sum_{k=0}^{K-1} \alpha_k \epsilon_k^2 \right) + \frac{\zeta P}{K \alpha_K}\sum_{k=0}^{K-1} \alpha_k \epsilon_k^2\Bigg),
\end{multline}
which gives (8).
Using Assumption~3(ii) and Jensen's inequality $$\E[\|\nabla f(\theta^k)\|]^2 \leq \E[\|\nabla f(\theta^k)\|^2] = r_k,$$ we conclude the second result.
\end{proof}

\section{Rates for polynomial and logarithmic schedules}
\label{supp:rates}

The following propositions bound the quantity $L_K$ of Theorem~1 and yield Corollary~1.

\begin{proposition}\label{prop:poly_poly_rate}
Let $\alpha_k=\mathcal{O}(k^{-q})$ with $\tfrac12<q<1$ and
$\epsilon_k=\mathcal{O}(k^{-p})$ for some $p>0$.  Then, for $K$ large enough, considering $L_K$ as defined in Theorem~1, we have
$$
L_K =
\begin{cases}
\mathcal{O}\big(K^{-\min\{2p,\,1-q\}}\big), & q+2p\neq 1,\\
\mathcal{O}\big(\tfrac{\log K}{K^{1-q}}\big), & q+2p=1.
\end{cases}
$$
In particular, $\lim_{K\to\infty} L_K = 0$ for every $p>0$ and $\tfrac12<q<1$.
\end{proposition}
\begin{proof}
By the assumptions, there exist constants $C>0$ and $K_0\in\mathbb{N}$ such that for all $k\ge K_0$
$$
\alpha_k\epsilon_k^2 \le C\,k^{-(q+2p)} .
$$
Hence for $K\ge K_0$,
$$
\sum_{k=K_0}^{K-1} \alpha_k\epsilon_k^2 \le C\sum_{k=K_0}^{K-1} k^{-(q+2p)}.
$$
Apply the integral-comparison bounds
$$
\int_{K_0}^{K-1} t^{-(q+2p)}\,dt \le \sum_{k=K_0}^{K-1} k^{-(q+2p)} \le 1 + \int_{K_0}^{K-1} t^{-(q+2p)}\,dt,
$$
and consider three cases.

\medskip\noindent\textbf{Case 1: $q+2p>1$.} The integral (and series) converges, so
$$
\sum_{k=K_0}^{K-1} \alpha_k\epsilon_k^2 = O(1).
$$
Using the assumption $\alpha_K=\mathcal O(K^{-q})$ we have $K\alpha_K=\mathcal O(K^{1-q})$, hence
$$
L_K=\frac{1}{K\alpha_K}\left(\sum_{k=0}^{K_0-1} \alpha_k \epsilon_k^2 + \sum_{k=K_0}^{K-1} \alpha_k\epsilon_k^2\right) = O\!\big(K^{-(1-q)}\big),
$$
and therefore $\lim_{K\to\infty}L_K=0$.

\medskip\noindent\textbf{Case 2: $q+2p=1$.} Then
$$
\sum_{k=K_0}^{K-1} \alpha_k\epsilon_k^2 =\sum_{k=K_0}^{K-1} k^{-(q+2p)} = O(\log K),
$$
and therefore
$$
L_K = O\!\left(\frac{\log K}{K\alpha_K}\right) = O\!\left(\frac{\log K}{K^{1-q}}\right).
$$
Since $1-q>0$ we conclude $\lim_{K\to\infty}L_K=0$.

\medskip\noindent\textbf{Case 3: $q+2p<1$.} The integral behaves like $K^{1-(q+2p)}$, so
$$
\sum_{k=K_0}^{K-1} \alpha_k\epsilon_k^2 =\sum_{k=K_0}^{K-1} k^{-(q+2p)} = O\!\big(K^{1-(q+2p)}\big).
$$
Dividing by $K\alpha_K = O(K^{1-q})$ gives
$$
L_K = O\!\left(\frac{K^{1-(q+2p)}}{K^{1-q}}\right) = O(K^{-2p}),
$$
and thus $\lim_{K\to\infty}L_K=0$.

\medskip
Collecting the three cases combined with the fact that $\tfrac{1}{2}<q<1$ yields the claimed piecewise estimates.
\end{proof}

\begin{proposition}\label{prop:poly_log_rate}
Let $\alpha_k=\mathcal{O}(k^{-q})$ with $\tfrac12<q<1$ and
$\epsilon_k=(\log k)^{-p}$ for $k\ge 2$ and some $p>0$. Then, for $K$ large enough,
$$
L_K \;=\; \mathcal{O}\!\big((\log K)^{-2p}\big).
$$
Consequently $\lim_{K\to\infty} L_K = 0$ for every $p>0$.
\end{proposition}
\begin{proof}
Let $K_0\ge 2$ be such that for all $k\ge K_0$ we have
$$
\alpha_k\epsilon_k^2 \le C_0\,k^{-q}(\log k)^{-2p}
$$
for some constant $C_0>0$. It suffices to bound
$$
S_K \coloneqq \sum_{k=K_0}^{K-1} k^{-q}(\log k)^{-2p},
$$
since $\sum_{k=2}^{K-1} \alpha_k\epsilon_k^2 \le C_0 S_K$ for $K\ge K_0$.

By the integral comparison test,
$$
\int_{K_0}^{K-1} t^{-q}(\log t)^{-2p}\,dt \le S_K \le 1 + \int_{K_0}^{K-1} t^{-q}(\log t)^{-2p}\,dt.
$$
Hence it is enough to show
$$
I(K)\coloneqq\int_{K_0}^{K-1} t^{-q}(\log t)^{-2p}\,dt = O\!\big(K^{1-q}(\log K)^{-2p}\big).
$$

Perform the substitution $u=\log t$ (so $t=e^u$, $dt=e^u du$). Then
$$
I(K)=\int_{\log K_0}^{\log (K-1)} e^{(1-q)u} u^{-2p}\,du.
$$
Set $V\coloneqq(1-q)\log (K-1)$ and $a\coloneqq(1-q)\log K_0$ and change variable $v=(1-q)u$ (so $dv=(1-q)du$). This yields
$$
I(K) = (1-q)^{-(1-2p)}\int_{a}^{V} e^{v} v^{-2p}\,dv.
$$
Write
$$
J(V)\coloneqq\int_a^V e^{v} v^{-2p}\,dv.
$$

Integrate by parts once with $g(v)=v^{-2p}$ and $h'(v)=e^{v}$ to obtain
$$
J(V) = e^{V}V^{-2p} - e^{a}a^{-2p} + 2p\int_a^V e^{v} v^{-2p-1}\,dv.
$$
For $V$ large enough we have $v^{-2p-1}\le V^{-1} v^{-2p}$ for all $v\in[a,V]$. Hence
$$
\int_a^V e^{v} v^{-2p-1}\,dv \le V^{-1}\int_a^V e^{v} v^{-2p}\,dv = V^{-1} J(V).
$$
Combining with the integration-by-parts identity yields, for large \(V\),
$$
J(V) \le e^{V}V^{-2p} + C_a + 2p V^{-1} J(V),
$$
where \(C_a \coloneqq e^{a}a^{-2p}\) is a fixed constant. Rearranging gives
$$
\big(1 - 2p V^{-1}\big) J(V) \le e^{V}V^{-2p} + C_a.
$$
Choose \(V_0\) sufficiently large so that \(1 - 2p V^{-1} \ge \tfrac{1}{2}\) for all \(V\ge V_0\). For such \(V\),
$$
J(V) \le 2\big(e^{V}V^{-2p} + C_a\big) \le C_1 e^{V}V^{-2p}
$$
for some constant \(C_1>0\). Tracing back the substitutions, we obtain for large \(K\)
$$
I(K) = (1-q)^{-(1-2p)} J(V) \le C_2 K^{1-q}(\log (K-1))^{-2p}
$$
for some constant \(C_2>0\). Therefore
$$
S_K = O\!\big(K^{1-q}(\log K)^{-2p}\big).
$$

Finally, since $\alpha_K=\mathcal{O}(K^{-q})$ we have $K\alpha_K=\mathcal{O}(K^{1-q})$, and hence
$$
L_K = \frac{1}{K\alpha_K}\Big(\alpha_0 \epsilon_0^2 + \alpha_1 \epsilon_1^2 + \sum_{k=2}^K \alpha_k\epsilon_k^2\Big)
= O\!\big((\log K)^{-2p}\big).
$$
As $p>0$ this implies $\lim_{K\to\infty} L_K = 0$, completing the proof.
\end{proof}

\paragraph{Summary of rates}
A comparison of convergence rates under different accuracy (polynomial and logarithmic) and step size (polynomial) schedules is presented in \cref{rate_table}.
\begin{table}[h!]
\centering
\caption{Convergence rates for different step size and accuracy decay regimes.}
\label{rate_table}
\resizebox{\columnwidth}{!}{%
\begin{tabular}{lll l}
\toprule
\textbf{Step size $\alpha_k$} & \textbf{Accuracy $\epsilon_k$} & \textbf{Gradient rate} & \textbf{Note} \\
\midrule
\multirow{4}{*}{$\mathcal{O}(k^{-q}),\; \tfrac{1}{2}<q<1$}
 & $\mathcal{O}(k^{-p}),\; 2p<1-q$
 & $\mathcal{O}(k^{-p})$
 & limited by accuracy \\[3pt]

 & $\mathcal{O}(k^{-p}),\; 2p=1-q$
 & $\mathcal{O}\!\big(\sqrt{\tfrac{\log k}{k^{1-q}}}\big)$
 & boundary: slower with $\log k$ \\[3pt]

 & $\mathcal{O}(k^{-p}),\; 2p>1-q$
 & $\mathcal{O}(k^{-(1-q)/2})$
 & limited by step size \\[3pt]

 & $\mathcal{O}((\log k)^{-p})$
 & $\mathcal{O}((\log k)^{-p})$
 & any $p>0$; slowest \\
\bottomrule
\end{tabular}%
}
\end{table}

\section{Proofs of Corollaries~2 and~3}
\label{supp:corollaries}

\subsection{Proof of Corollary~2}
\begin{proof}
By Corollary~1 with $p = 1/4$ and $q = 1/2 + \delta$, the expected gradient norm decays as $\mathcal{O}(K^{-(1-q)/2}) = \mathcal{O}(K^{-1/4 + \delta/2})$, so $K(\epsilon) = \mathcal{O}(\epsilon^{-4/(1-2\delta)})$. At iteration $k$, achieving $\delta_k = \Theta(k^{-1/4})$ from a cold start requires $\mathcal{O}(\sqrt{\kappa} \log(R/\delta_k)) = \mathcal{O}(\sqrt{\kappa} \log(k/\epsilon_0))$ solver iterations, where $R$ bounds the initial distance; the same order holds for the Krylov solve to residual $\tau_k$. Summing, $\sum_{k \le K} \log k = \mathcal{O}(K \log K)$, and substituting $K = K(\epsilon)$ gives the claim.
\end{proof}
\subsection{Proof of Corollary~3}
\begin{proof}
With constant steps the weights in \cref{SGD_IBES_descent_lemma} satisfy $w_k = (1 + L_{\nabla f}\tilde{A}\alpha^2)^{-k} \ge e^{-L_{\nabla f}\tilde{A}c^2}$ being bounded below by a constant for all $k \le K$ (since $k\alpha^2 \le c^2$), so \cref{bound_f_expectation} yields
$\min_k \E[\|\nabla f(\theta^k)\|^2] = \mathcal{O}\big(\tfrac{1}{K\alpha} + \alpha + \tfrac{1}{K}\sum_{k=1}^{K}\epsilon_k^2\big)$.
The first two terms are $\mathcal{O}(K^{-1/2})$ by the choice of $\alpha$, and $\tfrac{1}{K}\sum_{k \le K} k^{-1/2} = \mathcal{O}(K^{-1/2})$.
\end{proof}

\section{Additional experimental details and results}
\label{supp:experiments}
\subsection{Denoising with the convex ridge regularizer}

\begin{figure*}[tbhp]
\centering
\captionsetup[subfloat]{farskip=0pt,captionskip=0pt}
\subfloat[$\epsilon_0 = 1$]{\label{fig:1e0_decrease}\includegraphics[width=0.25\textwidth]{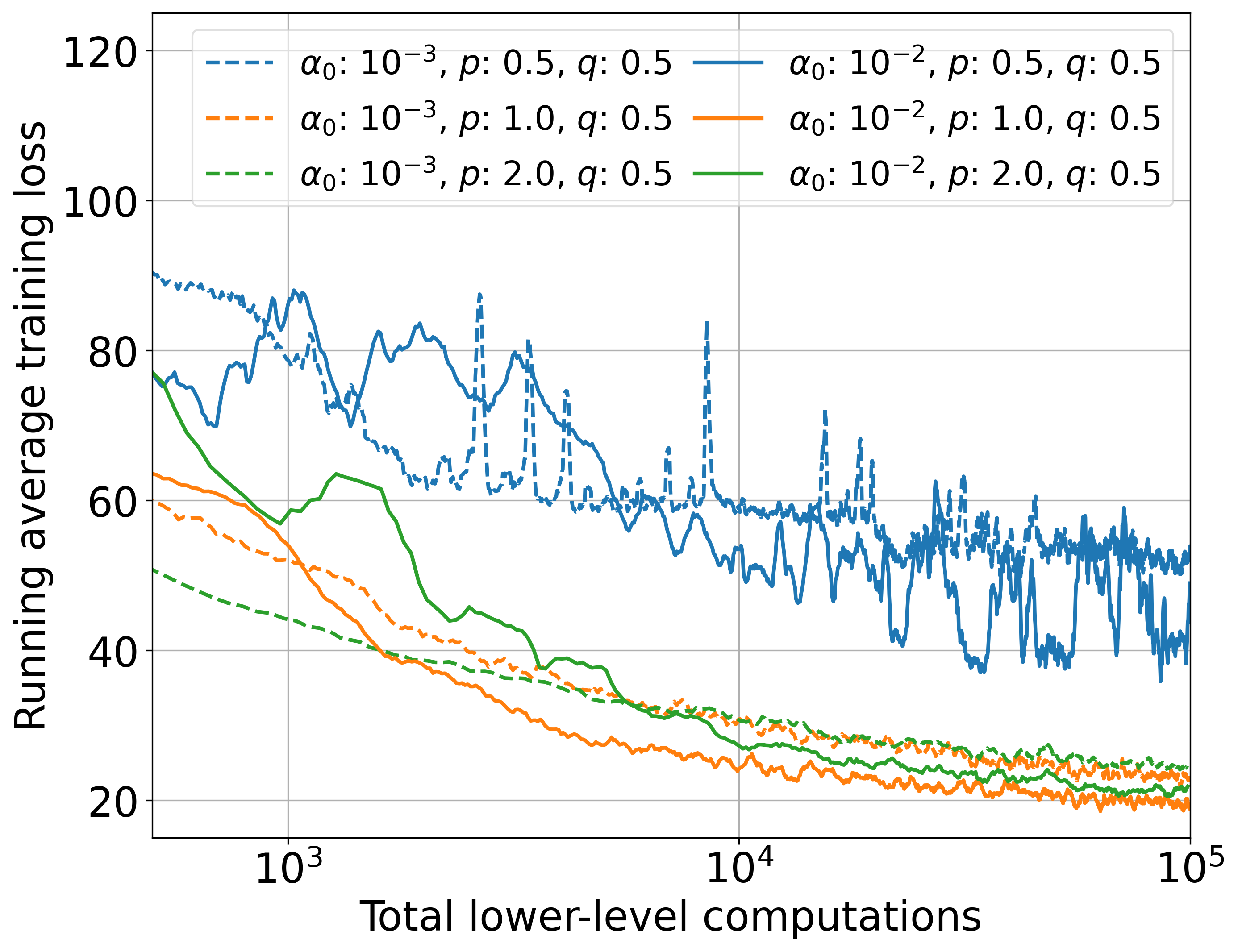}}\hspace{0.04\textwidth}
\subfloat[$\epsilon_0 = 10^{-2}$]{\label{fig:1e-2_decrease}\includegraphics[width=0.25\textwidth]{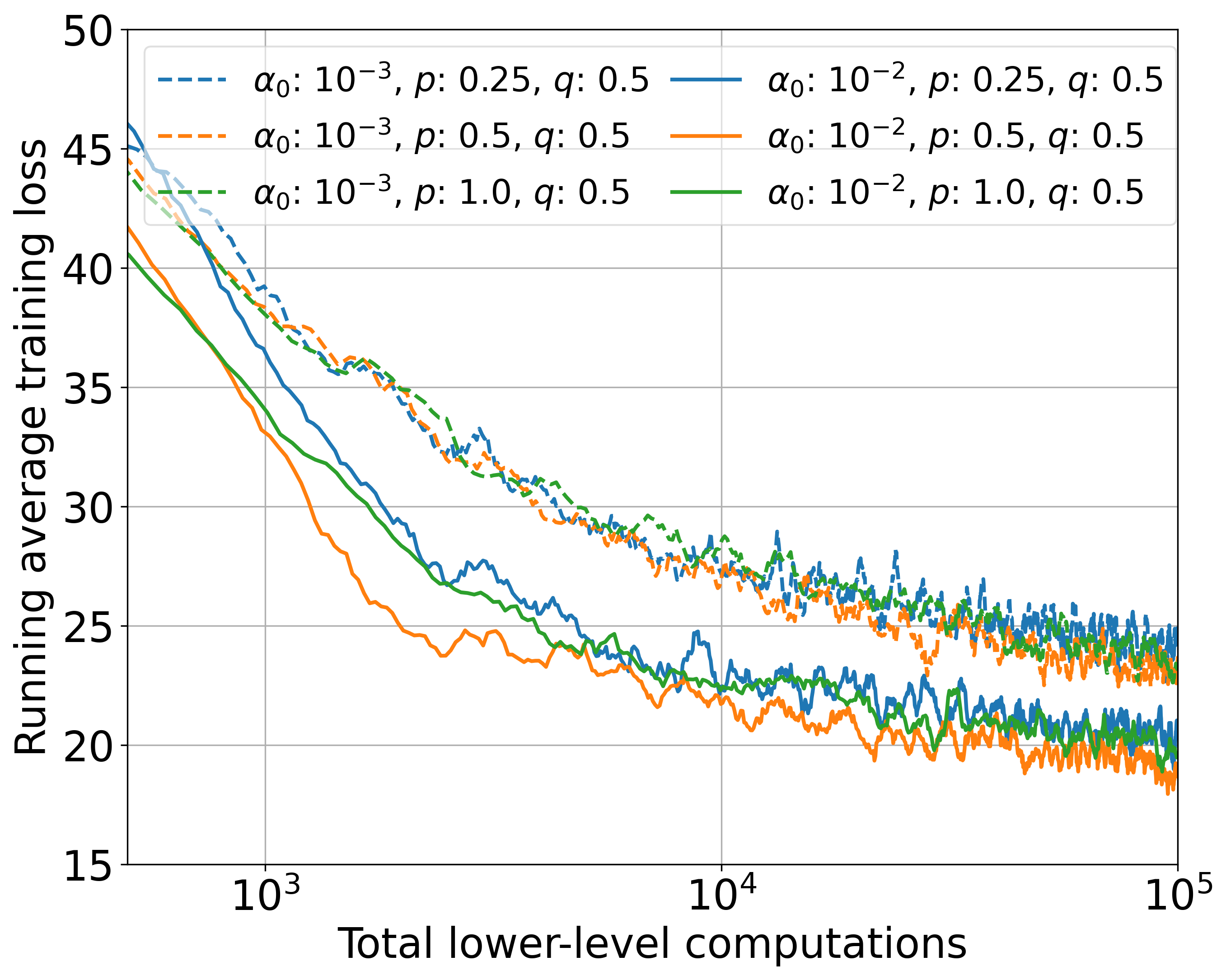}}\hspace{0.04\textwidth}
\subfloat[$\epsilon_0 = 10^{-4}$]{\label{fig:1e-4_decrease}\includegraphics[width=0.25\textwidth]{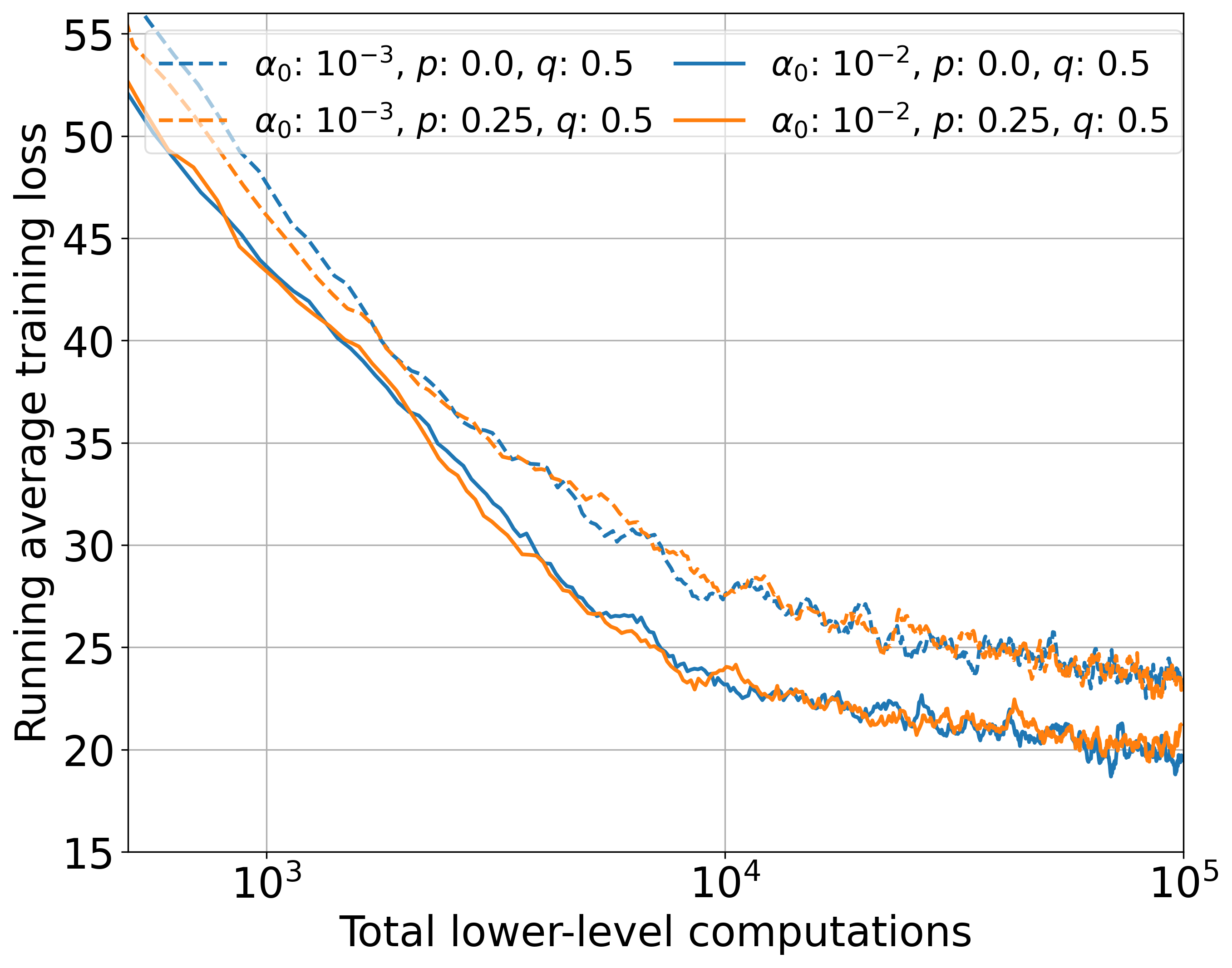}}
\caption{Denoising with CRR: training loss with a decreasing upper-level step size ($q > 0$) for the three initial accuracies of Fig.~1 in the main text. Decreasing the step size reduces the variance, but the ranking of the accuracy schedules is the same as with a fixed step size.}
\label{fig:denoise_decrease}
\end{figure*}

Each convolutional kernel is reparametrized to have zero mean and is initialized using zero mean Xavier normal initialization. To ensure scale-invariance and stability, the composite convolutional operator is normalized via an estimate of its Lipschitz constant computed either using power iterations or in the Fourier domain. The resulting operator is further constrained to unit spectral norm using spectral normalization \cite{goujon2023WCRR}.

\begin{figure}[tbhp]
\centering
\subfloat[]{%
    \label{fig:winner_denoise_train_loss}
    \includegraphics[width=0.46\columnwidth]{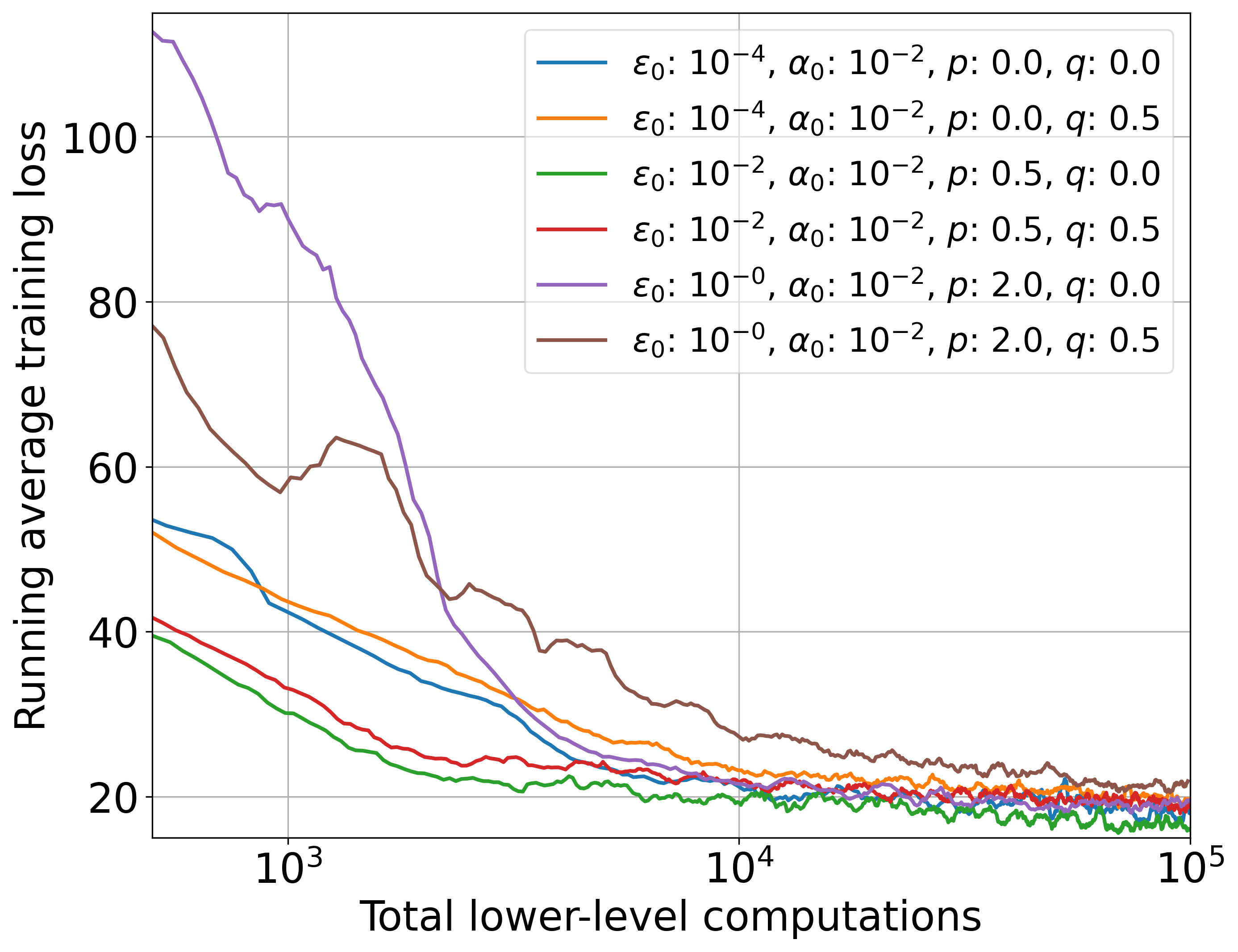}
}\hfill
\subfloat[]{%
    \label{fig:winner_denois_test_psnr}
    \includegraphics[width=0.46\columnwidth]{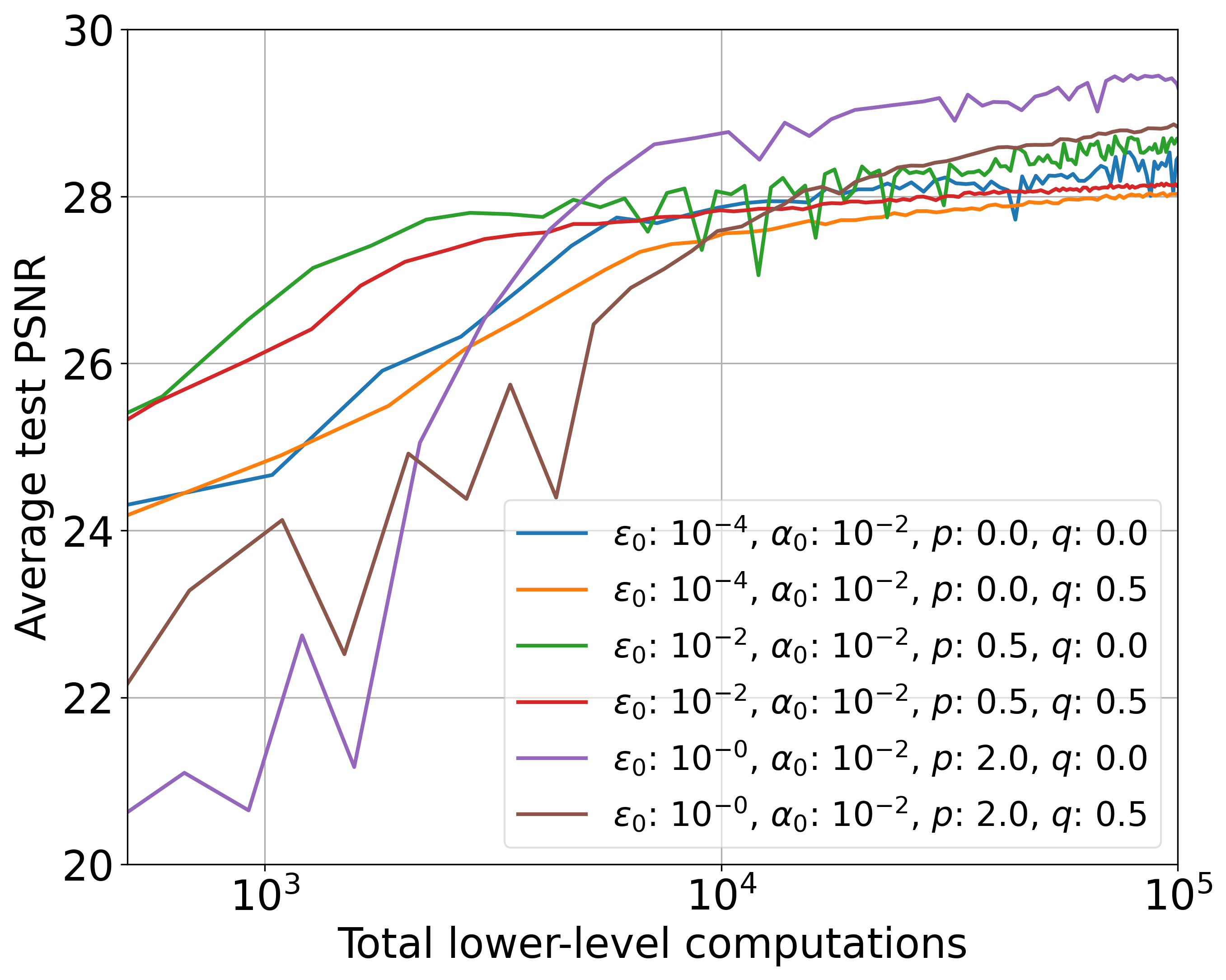}
}
\caption{Best-performing fixed and decreasing step size configurations across all accuracy settings $\epsilon_0 \in \{10^{0}, 10^{-2}, 10^{-4}\}$. (a) Training loss plotted against computational cost. (b) Test PSNR plotted against computational cost.}
\label{fig:winner_denoise}
\end{figure}

By selecting the best-performing settings and plotting them together, as shown in \cref{fig:winner_denoise}, it becomes clear that choosing an appropriate accuracy schedule based on the initial accuracy $\epsilon_0$, specifically, using a larger $p$ for larger $\epsilon_0$, can yield desirable training and test performance. The mid-accuracy case $\epsilon_0 = 10^{-2}$ with decreasing accuracy performed best initially up to around $10^4$ computations, and once finding a suitable accuracy, $\epsilon_0 = 1$ with $p = 2$ takes over, although this could not have been known a priori.

Furthermore, although our convergence theorem (Theorem~1) requires a decreasing step size with $0.5 < q < 1$, in practice a non-decreasing step size can yield better finite-budget performance, as shown in \cref{fig:winner_denoise}. This suggests room for improved non-asymptotic analysis, while decreasing step sizes reduce training variance and produce more stable test results.

The qualitative results at selected testing checkpoints corresponding to different training settings {(\cref{fig:denoise_ckp})} align well with the performance curves in \cref{fig:winner_denoise}. The configuration with $\epsilon_0 = 1$ and $p = 2$ achieves the overall best performance. While the setting with $\epsilon_0 = 10^{-2}$ and $p = 0.5$ performs best in the early stages, it ultimately ranks second, though still outperforming the high-accuracy setting $\epsilon_0 = 10^{-4}$, which could not have been anticipated a priori.

\begin{figure}[tbhp]
\centering
\newcommand{\dk}[2][0.245]{\includegraphics[width=#1\columnwidth]{Figs/Checkpoints/Denoise/tiles/#2}}
\setlength{\tabcolsep}{0.6pt}
\begin{tabular}{@{}c@{\hspace{1pt}}c@{\hspace{8pt}}c@{}}
\dk[0.24]{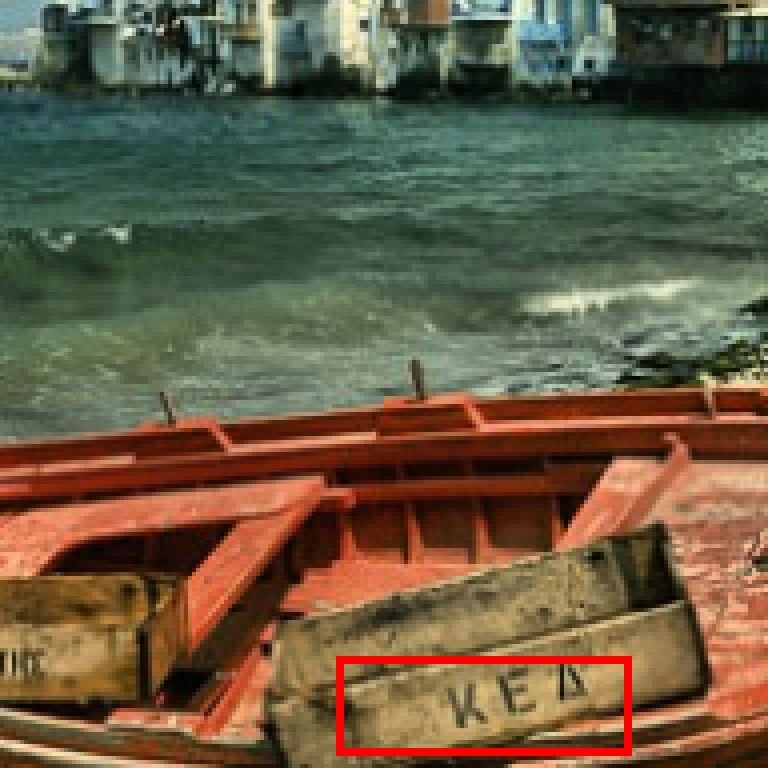} & \dk[0.24]{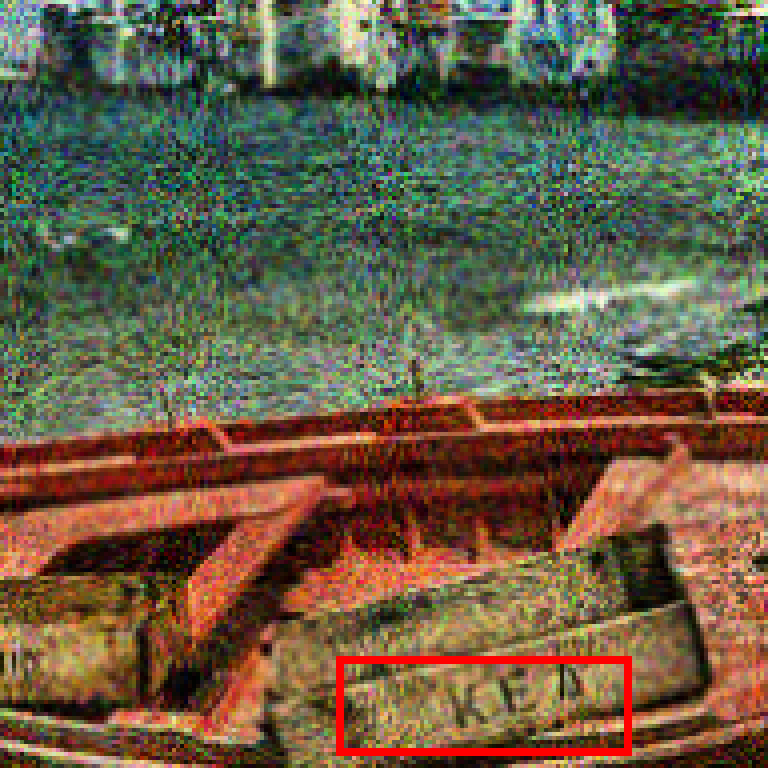} &
\begin{tabular}[b]{@{}c@{}}\dk[0.36]{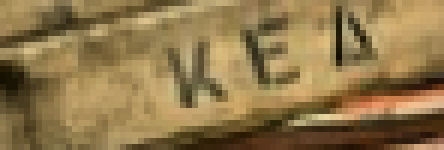}\\[1pt]\dk[0.36]{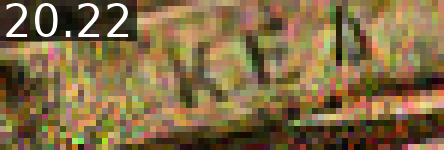}\end{tabular}\\
\scriptsize Ground truth & \scriptsize Noisy input & \scriptsize Enlarged region (ground truth, noisy)
\end{tabular}\\[5pt]
\begin{tabular}{@{}cccc@{}}
\multicolumn{4}{c}{\scriptsize $\alpha_0 = 10^{-2}$, $\epsilon_0 = 10^{-2}$, $p = 0.5$, $q = 0$ (higher initial accuracy)}\\
\dk{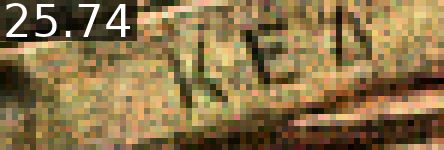} & \dk{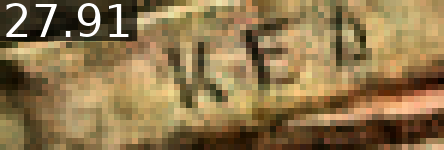} & \dk{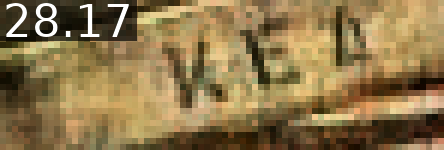} & \dk{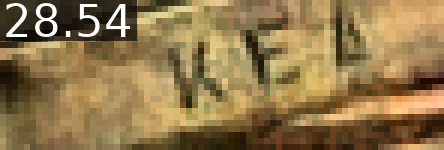}\\[2pt]
\multicolumn{4}{c}{\scriptsize $\alpha_0 = 10^{-2}$, $\epsilon_0 = 1$, $p = 2$, $q = 0$ (lower initial accuracy)}\\
\dk{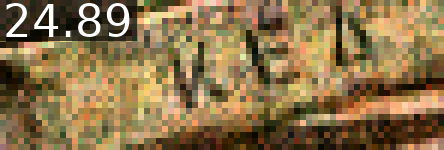} & \dk{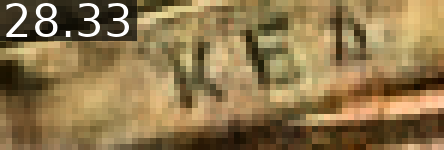} & \dk{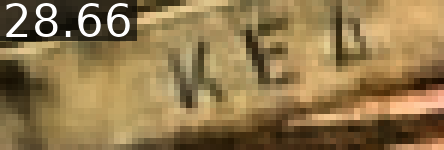} & \dk{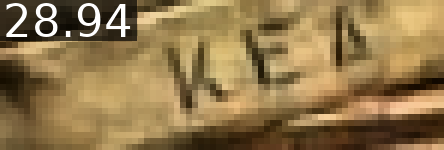}
\end{tabular}
\caption{Denoising with the CRR regularizer at training checkpoints of computational cost $5{,}000$, $10{,}000$, $50{,}000$, and $100{,}000$ (left to right) for two configurations. Top: ground truth and noisy test image with the red region, which is shown enlarged in all rows; inset numbers are test PSNR (dB).}
\label{fig:denoise_ckp}
\end{figure}

\subsection{Inpainting with the convex ridge regularizer}

Each training image is corrupted by randomly masking $70\%$ of the pixels using a Bernoulli mask
$M_{ij} \sim \mathrm{Bernoulli}(p), \ 1\leq i, j \leq 96$, with $p = 0.3$, and the observation is given by $y = M \odot x$,
which corresponds to retaining approximately $30\%$ of the original data, implemented via the \texttt{Bernoulli Splitting Mask Generator}\footnote{\href{https://deepinv.github.io/deepinv/api/stubs/deepinv.physics.generator.BernoulliSplittingMaskGenerator.html\#deepinv.physics.generator.BernoulliSplittingMaskGenerator}{BernoulliSplittingMaskGenerator documentation}}. Additionally, additive Gaussian noise with standard deviation \( \sigma = 0.02 \) is applied to the unmasked pixels. The same training and evaluation pipeline as in the denoising experiment is employed, with mini-batches and data augmentations appropriately adjusted for the new image size.

\begin{figure}[tbhp]
\centering
\subfloat[]{%
    \label{fig:inpainting_loss_1e-2}
    \includegraphics[width=0.41\columnwidth]{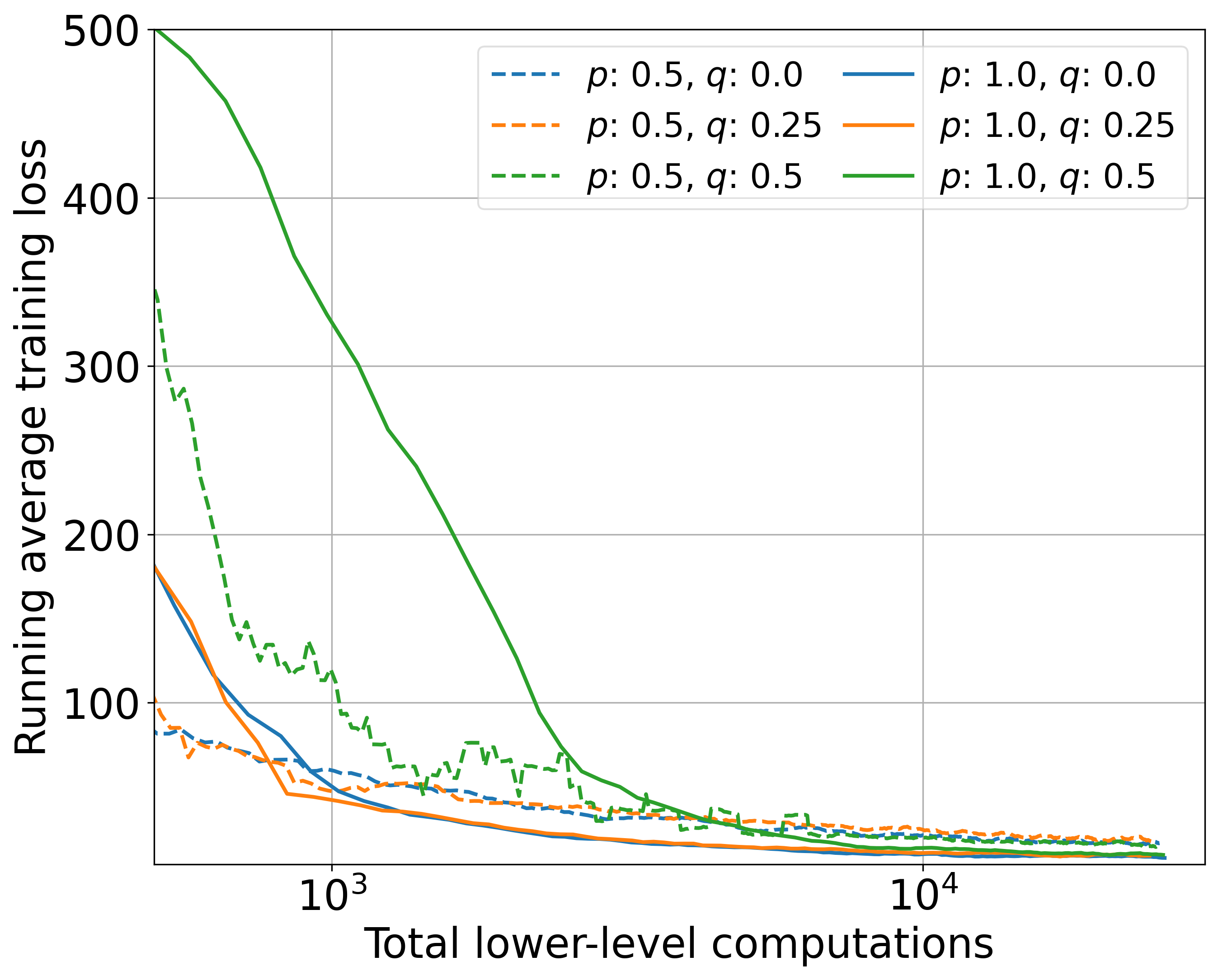}
}\hfill
\subfloat[]{%
    \label{fig:inpainting_test_1e-2}
    \includegraphics[width=0.41\columnwidth]{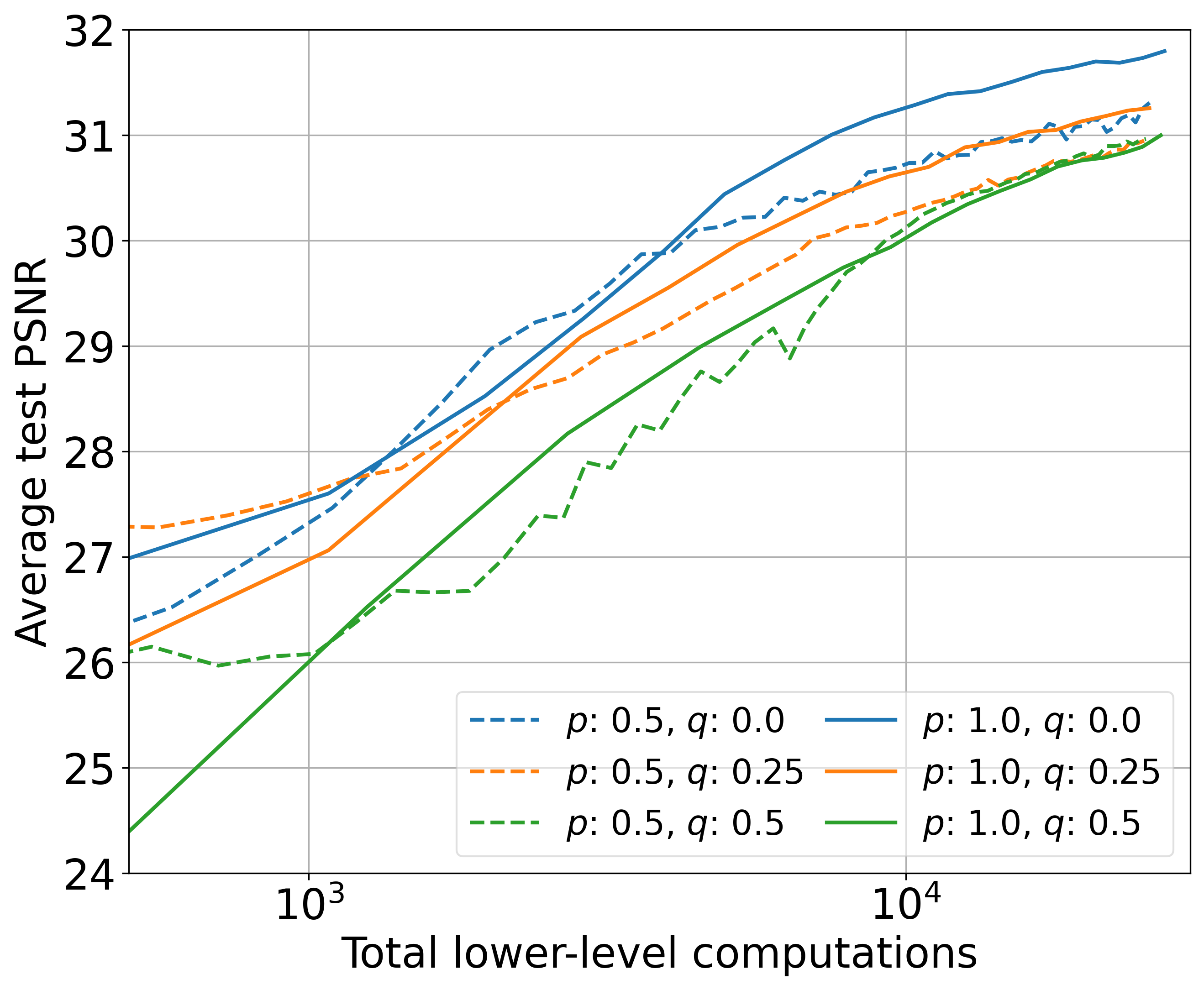}
}
\caption{Running average training loss and average test PSNR for inpainting, plotted against total computational cost. Results are shown for upper-level step size $\alpha_0 = 10^{-2}$, initial accuracy $\epsilon_0 = 10^{-1}$, accuracy schedule exponent $p \in \{0.5, 1\}$, and step size decay exponent $q \in \{0, 0.25, 0.5\}$. All settings with $p = 1$ appear to perform faster than their $p = 0.5$ counterparts.}
\label{fig:Inpainting_curves}
\end{figure}

Given the results in \cref{fig:Inpainting_curves}, when fixing $p=1$ and checkpointing the test quality after specific training computational costs, one can observe that the fixed step size consistently outperforms the decreasing variants with $q \in \{0.25, 0.5\}$ throughout training, resulting in a clear PSNR advantage under the same computational budget{, as illustrated by the qualitative checkpoint comparisons in Fig.~2 of the main text}. This indicates that, by appropriately tuning the initial step size, one can obtain a desirable learned regularizer even with a fixed step size during training.
With the more suitable choice of $p = 1$ in this experiment, the results show improved test stability for the decreasing step size variants. During training, however, $q = 0.5$ appears more aggressive than necessary under a fixed computational budget, while $q = 0.25$ performs slightly better than $q = 0$, particularly before $10^3$ computations.

\subsection{Input-convex neural network regularizer}

In the ICNN of Section~IV-C, $W_x$  is initialized with zero-mean filters to ensure improved performance.

\begin{figure}[tbhp]
\centering
\captionsetup[subfloat]{farskip=1pt,captionskip=0pt}
\subfloat[]{%
    \label{fig:ICNN_fixed} \includegraphics[width=0.41\columnwidth]{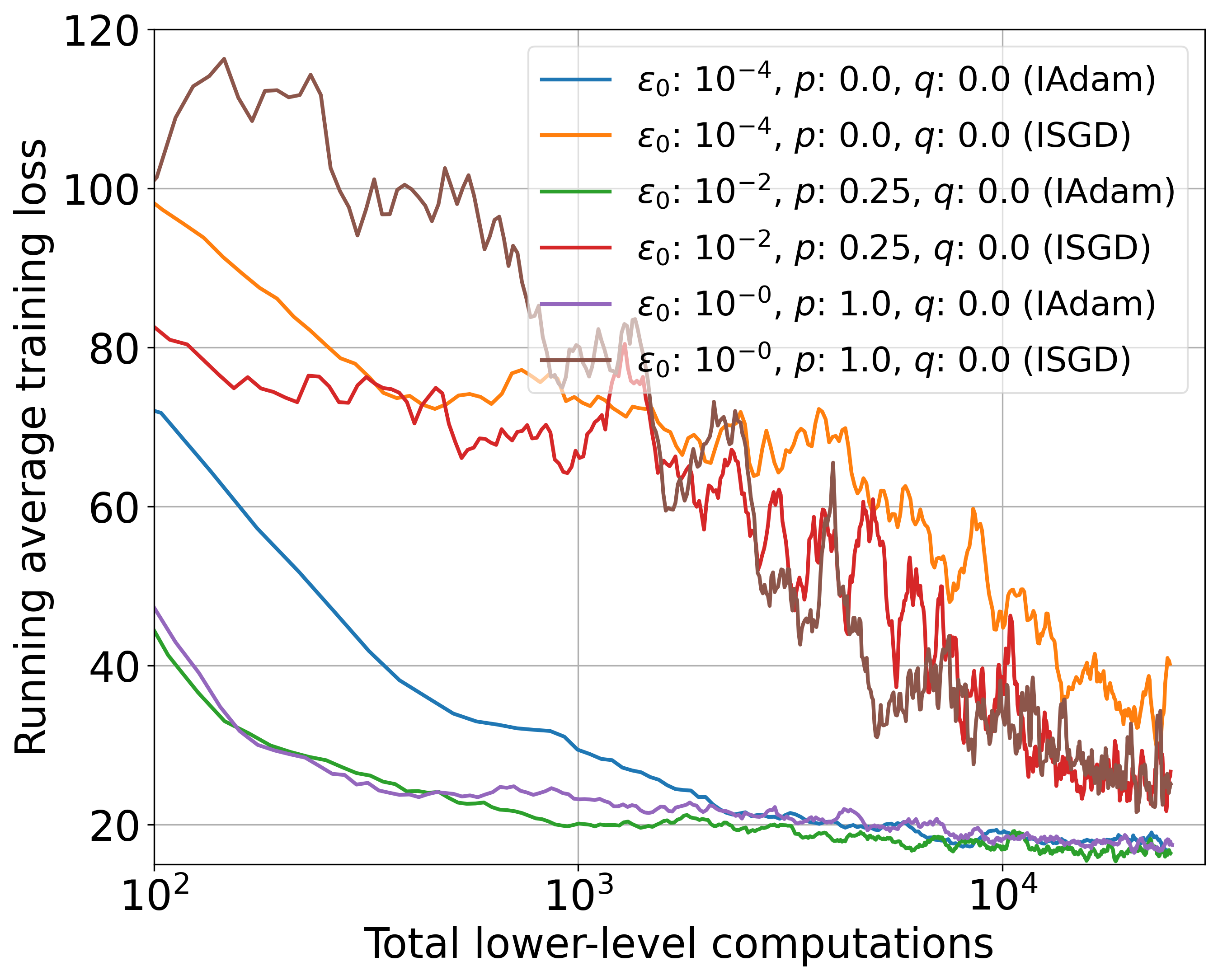}
}\hfill
\subfloat[]{%
    \label{fig:ICNN_decrease}
    \includegraphics[width=0.41\columnwidth]{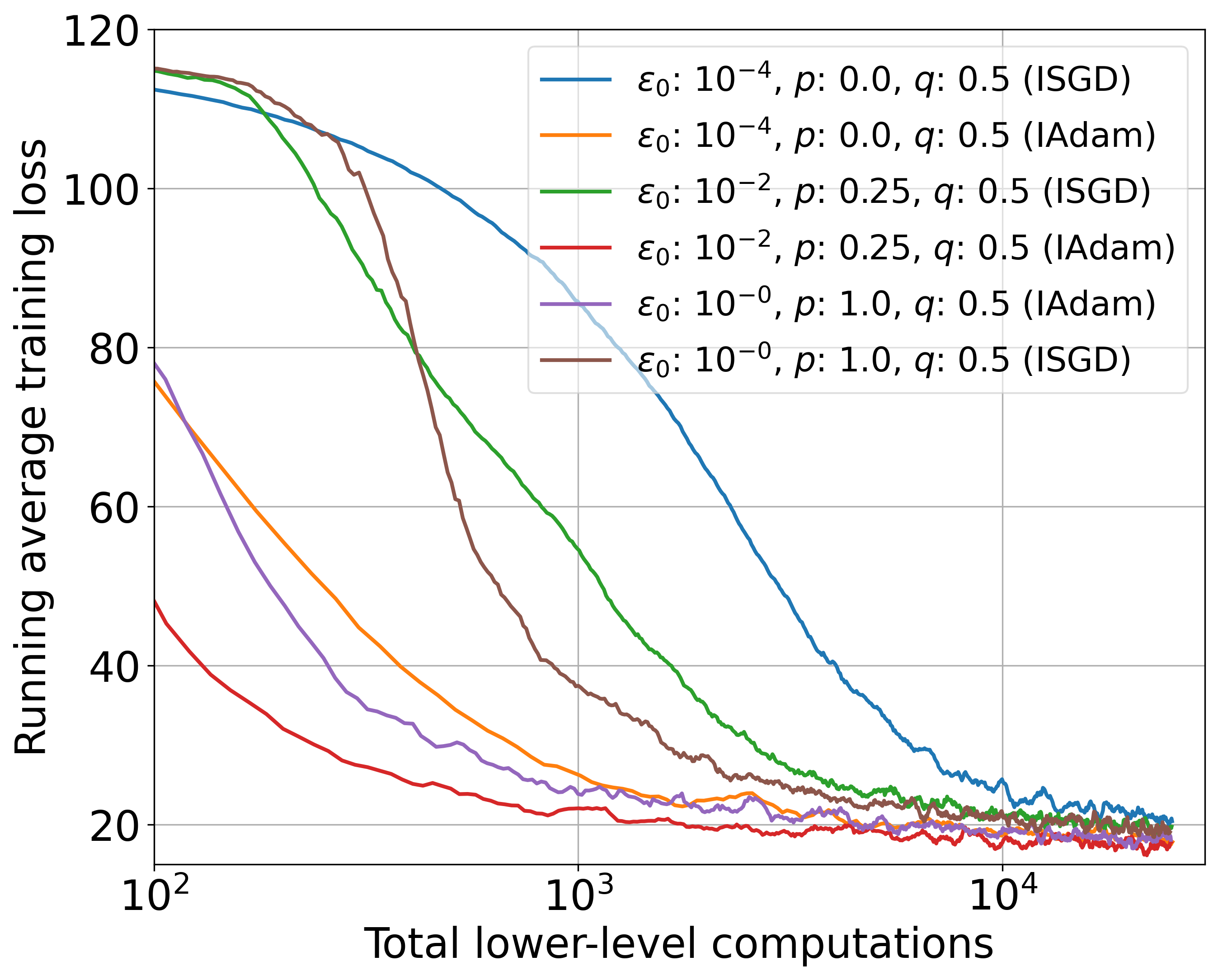}
}

\subfloat[]{%
    \label{fig:ICNN_winner}
    \includegraphics[width=0.41\columnwidth]{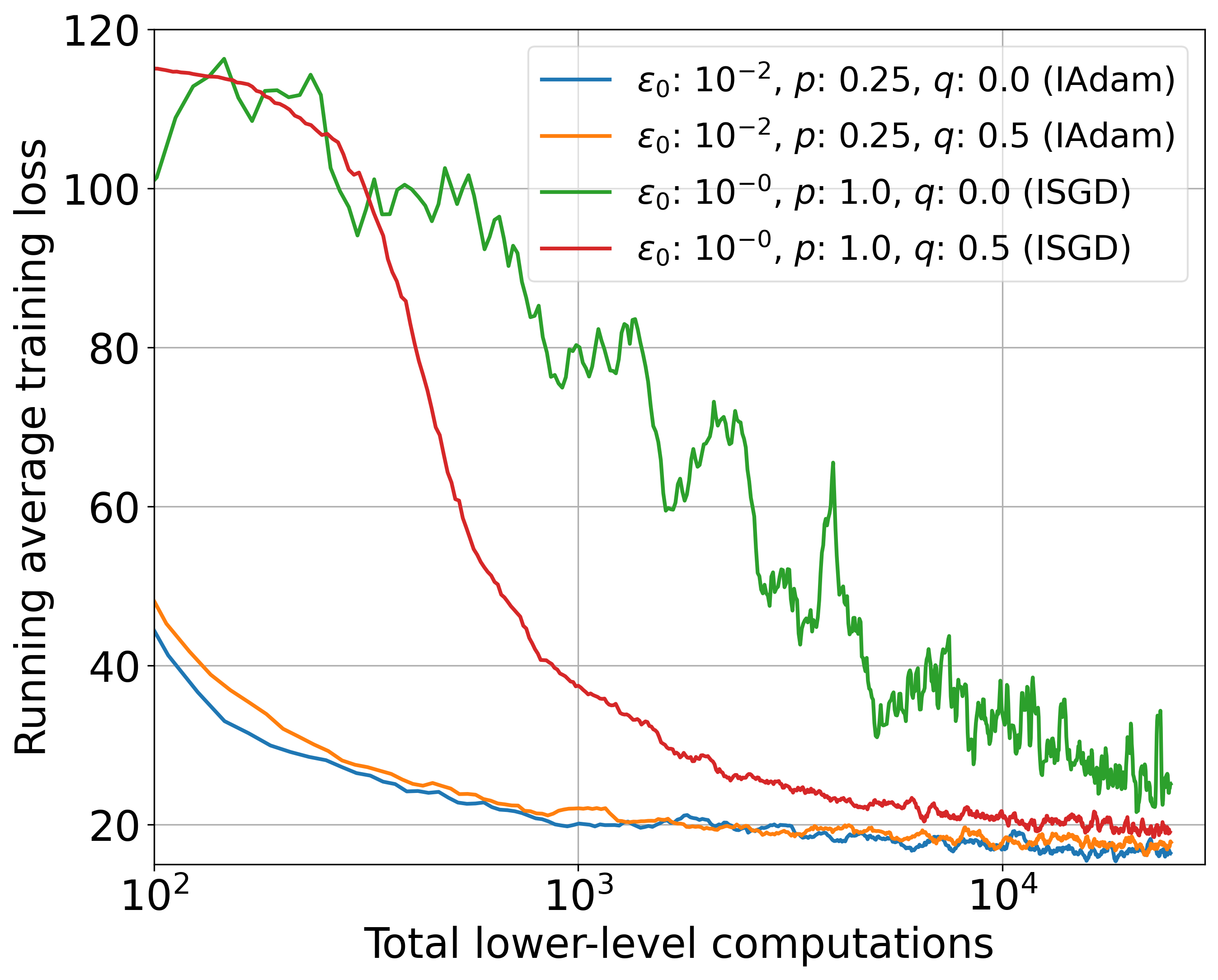}
}
\hfill
\subfloat[]{%
    \label{fig:ICNN_test}
    \includegraphics[width=0.41\columnwidth]{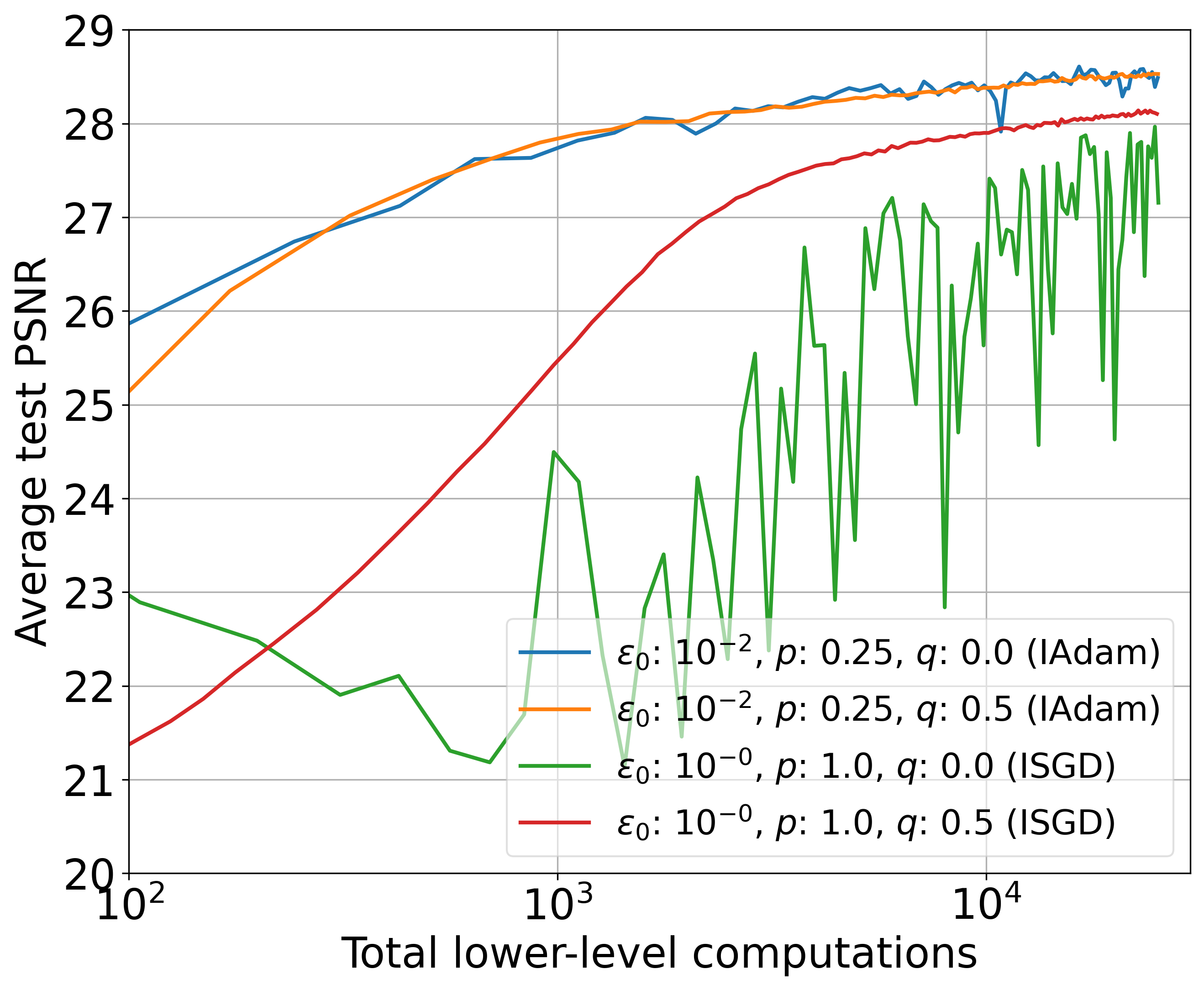}
}
\caption{Comparison between SGD and Adam with inexact stochastic hypergradients across different initial accuracies, using the most suitable values of accuracy schedule exponent $p$ and step size schedule exponent $q$ for each, found by trial and error. Top row: (a) shows configurations with fixed step size; (b) corresponds to decreasing step size. Bottom row: best-performing configurations of ISGD and IAdam in terms of training loss and test PSNR in (c) and (d), respectively.}
\label{fig:ICNN}
\end{figure}

\section{Single-loop baselines and validation of the refinements}
\label{supp:singleloop}

We verified the implementations directly: at a fixed $\theta$, the converged SOBA-family direction matches the high-accuracy hypergradient (computed with tight nmAPG and MINRES tolerances) to cosine similarity $1.0000$, and on a synthetic bilevel problem with closed-form solution the iterates converge to the analytic minimizer. Considerable tuning effort was nevertheless required, because $L_H$ grows by orders of magnitude as the regularizer's scaling parameters are learned (from $6$ to above $10^4$ within $10^3$ iterations in unstable configurations) and varies by more than $10\times$ across mini-batches.

\begin{figure}[!h]
\centering
\subfloat[]{%
    \label{fig:tightening_step}
    \includegraphics[width=0.44\columnwidth]{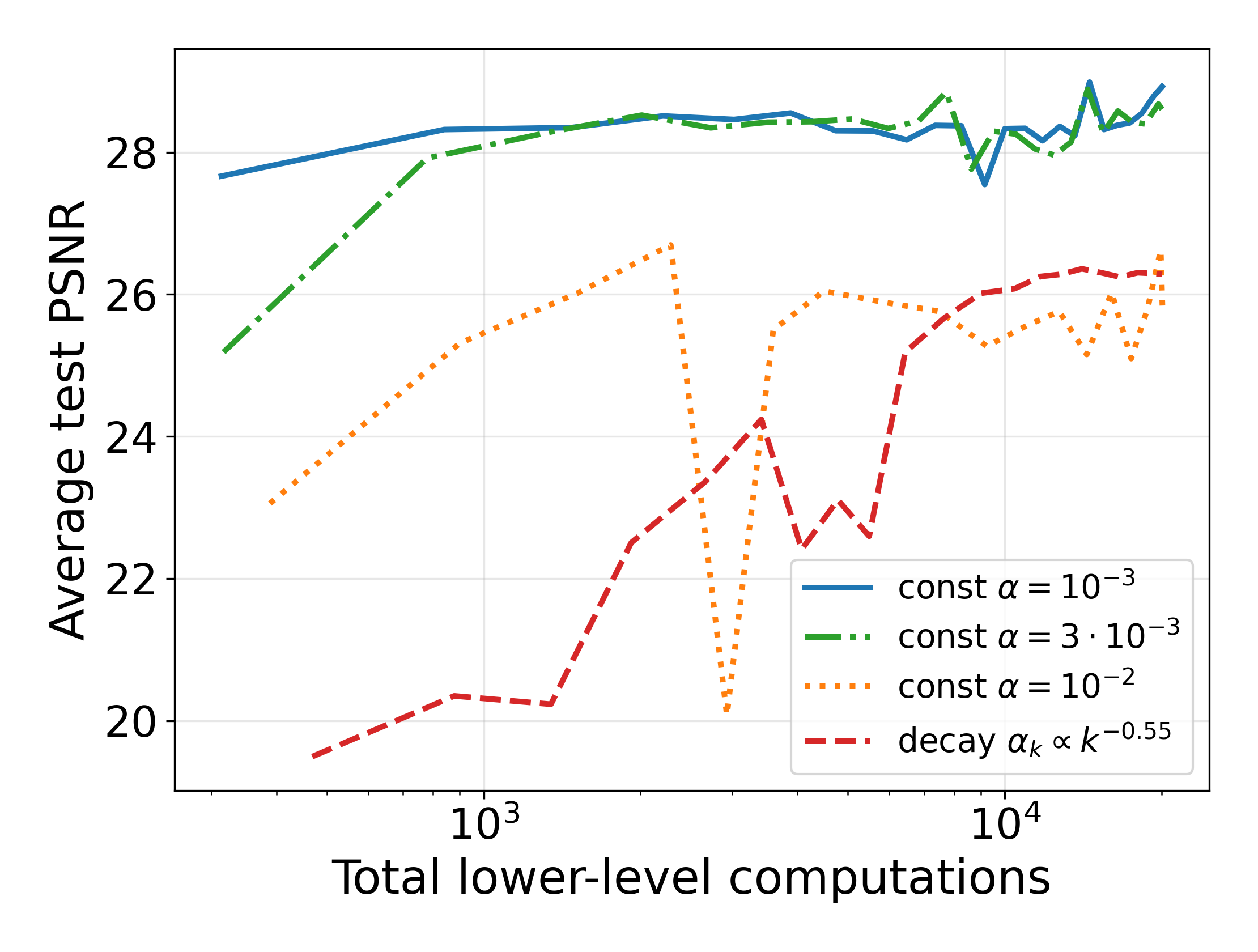}
}\hfill
\subfloat[]{%
    \label{fig:tightening_warm}
    \includegraphics[width=0.44\columnwidth]{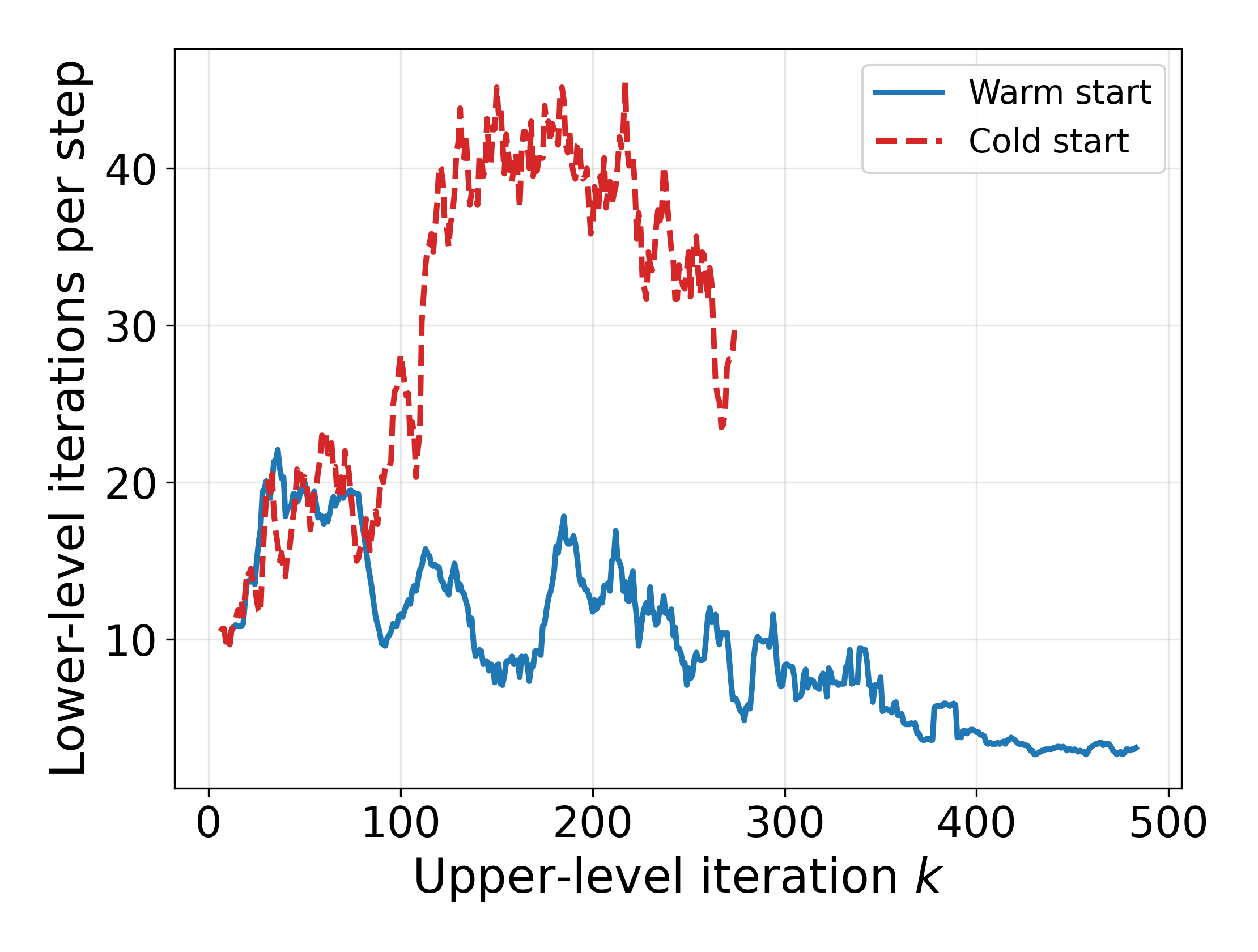}
}
\caption{{Empirical validation of the theoretical refinements, with the rate-optimal accuracy schedule $\epsilon_k = \mathcal{O}(k^{-1/4})$. (a) A horizon-tuned constant step size (Corollary~3) outperforms the anytime decaying schedule at equal cost. (b) With persistent warm starts the lower-level iterations per upper-level step decay towards the $\mathcal{O}(1)$ amortized regime of Remark~3, whereas cold starts grow as predicted by the $\log k$ term of Corollary~2.}}
\label{fig:tightening}
\end{figure}

\end{document}